\documentclass[bj]{imsart}

\RequirePackage{amsthm,amsmath,amsfonts,amssymb,latexsym,mathtools}
\RequirePackage[numbers]{natbib}
\usepackage[hypertexnames=false]{hyperref}
\usepackage[bb=boondox]{mathalfa}
\startlocaldefs
\theoremstyle{plain}
\newtheorem{thm}{Theorem}

\newtheorem{lem}[thm]{Lemma}
\newtheorem{cor}[thm]{Corollary}

\newtheorem{rem}[thm]{Remark}

\newtheorem{ass}[thm]{Assumption}
\numberwithin{thm}{section}
\numberwithin{equation}{section}
\usepackage{accents}
\newcommand{\ubar}[1]{\underaccent{\bar}{#1}}

\endlocaldefs

\newcommand{\eps}{\varepsilon}
\newcommand{\gep}{\epsilon}

\newcommand{\E}{\mathbb{E}}

\newcommand{\R}{\mathbb{R}}
\newcommand{\T}{\mathbb{T}}

\newcommand{\N}{\mathbb{N}}

\newcommand{\mX}{\mathcal{X}}

\newcommand{\om}{\mathcal{Z}}

\newcommand{\sht}{\mathcal{Q_{\alpha,\tau}}}

\newcommand{\omat}{\om_{\alpha,\tau}}

\renewcommand{\t}{\bar{\tau}}

\newcommand{\norm}[1]{\left\|#1 \right\|}

\newcommand{\rp}{\mathbb{P}}
\newcommand{\ip}[2]{\big\langle#1,#2\big\rangle}

\newcommand{\at}{\alpha,\tau}
\newcommand{\pth}{\Theta_n}
\newcommand{\sign}{{\rm sign}}

\newcommand{\enab}{\bar{\gep}_n}
\newcommand{\minimax}{{m}^\ast_n}
\newcommand{\lminimax}{l^\ast_n}
\newcommand{\enabp}{\gep_{n,\alpha,\beta,p}}
\newcommand{\cona}{\kappa}
\newcommand{\conr}{r}

\newcommand{\tth}{\tilde{\Theta}_n}

\DeclareMathOperator*{\argmax}{arg\,max}

\begin{document}

\begin{frontmatter}
\title{Adaptive inference over Besov spaces in the white noise model using \lowercase{$p$}-exponential priors}
\runtitle{Adaptive inference with \lowercase{$p$}-exponential priors}

\begin{aug}
\author[]{\fnms{Sergios} \snm{Agapiou}\ead[label=e1,mark]{agapiou.sergios@ucy.ac.cy}}
\author[]{\fnms{Aimilia} \snm{Savva}\ead[label=e2,mark]{savva.emilia@ucy.ac.cy}}
\address[]{Department of Mathematics and Statistics, University of Cyprus, \\\printead{e1,e2}}
\end{aug}

\begin{abstract}
In many scientific applications the aim is to infer a function which is smooth in some areas, but rough or even discontinuous in other areas of its domain. Such \emph{spatially inhomogeneous} functions can be modelled in Besov spaces with suitable integrability parameters. In this work we study adaptive Bayesian inference over Besov spaces, in the white noise model from the point of view of rates of contraction, using \emph{$p$-exponential} priors, which range between Laplace and Gaussian and possess regularity and scaling hyper-parameters. To achieve adaptation, we employ empirical and hierarchical Bayes approaches for tuning these hyper-parameters. Our results show that, while {it is known that} Gaussian priors can attain the minimax rate \emph{only} in Besov spaces of spatially homogeneous functions, Laplace priors {lead to adaptive or nearly adaptive procedures} in \emph{both} Besov spaces of spatially homogeneous functions \emph{and} Besov spaces permitting spatial inhomogeneities.

\end{abstract}



\end{frontmatter}



\section{Introduction}\label{sec:intro}
A common goal in statistical practice is the inference of a functional unknown. This unknown can have uniform smoothness across its domain or it can be \emph{spatially inhomogeneous}, that is, smooth in some areas and rough in other areas of its domain. For example, in imaging applications typically there are different smooth objects separated by sharp edges, in geophysical applications the physical parameter of interest has jumps at the intersection of different media and in signal processing some signals are mostly idle and only exhibit activity in localized outbursts. In this work, we will study \emph{adaptive} Bayesian inference for both of these types of unknowns, in the \emph{white noise model} and under a class of priors termed \emph{$p$-exponential}, from the point of view of rates of posterior contraction in the small noise limit.

It is well known that both spatially homogeneous and spatially inhomogeneous functions can be modelled using Besov spaces, $B^\beta_{qq'}$ with smoothness $\beta>0$ and integrability indices $q\ge 2, 1\leq q'\leq \infty$ for the former and $1\leq q<2, 1\leq q'\leq \infty$ for the latter {(for a nice heuristic explanation why $q<2$ relates to spatial inhomogeneity we refer to \cite[Section 9.6]{IJ19})}. Function estimation for such unknowns in the white noise model has been studied extensively in the frequentist statistical literature, and minimax rates under Besov-type regularity in $L_2$-loss have been established by Donoho and Johnstone in \cite{DJ98}. A main finding in \cite{DJ98} is that, unlike the homogeneous case $q\geq2$ in which linear estimators can achieve the minimax rate (Sobolev-type spaces correspond to $q=2$), in the inhomogeneous case linear estimators only achieve a rate which is polynomially slower than minimax. On the other hand, there are a few (nonlinear) frequentist procedures which are (nearly) adaptive in the minimax sense in both cases, for example based on wavelet thresholding \cite{DJ94, DJ95} or locally variable bandwidth kernel estimates \cite{LMS97}. We also mention the Bayesian wavelet thresholding method in \cite{ASS98}.  

We will restrict our attention to Besov spaces with equal integrability indices $q=q'$, $q\geq 1$, which will enable us to simplify our analysis while still exhibiting the main statistical phenomena we are interested in. Rates of posterior contraction in $L_2$-loss under such Besov regularity have been studied in the (direct) white noise model in \cite{ADH21}, for $p$-exponential priors, which is a family of sequence priors with tails between Gaussian ($p=2$) and exponential ($p=1$). The obtained upper bounds suggested that in the spatially inhomogeneous case ($1\leq q<2$), Gaussian priors are limited by the "linear-minimax" rate; this agrees with the intuition that Gaussian process regression is a linear statistical procedure due to the linearity of the posterior mean in this setting. Indeed, in the recent preprint \cite{AW21} this was made rigorous by establishing the corresponding lower bounds for rates of contraction holding uniformly over Besov bodies and for sequences of \emph{arbitrary} Gaussian priors. On the contrary, Laplace priors with appropriately chosen regularity \emph{and} scaling parameters were shown in \cite{ADH21} to attain the minimax or nearly the minimax rate (for $q=1$ and $1<q<2$, respectively). More generally, it was shown that when $p\leq q<2$, appropriately tuned $p$-exponential priors attain the minimax rate (up to logarithmic terms when $p<q$). In the less intricate spatially homogeneous case ($q\geq 2$), it was shown that the minimax rate can be achieved using appropriately tuned $p$-exponential priors with any $p\in[1,2]$. 

The purpose of this work is to theoretically establish that the required tuning of $p$-exponential priors can be performed automatically. Empirical Bayes procedures based on the maximum marginal likelihood estimator (MMLE) as well as hierarchical Bayes procedures, for choosing either the scaling or the regularity hyper-parameter of a class of Gaussian process priors in the white noise model, have been studied from the point of view of contraction rates in \cite{SVZ13} and \cite{KSVZ16}, respectively. In those sources, it was shown that both the empirical and hierarchical procedures lead to partial adaptation over hyper-rectangles when selecting the scaling hyper-parameter (extended to Sobolev spaces in \cite{RS17}), while both procedures are fully adaptive over Sobolev spaces when selecting the regularity hyper-parameter (up to logarithmic terms, improved in \cite{RS17}; see also \cite{BE08,SVZ15}). Both contributions \cite{SVZ13} and \cite{KSVZ16}, relied on the fact that for Gaussian priors the marginal likelihood can be expressed explicitly.
On the other hand, a general theory for studying adaptivity with MMLE empirical Bayes and hierarchical Bayes procedures, which does not require the explicit availability of the marginal likelihood, has been developed in \cite{RS17}; see also \cite{DRRS18} concerning general empirical Bayes procedures. We employ this theory to study adaptation with $p$-exponential priors in the white noise model, since for $p<2$ the marginal likelihood is not explicit. 

We initially study adaptation over Sobolev spaces, by choosing either the scaling or the regularity hyper-parameter of the $p$-exponential prior via the MMLE empirical and hierarchical Bayes procedures. Our results are analogous to the Gaussian case: first, we show that when choosing the scaling hyper-parameter, the minimax rate is achieved provided the truth is not too smooth compared to the prior; second, we show that when choosing the regularity hyper-parameter, the minimax rate is effectively always achieved. These results trivially generalize to Besov spaces of spatially homogeneous functions ($q\geq2$) as well.

We then study adaptation over Besov spaces permitting spatial inhomogeneities, $B^\beta_{qq}$ with $1\leq q<2$, by choosing \emph{both} the scaling and regularity parameters \emph{simultaneously}. We establish that for $p$-exponential priors with $p\leq q$, both the MMLE empirical and hierarchical Bayes posteriors are minimax adaptive (up to logs when $p<q$), for $\beta\in(\ubar{\alpha}+1/p,\bar{\alpha}+1/p)$ where $\ubar{\alpha},\bar{\alpha}>0$ can be chosen arbitrarily small and large, respectively. In particular, \emph{the main contribution of this work is the establishment that Laplace priors lead to minimax or near minimax adaptation over Besov spaces of spatially homogeneous functions as well as Besov spaces permitting spatially inhomogeneous functions.} 

{Although our results refer to Besov spaces with equal integrability indices $q=q'$, the setting is sufficiently rich to exhibit the advantage of using Laplace priors over Gaussian priors from the point of view of contraction rates. This is not surprising, since the minimax and linear minimax rates established in \cite{DJ98} are independent of $q'$. The extension of our results to the general $B^\beta_{qq'}$  case, while outside the scope of this work, is expected to be straightforward, despite requiring more technical calculations and a heavier notation.}

We mention here that adaptive Bayesian procedures for spatially inhomogeneous functions are also studied in the recent preprint \cite{RR21}, where adaptive contraction rates are obtained for the spike-and-slab prior and Bayesian CART under spatially varying H\"older smoothness assumptions. We also mention the upcoming work \cite{MG22}, which studies adaptation over Besov spaces in density estimation based on (hierarchical) Besov priors. {To the best of our knowledge, these two works and the present article are the first Bayesian adaptive results for spatially inhomogeneous functions in the literature.}

In the rest of this introductory discussion we briefly introduce the setting, and overview the general theory for studying adaptivity with MMLE empirical and hierarchical Bayes procedures from \cite{RS17}, as well as the concentration theory for $p$-exponential priors from \cite{ADH21}. In Section \ref{sec2} we state our main results, while their proofs are contained in Section \ref{sec3}. The proofs rely on a number of technical results contained in the Supplement. Finally, Section \ref{sec:conc} contains a brief concluding {discussion on computation, including two simulation experiments relating to the presented hierarchical Bayes procedures for Sobolev truths}.

\subsection{The white noise model}
We consider the white noise model, which serves as a continuous limit of Gaussian nonparametric regression, \cite{BL96}. This idealized model is typically used as a tractable stepping stone, before studying more complex nonparametric models. 

We are interested in inferring a square integrable function $\theta$ on the unit interval, from an observation of a path of the stochastic process
\begin{equation}  \label{eq:wnm} X^{(n)}_{t} = \int^{t}_{0} \theta(s) ds + \frac{1}{\sqrt{n}} W_{t}, \quad t \in [0,1].\end{equation}
Here $W$ is a standard Brownian motion, $n^{-1}$ is the noise level and we will study the small noise asymptotic regime, $n\to\infty$. {The restriction to one dimension is for reasons of notational convenience; we expect the presented theory to generalize to higher dimensional domains in a straightforward way.} Furthermore, for simplicity of our exposition, we assume that in addition to square summable, the unknown function $\theta$ is periodic; we write $\theta\in L_2(\T)$, where $\T=(0,1]$, and equip this space with the usual $L_2$-inner product $\ip{\cdot}{\cdot}_2$ and norm $\norm{\cdot}_2$. We let $P^n_{\theta}$ denote the law of the sample path $X^{(n)}:=(X^{(n)}_{t}:0\leq t\leq 1)$ in the sample space $\mX:=C[0,1]$, where the latter denotes the space of continuous real functions on the unit interval. In particular, $P^n_0$ denotes the law of the scaled Brownian motion $W_{t}/\sqrt{n}$. Then by the Cameron-Martin theorem, the likelihood function is given as
\begin{equation}\label{eq:lr}
p^n_\theta(X):=\frac{dP^n_\theta}{dP^n_0}(X)=\exp\Big(n\ip{X}{\theta}_{2}-\frac{n}2\norm{\theta}^2_{2} \Big), \quad X\in\mX.
\end{equation}
We refer to Section 6.1.1 in \cite{GN16} for more details on the white noise model.

\subsection{The MMLE and adaptive contraction rates}\label{ssec:mmle}
Given a class of prior distributions $\Pi(\cdot|\lambda)$ on the parameter space $\Theta:=L_2(\T)$ indexed by a hyper-parameter $\lambda\in\Lambda$, the associated posterior distributions are given as
\[\Pi(B|X^{(n)},\lambda)=\frac{\int_Bp^n_\theta(X^{(n)})d\Pi(\theta|\lambda)}{\bar{m}(X^{(n)}|\lambda)},\] for any Borel set $B\subseteq\Theta$. Here 
\[\bar{m}(X^{(n)}|\lambda):=\int_\Theta p^n_\theta(X^{(n)})d\Pi(\theta|\lambda),\] denotes the \emph{marginal likelihood} of the path $X^{(n)}$ in \eqref{eq:wnm}, given $\lambda$. Fixing a set of candidate hyper-parameter values $\Lambda_n\subseteq \Lambda$, the MMLE is defined as
\[\hat{\lambda}_n\in \argmax_{\lambda\in \Lambda_n}\,\,\bar{m}(X^{(n)}|\lambda),\]
and leads to the associated empirical Bayes posterior distribution $\Pi(\cdot|X^{(n)},\hat{\lambda}_n)$.

The asymptotic behaviour of the MMLE $\hat{\lambda}_n$ was studied in general settings in \cite{RS17} and sufficient conditions were given securing that, as $n\to\infty$, $\hat{\lambda}_n$ belongs to a set   $\Lambda_0$ of 'good' values of the hyper-parameter with probability tending to 1. More specifically for $\Theta=L_2(\T)$ as assumed here, let $\theta_0\in \Theta$ denote the ground truth and let $\gep_n(\lambda)=\gep_n(\lambda;K,\theta_0)$ be such that 
\begin{equation}\label{eq:en}\Pi(\theta\in\Theta : \norm{\theta-\theta_0}_2\leq K\gep_n(\lambda)|\lambda)=e^{-n\gep_n^2(\lambda)},\end{equation}
where $K>0$ is a constant introduced for additional flexibility. Note that, for 'continuous' priors as the ones considered in this work, $\gep_n(\lambda)$ is defined uniquely by \eqref{eq:en} (see Lemma \ref{lem:existenceenl} in the Supplement below). Let $\gep_{n,0}>0$ be defined through
\begin{equation}\label{defn:en0}
\gep_{n,0}^2 = \inf_{\lambda \in \Lambda_n} \Big\{ \gep_{n}^{2}(\lambda):  \gep_{n}^{2}(\lambda) \geq m_n (\log n) / n \Big\},
\end{equation}
with $m_n$ tending to infinity arbitrarily slowly. Then the set $\Lambda_0$ is given as
\begin{equation}\label{eq:L0Mnint}
\Lambda_{0} (M_{n}) := \{ \lambda \in \Lambda_n : \gep_{n}(\lambda) \leq M_{n} \gep_{n,0}\},
\end{equation}
with $M_{n}$ tending to infinity arbitrarily slowly. 

Having identified $\Lambda_0$, additional sufficient conditions were derived securing that the contraction rate of the MMLE empirical Bayes posterior is bounded by $M_n\gep_{n,0}$. This was based on the techniques in \cite{DRRS18}, where rates of contraction for empirical Bayes posteriors were studied with general estimators of the hyper-parameter $\lambda$ and where the existence of a 'good' set $\Lambda_0$ in which these estimators concentrate asymptotically was postulated. Further sufficient conditions were derived for establishing the sharpness of the rates for each $\theta_0$, as well as conditions for showing that the contraction rates for the MMLE empirical and hierarchical posteriors share the same bounds. We recall these conditions adapted to the assumed white noise model setting in Subsection \ref{sec:conditions} below.


\subsection{$p$-exponential priors and their concentration}\label{ssec:pexp}
The class of \emph{$p$-exponential priors} has been introduced in \cite{ADH21} as a generalization of Gaussian and Besov priors. $B^s_{11}$-Besov priors, which are sequence priors corresponding to wavelet expansions with independent and appropriately weighted Laplace-distributed coefficients, have been proposed in the applied Bayesian inverse problems literature in order to achieve 'edge-preserving' reconstruction; see \cite{LSS09} as well as \cite{LS04, DSS12, KLNS12, DS17}. The weights of the wavelet coefficients are chosen in order to penalize the $B^s_{11}$-norm in the prior probability and, indeed, in \cite{ABDH18} it was established in a general nonparametric inverse regression setting that the \emph{maximum a posteriori estimator} (MAP) arising from these priors is the minimizer of a penalized least squares functional with a $B^s_{11}$-penalty. 

In this work we will study \emph{$\alpha$-regular and $\tau$-scaled $p$-exponential priors}, defined as
\[\Pi(\cdot|\lambda):=\mathcal{L}\big((\gamma_\ell\xi_\ell)_{\ell\in\N}\big), \;\lambda=(\at),\]
where $\xi_\ell$ are independent and identically distributed real random variables with probability density function $f_p(x)\propto \exp(-|x|^p/p), \,x\in\R, \,p\in[1,2]$ and $\gamma_\ell$ are deterministic scalings of the form $\gamma_\ell=\tau\ell^{-1/2-\alpha}$ for some $\alpha, \tau>0$. Note that for $p=1$, these priors are merely a different parametrization of $B^s_{11}$-Besov priors, while for $p=2$ they correspond to the Gaussian priors studied in \cite{SVZ13, KSVZ16}. The concentration properties of $p$-exponential priors in general and $\alpha$-regular $\tau$-scaled $p$-exponential priors in particular, have been studied in \cite{ADH21}. Here we recall two key results for our analysis.

First, note that the term "$\alpha$-regular" is used since, for any $q\ge1$, draws from $\Pi(\cdot|\lambda)$ live in $B^t_{qq}$ with probability one for $t<\alpha$ and with probability zero for $t\ge\alpha$, see \cite[Lemma 5.2]{ADH21}. Associated to an $\alpha$-regular $\tau$-scaled $p$-exponential prior $\Pi(\cdot|\lambda)$ are two sequence spaces. The first is the space of translations of $\Pi$ which produce equivalent measures, which is a separable Hilbert space: 
\[\sht =  \Big\{h \in \R^{\infty}: \tau^{-2} \sum_{\ell=1}^{\infty} h_{\ell}^{2} \ell^{1+2\alpha} < \infty \Big\}, \quad\norm{h}_{\sht}  = \tau^{-1}  \left(\sum_{\ell=1}^{\infty} h_{\ell}^{2} \ell^{1+2\alpha} \right)^{1/2}.\]
The second is the following Banach space 
\[\omat =  \Big\{ h \in \R^{\infty}: \tau^{-p} \sum_{\ell=1}^{\infty} |h_{\ell}|^{p} {\ell}^{p/2 + \alpha p} < \infty \Big\}, \quad\norm{h}_{\omat} = \tau^{-1} \left(\sum_{\ell=1}^{\infty} |h_{\ell}|^{p} {\ell}^{p/2 + \alpha p}\right)^{1/p}.\] 
Note that $\omat \subsetneq \sht \subsetneq \ell_2$ if $p \in [1,2)$, while $\omat = \sht$ if $p = 2$, in which case both spaces are identified with the reproducing kernel Hilbert space (RKHS) of the corresponding Gaussian prior. 

For any $\alpha$-regular $\tau$-scaled $p$-exponential prior, it is straightforward to show that for each $n\in\N$, there exists a unique $\gep_n(\lambda)$ such that \eqref{eq:en} holds, see Lemma \ref{lem:existenceenl} in the Supplement below. Given $\at>0$, we define the \emph{concentration function} at $\theta \in \ell_2$
\begin{align}\label{defn:concfcn}
\varphi_{\theta}(\gep)&:=\inf_{h\in \omat: \norm{h-\theta}_{2}\leq\gep}\norm{h}^p_{\omat}-\log\Pi\left(\gep B_{\ell_2}\mid\alpha,\tau\right),
\end{align}
where $B_{\ell_{2}}$ is the centered unit ball of $\ell_2$. For $\theta=0$, the concentration function measures the probability of centered $\gep$-balls, while for $\theta \neq 0$ it gives a lower bound on the probability of $\gep$-balls around $\theta$
\begin{equation}\label{eq:lowerconc}\Pi(\theta+\gep B_{\ell_2}|\at)\geq \exp\big(-\varphi_\theta(\gep/2)\big),\end{equation} see \cite[Theorem 2.13]{ADH21}. This enables us to obtain upper bounds on the value of $\gep_n(\lambda)$ solving \eqref{eq:en}. However, except in the Gaussian case $p=2$ \cite[Lemma 5.3]{VZ08rkhs}, the concentration function is not known to give an upper bound on the probability of $\gep$-balls around $\theta$. Since $p$-exponential priors are log-concave, one can at least use Anderson's inequality which upper bounds the probability of any non-centered ball by the probability of the corresponding centered one, leading to a lower bound on $\gep_n(\lambda)$, see Lemma \ref{lem:2.2b}. It will transpire that this naive bound is sufficient for our analysis.

Furthermore, $\alpha$-regular $\tau$-scaled $p$-exponential priors satisfy the following concentration inequality derived from the work of Talagrand in \cite{TA94},
\begin{equation}\label{eq:tal}\Pi(\gep B_{\ell_2}+R^{p/2}B_{\sht}+RB_{\omat}|\at)\geq1-\frac{e^{-R^p/\tilde{K}}}{\Pi(\gep B_{\ell_2}|\at)}, \quad \gep>0, R>0,\end{equation}
where $B_\mathcal{Y}$ denotes the centered unit ball in the space $(\mathcal{Y},\norm{\cdot}_{\mathcal{Y}})$ and $\tilde{K}$ is a positive constant depending only on $p$, see \cite[Proposition 2.15]{ADH21}. Inequality \eqref{eq:tal} shows that, even though it is straightforward to check that both spaces $\sht$ and $\omat$ have probability zero under the prior, the prior mass concentrates in small $\gep$-enlargements in $\ell_2$ around sums of large balls of these two spaces.

Together the lower bound \eqref{eq:lowerconc} and the concentration inequality \eqref{eq:tal}, have been employed to study contraction rates by verifying the 'prior-mass around the truth' and 'sieve-set' conditions in non-adaptive general posterior contraction results, such as the ones in the seminal papers \cite{GGV00, GV07} and the textbooks \cite{GV17, GN16}. This was done in \cite{ADH21} in direct settings, while \cite{AW21} and \cite{GR20} supplemented these results with other novel concentration inequalities in order to study non-linear PDE inverse problems and drift estimation for multi-dimensional diffusions, respectively. We will also build on \eqref{eq:lowerconc} and \eqref{eq:tal} in order to verify the more elaborate assumptions of \cite{RS17} and study adaptivity.


\subsection{Periodic Besov spaces}
We consider an orthonormal wavelet basis $\mathcal{E}=\{e_{kl}\}_{k\in\N, 1\leq l \leq 2^k}$ in $L_2(\T)$. We can then uniquely express a function $\theta\in L_2(\T)$ as \[\theta=\sum_{k=1}^\infty\sum_{l=1}^{2^k}\theta_{kl}e_{kl},\] where $\theta_{kl}=\ip{\theta}{e_{kl}}_2$. This allows us to identify a function $\theta\in L_2(\T)$ with the corresponding sequence of coefficients $\{\theta_{kl}\}\in\ell_2$, and in particular to define $p$-exponential priors on sequences, even though the goal is to reconstruct functions. We additionally assume that the basis functions $e_{kl}$ have sufficient H\"older regularity, which allows us to characterize Besov function spaces and in particular the corresponding Besov norms, via the coefficients $\theta_{kl}$ 
\[\norm{\theta}_{B^s_{qq'}}=\left(\sum_{k=1}^\infty 2^{q'k(s+\frac12-\frac1q)}\Big(\sum_{l=1}^{2^k}|\theta_{kl}|^q\Big)^{q'/q}\right)^{1/q'}, \quad s\in\R, \;1\leq q,q'<\infty.\] Note that by working on the torus $\T$, we benefit from a particularly clean sequence representation of the Besov norm, however analogous representations exist which allow treating more general domains with additional technical effort (see for example \cite{AW21} in the non-adaptive inverse problems setting). For details on periodic Besov spaces see \cite[Section 4.3.4]{GN16}.

Since we restrict to the case $q=q'$, enumerating $\{\theta_{kl}\},~{k\in\N, 1\leq l\leq 2^k}$, using a single index as $\{\theta_\ell\}_{\ell\in\N}$, and using the asymptotic equivalence between the sequences $2^{sk}, \,k\in \N, \,1\leq l \leq 2^k$ and $\ell^{s}, \ell\in\N$, for any $s\in\R$, the above characterization of the Besov norms can be simplified to 
\[\norm{\theta}_{B^s_{qq}}=\left(\sum_{\ell=1}^\infty \ell^{qs+\frac{q}2-1}|\theta_\ell|^q\right)^{1/q}, \quad s\in\R, \;1\leq q<\infty.\]
This justifies our definition of the $p$-exponential priors in Subsection \ref{ssec:pexp} via a single index. Note that for $q=2$ we obtain the usual Sobolev-Hilbert spaces which we denote as $H^{s}$. The latter can also be defined using the Fourier basis.

\subsection{Additional notation}
We denote by $E^{n}_{\theta}$ the expectation with respect to $P^n_\theta$, while we denote the log-likelihood of a $\theta$ as $\ell_n(\theta) := \log p_{n}^{\theta} (X)$. The notation $N(\gep, A, d)$ is used for the $\gep$-covering number of a set $A$ with respect to a metric $d$, that is the minimum number of balls of radius $\gep$ with respect to $d$ which are needed to cover the set $A$. For two positive sequences $(a_n), (b_n)$, the notation $a_n\asymp b_n$ means that $a_n/b_n$ is bounded away from zero and infinity, while $a_n\lesssim b_n$ means that $a_n/b_n$ is bounded. For $L>0$ and specifically for Sobolev and Besov spaces, we will use the notation $H^\beta(L)$ and $B^s_{qq}(L)$, respectively, to denote the corresponding ball of radius $L$.


\section{Setup and main results}\label{sec2}
In this section we formulate our main results, which concern adaptive rates of contraction for empirical and hierarchical Bayes procedures based on $\alpha$-regular and $\tau$-scaled $p$-exponential priors in the white noise model. 

{In the case of a truth with Sobolev regularity $\beta$, the non-adaptive upper bounds in \cite[Theorem 5.5]{ADH21} show that the minimax rate can be achieved without rescaling if we choose $\alpha=\beta$, suggesting that a possible strategy for adaptation is to fix $\tau=1$ and choose the regularity, $\lambda=\alpha$. We will see that (similarly to Gaussian priors) another strategy is to fix the regularity and choose the scaling, $\lambda=\tau$, provided $\alpha$ is not too small. We study the Sobolev setting in Subsection \ref{ssec:Sobolev}.}

{For a truth in a Besov space $B^\beta_{qq}$ with $1\leq q<2$, the non-adaptive upper bounds in \cite[Proposition 5.8]{ADH21} only match or nearly match the minimax rate if $p\leq q$, for $\alpha=\beta-1/p$ \emph{and} for an appropriately vanishing sequence of scalings $\tau$ as $n\to\infty$. As a result, in Subsection \ref{ssec:Besov} we study adaptation in the spatially inhomogeneous Besov setting only for $p\leq q$ and by \emph{simultaneously} choosing \emph{both} the regularity \emph{and} the scaling, $\lambda=(\at)$. Note that, since it is not known whether the bounds in \cite[Proposition 5.8]{ADH21} are sharp, it is also not known whether $p\leq q$ and varying $\alpha$ and $\tau$ simultaneously are necessary.}

\subsection{Minimax and linear minimax rates under Besov regularity}
We first recall the minimax rates in $L_2$-loss for the white noise model and under Besov regularity, established in \cite{DJ98}. For a Besov class $B^\beta_{qq}, q\geq1$, with $\beta>1/q$ or $\beta\geq 1$ when $q=1$, the minimax rate depends only on the smoothness $\beta$ 
\begin{equation}\label{eq:minimax}
\minimax:=n^{-\frac{\beta}{1+2\beta}}.
\end{equation}
When restricting to linear estimators, the minimax rate depends on $q$ as well, 
\begin{equation}\label{eq:lminimax}
\lminimax:=n^{-\frac{\beta-\gamma/2}{1+2\beta-\gamma}}, \quad\text{where}\quad \gamma=\frac{2}{q}-\frac{2}{\max(q,2)}.
\end{equation}
In particular, for $q<2$ the linear minimax rate is polynomially slower than the global minimax rate, while for $q\geq2$ the linear and global minimax rates coincide.

\subsection{Prior specification}\label{ssec:prior}
We consider $\alpha$-regular and $\tau$-scaled $p$-exponential priors, $\Pi(\cdot|\lambda)$, as defined in Subsection \ref{ssec:pexp}. Depending on the setting, we will consider either $\lambda=\tau$, $\lambda=\alpha$ or $\lambda=(\at)$ as the hyper-parameter to be chosen, which accordingly lives in either $\Lambda=(0,\infty)$ or $\Lambda=(0,\infty)^2$. We will study two approaches for automatically tuning $\lambda$:
\begin{enumerate}
\item[i)] Empirical Bayes: choose $\lambda$ as the maximum marginal likelihood estimator $\hat{\lambda}_n$ defined in Subsection \ref{ssec:mmle} for an appropriate set of candidate hyper-parameters values $\Lambda_n$, leading to the empirical Bayes posterior $\Pi(\cdot| X^{(n)}, \hat{\lambda}_n);$
\item[ii)] Hierarchical Bayes: postulate a hyper-prior $\tilde{\pi}$ on $\lambda\in \Lambda$, giving rise to the hierarchical prior
\[\Pi(\cdot)=\int_{\Lambda}\tilde{\pi}(\lambda)\Pi(\cdot|\lambda)d\lambda.\]
\end{enumerate}

In the hierarchical Bayes setting, we make the following assumptions for the hyper-priors on $\tau$ and/or $\alpha$, respectively:
\begin{ass}[Hyper-prior distributions]\label{ass:hyper_prior}
\;{$ $}
\begin{enumerate}
\item[i)] Given $\alpha>0$, $\tilde{\pi}(\tau)$ is supported on $\big[n^{-1/(2+p+2\alpha p)},\infty\big)$ and for some ${\conr_0}>0$, there exist ${\conr_1}, {\conr_3}>0$ and ${\conr_2}>1+1/{\conr_0}$ such that
\begin{align} \label{thm1_1}
\begin{split}
e^{-{\conr_1} \tau^{\frac{2}{1+2\alpha}}} \lesssim \tilde{\pi}(\tau) \lesssim \tau^{-{\conr_2}}, \quad \tau \geq 1, \\
 \tilde{\pi}(\tau)\gtrsim e^{-{\conr_3} \tau^{-p}}, \quad n^{-\frac1{2+p+2\alpha p}}\leq \tau \leq 1.
\end{split}
\end{align} 
\item[ii)] $\tilde{\pi}(\alpha)$ is supported on $[\ubar{\alpha},\bar{\alpha}]$ for some fixed $\bar{\alpha}>\ubar{\alpha}>0$, {and there exists  $ \cona>0$ such that}
\begin{equation} \label{thm1_2}
 {\tilde{\pi}(\alpha)\geq\cona, \quad \text{for all}\quad \alpha\in[\ubar{\alpha},\bar{\alpha}].}
\end{equation}
\item[iii)]$\tilde{\pi}(\at)=\tilde{\pi}_\alpha(\alpha)\tilde{\pi}_\tau(\tau|\alpha)$, where $\tilde{\pi}_\alpha(\alpha)$ satisfies part (ii), while $\tilde{\pi}_\tau(\tau|\alpha)$ satisfies  part (i)  with constants $\conr_0,\dots,\conr_3$ and constants in the upper and lower bounds in \eqref{thm1_1} which are uniform in $\alpha$ (note that the lower bound for $\tau\ge1$ is redundant here).
\end{enumerate}
\end{ass}

For example, \eqref{thm1_1} is satisfied  by appropriately left-truncated inverse gamma distributions with constants $\conr_0,\dots,\conr_3$ independent of $\alpha$, \eqref{thm1_2} by truncated exponential distributions, while combinations of such distributions lead to hyper-priors $\tilde{\pi}(\at)$ satisfying Assumption \ref{ass:hyper_prior}(iii).

\subsection{Adaptive rates of contraction over Sobolev spaces}\label{ssec:Sobolev}
We first study adaptation over Sobolev smoothness, using $\alpha$-regular $\tau$-scaled $p$-exponential priors, with a fixed regularity $\alpha$ and by choosing the scaling parameter $\tau$.
  
\begin{thm}[$\alpha$-fixed, $\lambda=\tau$]\label{thm:Sobolev-tau}
 Consider $\alpha$-regular $\tau$-scaled $p$-exponential priors $\Pi(\cdot \mid \at)$ in $\Theta$, where $\alpha>0$ is fixed and $\lambda=\tau>0$. 
 Assume that $\theta_0\in H^\beta{(L)}$ for some $\beta\geq(1+\alpha p)/(p+2\alpha p)$ and $L>0$. 
\begin{enumerate}
\item[i)]Let $\Lambda_n\coloneqq\big[n^{-1/(2+p+2\alpha p)},n^{\alpha}\big]$  be the set of candidate hyper-parameter values for the MMLE $\hat{\lambda}_n$. Then, for $M_n$ tending to infinity arbitrarily slowly, it holds
\[P^n_{\theta_0}\left(\hat{\lambda}_n\in \Lambda_0(M_n)\right)\to1\] 
and the empirical Bayes posterior satisfies

\begin{equation}
\sup_{\theta_0\in H^\beta(L)}E_{\theta_0}^{n} \Pi(\theta:\norm{\theta-\theta_0}_2 \geq M_{n} \enabp | X^{(n)}, \hat{\lambda}_n)=o(1),
\end{equation}
where
\begin{equation} 
\enabp \asymp
\begin{cases}    
     n^{-\frac{\beta}{1+2\beta}}, &\text{for} \quad\text{$\beta < \alpha + 1/p$,}\\
    n^{-\frac{1+\alpha p}{2+p(1+2\alpha)}},  &\text{for} \quad \text{$\beta > \alpha + 1/p$,}\\
    n^{-\frac{\beta}{1+2\beta}}\big(\log{n}\big)^{\frac{2-p}{2p(1+2\beta)}}, &\text{for} \quad  \text {$\beta = \alpha + 1/p$.}
 \end{cases}\end{equation} 
\item[ii)] Similarly, for a hyper-prior on $\lambda=\tau$ which satisfies Assumption \ref{ass:hyper_prior}(i) and for $M_n$ tending to infinity arbitrarily slowly, the hierarchical posterior satisfies 
\begin{equation}
\sup_{\theta_0\in H^\beta(L)}E_{\theta_0}^{n} \Pi(\theta:\norm{\theta-\theta_0}_2 \geq M_{n} \enabp | X^{(n)})=o(1).
\end{equation}
\end{enumerate}  
\end{thm}

Note that, since $p\le 2$, the assumption $\beta\ge (1+\alpha p)/(p+2\alpha p)$ is satisfied for example when $\beta\ge 1/p$. Assuming $\beta\ge (1+\alpha p)/(p+2\alpha p)$, the  minimax rate $\minimax$ is attained by both the empirical and hierarchical Bayes approaches of choosing the scaling $\tau$, provided that $\beta<\alpha+1/p$, that is provided the truth is not too smooth compared to the prior on $\theta$. Interestingly, the smaller $p$ is, the more undersmoothing the prior can be while still achieving the minimax rate by rescaling. {This is perhaps counterintuitive, since on the one hand Gaussian priors were expected to be the most compatible with the $\ell_2$ structure of Sobolev spaces among all $p$-exponential priors, and on the other hand the smaller $p$ is the heavier-tailed the prior which should be less suitable for smoother functions. However, it turns out that the smaller $p$ is, the larger the capacity of a $p$-exponential prior to exploit vanishing scalings $\tau$ to achieve the minimax rate for functions with smoothness higher than the prior regularity. This can be understood by examining the concentration function \eqref{defn:concfcn}: the penalty payed in the infimum term for letting $\tau\to0$ is smaller the smaller $p$ is, while the gain in the centered small ball probability term is identical for all $p$ (see \eqref{lem2.2_3} in Lemma \ref{lem:2.2b} below).}  Note that for $p=2$ our results are compatible with the results of \cite{SVZ13, RS17} (which are in slightly different settings: the former over hyper-rectangles, the latter using truncated rescaled $\alpha$-regular Gaussian priors).

We next study adaptation over Sobolev smoothness, by $\alpha$-regular $\tau$-scaled $p$-exponential priors, with fixed scaling $\tau=1$ and choosing the regularity parameter $\alpha$.  
\begin{thm}[$\tau$-fixed, $\lambda=\alpha$]\label{thm:Sobolev-alpha}
Consider $\alpha$-regular $\tau$-scaled $p$-exponential priors $\Pi(\cdot \mid \at)$ in $\Theta$, where $\tau=1$, $\lambda=\alpha$ and $\alpha\in [\ubar{\alpha},\bar{\alpha}]$ for some fixed $\bar{\alpha}>\ubar{\alpha}>0$.  Assume that $\theta_0\in H^\beta{(L)}$ for some $\beta\in(\ubar{\alpha},\bar{\alpha})$ and $L>0$. 
\begin{enumerate}  
\item[i)] Let $\Lambda_n\coloneqq[\ubar{\alpha},\bar{\alpha}]$ be the set of candidate hyper-parameter values for the MMLE $\hat{\lambda}$ (constant for all $n\in\N$).  
Then, for any $M_n$ tending to infinity arbitrarily slowly, it holds
\[P^n_{\theta_0}\Big(\hat{\lambda}_n\in \Lambda_0(M_n)\Big)\to1\]
and the empirical Bayes posterior satisfies\begin{equation}\label{eq:thm2328}
\sup_{\theta_0\in H^\beta(L)}E_{\theta_0}^{n} \Pi(\theta:\norm{\theta-\theta_0}_2 \geq M_{n} \minimax | X^{(n)}, \hat{\lambda}_n)=o(1),
\end{equation}
where $\minimax$ is the minimax rate from \eqref{eq:minimax}.
\item[ii)] Similarly, for a hyper-prior on $\lambda=\alpha$ which satisfies Assumption \ref{ass:hyper_prior}(ii), for $M_n$ tending to infinity arbitrarily slowly, the hierarchical posterior satisfies 
\begin{equation}
\sup_{\theta_0\in H^\beta(L)}E_{\theta_0}^{n} \Pi(\theta:\norm{\theta-\theta_0}_2 \geq M_{n} \minimax | X^{(n)})=o(1).
\end{equation}
\end{enumerate} 
\end{thm}

For any $p\in[1,2]$, the minimax rate $\minimax$ is attained by both the empirical and hierarchical Bayes approaches of choosing the regularity $\alpha$, provided the truth has Sobolev regularity in the arbitrarily large interval $(\ubar{\alpha},\bar{\alpha})$, $\bar{\alpha}>\ubar{\alpha}>0$. For $p=2$, this is compatible with the result for truncated Gaussian priors in \cite{KSVZ16} and \cite[Proposition 3.2]{RS17}.

\begin{rem}
Note that Theorems \ref{thm:Sobolev-tau} and \ref{thm:Sobolev-alpha} remain true for $\theta_0\in B^\beta_{qq}$ with $q\geq2$, that is over Besov spaces which do not permit spatial inhomogeneity in the truth.
This is a trivial implication of \cite[Proposition 5.4]{ADH21} which shows that our upper bounds on the rates $\gep_n(\lambda)$ solving \eqref{eq:en} are identical  for all $q\ge2$.
\end{rem}

\subsection{Adaptive rates of contraction over Besov spaces permitting spatial inhomogeneity}\label{ssec:Besov}
We finally study adaptation over Besov spaces $B^\beta_{qq}$, with $1\leq q<2$, using $\alpha$-regular $\tau$-scaled $p$-exponential priors with $p\le q$ and choosing both the regularity $\alpha$ and the scaling $\tau$, simultaneously. 

\begin{thm}\label{thm:Besov}
Consider $\alpha$-regular $\tau$-scaled $p$-exponential priors $\Pi(\cdot \mid \at)$ in $\Theta$, where $\lambda=(\at)$ for $\tau>0$ and $\alpha\in [\ubar{\alpha},\bar{\alpha}]$ with $\bar{\alpha}>\ubar{\alpha}>0$ fixed. Assume that $\theta_0\in B^\beta_{qq}(L)$ for $1\leq p\leq q<2$, $\beta\in(\ubar{\alpha}+1/p,\bar{\alpha}+1/p)$ and $L>0$. 

\begin{enumerate}
\item[i)]Let 
$$\Lambda_{n} := \Big\{(\at): \alpha\in[\ubar{\alpha}, \bar{\alpha} ], \tau\in\big[n^{-\frac1{2+p+2\alpha p}}, n^\alpha \big]\Big\}, $$
 be the set of candidate hyper-parameter values for the MMLE $\hat{\lambda}_n$. Then, for $M_n$ tending to infinity arbitrarily slowly, it holds
\[P^n_{\theta_0}\Big(\hat{\lambda}_n\in \Lambda_0(M_n)\Big)\to1\]
and the empirical Bayes posterior satisfies
\begin{equation}
\sup_{\theta_0\in B^\beta_{qq}(L)}E_{\theta_0}^{n} \Pi(\theta:\norm{\theta-\theta_0}_2 \geq M_{n} \minimax{\log(n)^{\frac{q-p}{pq(1+2\beta)}}} | X^{(n)}, \hat{\lambda}_n)=o(1),
\end{equation}
where $\minimax$ is the minimax rate from \eqref{eq:minimax}.
\item[ii)] Similarly, for a hyper-prior on $\lambda=(\at)$ which satisfies Assumption \ref{ass:hyper_prior}(iii), for $M_n$ tending to infinity arbitrarily slowly, the hierarchical posterior satisfies 
\begin{equation}
\sup_{\theta_0\in B^\beta_{qq}(L)}E_{\theta_0}^{n} \Pi(\theta:\norm{\theta-\theta_0}_2 \geq M_{n} \minimax\log(n)^{\frac{q-p}{pq(1+2\beta)}} | X^{(n)})=o(1).
\end{equation}
\end{enumerate} 
\end{thm}

The minimax rate is attained when $p=q$, while for $p<q$ the minimax rate is attained up to logarithmic terms. {Note, that since we consider $\alpha\in[\ubar{\alpha}, \bar{\alpha}]$ and since in the non-adaptive results in \cite[Proposition 5.8]{ADH21} the minimax rate is achieved for $\alpha=\beta-1/p$ (for an appropriate choice of $\tau$), it is not surprising that the adaptation range is $\beta\in(\ubar{\alpha}+1/p, \bar{\alpha}+1/p)$.} In particular, \emph{Laplace priors} attain the minimax or nearly the minimax rate over all Besov spaces $B^\beta_{qq}$ with $1\le q<2$ and regularity $\beta$ in the arbitrarily large interval $(\ubar{\alpha}+1, \bar{\alpha}+1)$. Importantly, by choosing $\ubar{\alpha}$ sufficiently small, we can guarantee that Laplace priors attain the minimax rate over $B^\beta_{11}$ where $\beta$ can be arbitrarily close to 1. [The space $B^1_{11}$ is the well known "bump algebra" which contains functions of "considerable spatial inhomogeneity", \cite{DJ98}.]  In the case $p>q$ which is not considered here, the non-adaptive bounds in \cite[Proposition 5.8]{ADH21} suggest that the attained rate is substantially slower than minimax. Although it is not known whether these bounds are sharp, it is shown in \cite[Theorem 4.1]{AW21} that Gaussian priors ($p=2$) are fundamentally limited by the suboptimal linear minimax rate from \eqref{eq:lminimax}.
 
 \subsection{Technical comments}
The proofs of our three main results can be found in Section \ref{sec3}, and broadly follow the techniques of \cite[Section 3.5]{RS17} where adaptive rates of contraction are studied with truncated Gaussian priors in the setting of nonparametric regression. Due to the fact that we consider the broader class of $p$-exponential priors and without  truncation, we need to introduce certain new techniques in order to verify the assumptions of the invoked general theory of \cite{RS17}. {For example, we rely on a Fernique-type theorem for log-concave priors to obtain concentration of measure inequalities with constants which are uniform with respect to the hyper-parameter(s) $\lambda$ and we borrow ideas from the recent literature on posterior contraction for inverse problems to get better controlled sieve sets, see Lemma \ref{lem:fernique}  
and Remark \ref{rem:intersect} in the Supplement below, respectively. A significant amount of effort is also devoted in the careful control of the constants in the small ball probability estimates, uniformly with respect to the hyper-paremeter(s) $\lambda$ as required by the general theory, see Remark \ref{rem:eta} and Lemma \ref{lem:centeredconstants} in the Supplement below}. We close this section with some technical comments on our results.

\begin{rem}\label{rem2.6}{\quad}
\begin{enumerate}
\item[a)] Theorems \ref{thm:Sobolev-tau}, \ref{thm:Sobolev-alpha} and \ref{thm:Besov} assume that  $\beta\ge (1+\alpha p)/(p+2\alpha p)$, $\beta\in(\ubar{\alpha},\bar{\alpha})$ and $\beta\in(\ubar{\alpha}+1/p, \bar{\alpha}+1/p)$, respectively. These assumptions on the regularity of the truth are stronger than the corresponding assumptions in non-adaptive settings: $\beta>0$ under Sobolev regularity and $\beta>1/p$ under Besov regularity, see Lemmas \ref{lem:Lemma3.3} and \ref{lem:BBesov} below, respectively. The source of these more stringent conditions, is the fact that in the upper and lower bounds for the centered small ball probabilities with respect to $\alpha$-regular $\tau$-scaled $p$-exponential priors, it is impossible to choose the constants uniformly for all $\tau>0$ and $\alpha>0$; see Lemma \ref{lem:2.2b} below as well as Lemma \ref{lem:centeredconstants} in the Supplement below. Since the employed general theory of \cite{RS17} (naturally) requires a uniform control of the constants appearing in the small ball probability bounds with respect to the prior $\Pi(\cdot|\lambda)$ for all $\lambda$ in the set of candidate hyper-parameter values $\Lambda_n$, these assumptions appear to be unavoidable at present. Nevertheless, the additional restrictions are quite mild and do not obscure the main message of this article.
\item[b)] {More specifically, for fixed prior regularity $\alpha$ and varying scaling $\tau$, the upper and lower bounds for the centered small ball probabilities hold with uniform constants if $\gep\leq \tilde{C}\tau$, for any fixed $\tilde{C}>0$, see Lemma \ref{lem:2.2b} below. Our upper bounds on $\gep_n(\alpha,\tau)$ solving \eqref{eq:en}, suggest that adaptation is possible for truth regularity $\beta<\alpha+1/p$ by choosing $\tau=n^{(\alpha-\beta)/(1+2\beta)}$, see Lemma \ref{lem:Lemma3.3}(i) below. Hence, the smallest $\tau$ we need to consider is $n^{-1/(2+p+2\alpha p)}$ and the assumption $\beta\geq (1+\alpha p)/(p+2\alpha p)$ is imposed in order to ensure that $\gep_n(\alpha,\tau)\leq \tilde{C}\tau$ for all considered $\tau$. Note that for $p=2$ the assumption becomes $\beta\ge 1/2$, while as $p$ decreases it becomes more stringent. This is because, for a fixed truth regularity $\beta$, when $p$ is smaller adaptation can be achieved down to smaller prior regularities $\alpha$ via faster decaying choices of the scaling $\tau$. Hence, for a smaller $p$ we need to consider smaller values of $\tau$, which imposes a stricter assumption on $\beta$ so that $\gep_n(\alpha,\tau)\leq \tilde{C}\tau$.} 

\item[c)]In all three theorems, the set $\Lambda_0$ on which the MMLE concentrates is only implicitly defined in \eqref{eq:L0Mnint}. Upper bounds for $\gep_{n}(\lambda)$ and $\gep_{n,0}$ for Sobolev truths, can be found in  Lemmas \ref{lem:Lemma3.2} and \ref{lem:Lemma3.3} below (by \cite[Proposition 5.4]{ADH21} these bounds hold more generally for Besov truths with $q\ge2$). For Besov truths with $q<2$, upper bounds for $\gep_n(\lambda)$ are implicitly computed in the proof of \cite[Proposition 5.8]{ADH21} while bounds for $\gep_{n,0}$ are given in the statement of the same result. In general, the upper bounds on $\gep_n(\lambda)$ are not known to be sharp (see the brief discussion after \eqref{eq:lowerconc} in Subsection \ref{ssec:pexp}). Crude lower bounds based on the logarithmic-concavity of $p$-exponential measures are derived in Lemma \ref{lem:2.2b} below, however, there is a gap between the upper and lower bounds when the prior is too smooth compared to the truth. As a result, we cannot explicitly characterize $\Lambda_0$. For the same reason, we do not study lower bounds on the rate of contraction.
\item[d)] Explicitly identifying $\Lambda_0$ is not necessary for our results. In fact, similarly to \cite[Remark 3.4]{RS17} our empirical Bayes results are stronger than the minimax results stated in Theorems \ref{thm:Sobolev-tau}-\ref{thm:Besov}. The invoked Theorem 2.1 and Corollary 2.1  of \cite{RS17} in the proof of our results, show that the MMLE Bayes posterior contracts around the truth for every $\theta_0\in\Theta$ with rate $\gep_{n,0}(\theta_0)$ which can be faster than the minimax rate. A tighter understanding of $\Lambda_0$ is, however, required for the hierarchical result, in particular we need to identify a subset of $\Lambda_0$ with sufficient mass under the hyper-prior, see \cite[Section 2.2]{RS17} or assumption \eqref{eq:H1a} below. Due to the gap between the available upper and lower bounds we cannot identify such a subset for arbitrary sequences $M_n\to\infty$, but we can do so for $M_n$ such that $M_n\gep_{n,0}$ matches our optimized upper bounds on $\gep_n(\lambda)$ (for details see e.g. the proof of Theorem \ref{thm:Sobolev-tau} below). As a result, for the hierarchical posterior we can only invoke  \cite[Theorem 2.3]{RS17} with this choice of $M_n$, and thus can only establish the minimax results as stated in our theorems.
\end{enumerate}
\end{rem}


\section{Proofs of main results}\label{sec3}

\subsection{Notation and assumptions of the general contraction theory of \cite{RS17}}\label{sec:conditions}
In this subsection we briefly recall the assumptions of the general contraction theory for empirical and hierarchical Bayes priors of \cite{RS17}, which we employ in order to establish our main results. We formulate them specifically for the white noise model, and restrict our attention on the conditions which suffice for obtaining upper bounds on the rates of contraction.

Recall the definitions \eqref{eq:en}-\eqref{eq:L0Mnint} from Subsection \ref{ssec:mmle}.
Consider the transformation $\psi_{\lambda, \lambda'}: \Theta \mapsto \Theta$ for all $\lambda, \lambda' \in \Lambda_n$, which is such that if $\theta \sim \Pi(\cdot \mid \lambda)$ then $\psi_{\lambda, \lambda'}(\theta) \sim \Pi(\cdot \mid \lambda')$. For the different choices of hyper-parameter $\lambda$ studied in this article, the transformation of a $\theta=(\theta_j)_{j\in\N}\in\Theta$ can be expressed as $\psi_{\lambda,\lambda'}(\theta)=\big(\psi_{\lambda, \lambda'}(\theta_j)\big)_{j\in\N}$, where
\begin{equation}\label{eq:psi}
\psi_{\lambda, \lambda'}(\theta_j)=
\begin{cases}    \vspace{0.2cm}
  \frac{\tau'}{\tau}  \theta_{j}, & \text {for $\lambda=\tau$,}\\\vspace{0.2cm}
   j^{\alpha-\alpha'} \theta_{j}, & \text{for $\lambda=\alpha$,}\\
   \frac{\tau'}{\tau} j^{\alpha-\alpha'} \theta_{j}, & \text {for $\lambda=(\alpha, \tau)$}.
 \end{cases}\end{equation}
Given a loss function $\rho: \Lambda_n \times \Lambda_n \to \R^{+}$ and a sequence $u_n\to0$, introduce the notation 
\begin{equation}\label{eq:Q}
q_{\lambda, n}^{\theta} (X) = \sup_{\rho(\lambda, \lambda') \leq u_n} p_{\psi_{\lambda, \lambda'}(\theta)}^{n} (X), \quad X\in\mX
\end{equation}
and let $Q_{\lambda, n}^{\theta}$ be the associated measure. In our setting, the loss function will have the following three forms depending on the choice of hyper-parameter $\lambda$:
\begin{equation}\label{eq:rhodefn}
\rho(\lambda, \lambda')=
\begin{cases}    
   \mid \log \tau - \log \tau' \mid, & \text {for $\lambda=\tau$,}\\
    \mid \alpha - \alpha' \mid, & \text{for $\lambda=\alpha$,}\\
    \mid \log \tau - \log \tau' \mid + \mid \alpha - \alpha' \mid, & \text {for $\lambda=(\alpha, \tau)$}.
 \end{cases}\end{equation}
Additionally, denote by $N_n(\Lambda_0), N_n(\Lambda_n\setminus\Lambda_0)$ and $N_n(\Lambda_n)$ the covering numbers of $\Lambda_0, \Lambda_n\setminus \Lambda_0$ and $\Lambda_n$ by $\rho$-balls of radius $u_n$, respectively.

The following assumptions are employed in \cite[Theorem 2.1]{RS17} in order to show that given $\theta_0\in\Theta$, the MMLE $\hat{\lambda}_n$ belongs to $\Lambda_0=\Lambda_0(M_n)$ from \eqref{eq:L0Mnint}, with $P^n_{\theta_0}$-probability tending to 1 as $n\to\infty$. This is done by establishing that the marginal likelihood is small for $\lambda\in\Lambda_n\setminus\Lambda_0$.
\begin{itemize}
\item (A1) There exists $N>0$ such that for all $\lambda \in \Lambda_n \backslash \Lambda_0$ and $n\geq N$, there exists $\Theta_n(\lambda) \subset \Theta$ such that
\begin{equation}\label{eqA1.1}
\sup_{\{ \norm{\theta-\theta_0}_2 \leq K\gep_{n}(\lambda)\} \cap \Theta_{n}(\lambda) } \frac{\log Q_{\lambda, n}^{\theta}(\mX)}{n \gep_{n}^2 (\lambda)} = o(1),
\end{equation}
and 
\begin{equation}\label{eqA1.2}
\int_{\Theta_{n}(\lambda)^{c}} Q_{\lambda, n}^{\theta}(\mX) d\Pi(\theta \mid \lambda) \leq e^{-w_n^2 n \gep_{n,0}^2},
\end{equation}
for some positive sequence $w_n$ going to infinity. 

\item (A2) There exist $\zeta >0, c_1 <1$ such that for all $\lambda \in \Lambda_n \backslash \Lambda_0$ and all $\theta \in \Theta_{n}(\lambda)$, there exist tests $\varphi_{n}(\theta)$ such that
\begin{equation} \label{eqA2.1} 
 E^{n}_{\theta_{0}} \varphi_{n}(\theta)  \leq e^{-c_1 n \norm{\theta-\theta_{0}}^{2}_{2}}, \\ 
 \sup_{\substack{\theta' \in \pth(\lambda)  \\ \norm{\theta-\theta'}_{2}< \zeta \norm{\theta-\theta_{0}}_{2}}} \int_{\mX} \big(1-\varphi_{n}(\theta)\big) dQ^{\theta'}_{\lambda, n} (X) \leq e^{-c_1 n \norm{\theta-\theta_{0}}^{2}_{2}}.
 \end{equation}
 In addition, again for $\lambda\in\Lambda_n\setminus\Lambda_0$, it holds
\begin{equation} \label{eqA2.2} 
\log N\big(\zeta u, \{ u \leq \norm{\theta-\theta_0}_2 \leq 2u \} \cap \Theta_{n}(\lambda), \norm{\cdot}_2\big) \leq c_1 n u^2/2
\end{equation}
for all $u\geq K\gep_n(\lambda)$, where $K$ is as in \eqref{eq:en}.
\end{itemize}
In general, there are two additional conditions in \cite{RS17}, which in our assumed setting hold trivially. The first one (equation (2.8) in that source), compares neighbourhoods with respect to the norm of the parameter space $\Theta$ to neighbourhoods with respect to the metric used as loss in the definition of contraction rates. Since here both of these are given by the $L_2$-norm, this comparison is trivial (more specifically note that for $w_n=o(M_n)$ as we will choose later on and for $\lambda\in\Lambda_n\setminus\Lambda_0$, $c(\lambda)$ in (2.8) of \cite{RS17} is allowed to take the constant value $K$ from \eqref{eq:en}).  The second one (assumption (B1) in the same source), compares balls with respect to the norm of the parameter space (here $L_2$-norm) to Kullback-Leibler neighbourhoods. Since we are interested in the white noise model, \cite[Lemma 8.30]{GV17} shows that these two coincide, hence this assumption also holds trivially. Moreover, note that the assumption on the type-II error bound of the tests in \cite[Equation (2.7)]{RS17}, is formulated without restricting $\theta'$ on the sieve set $\Theta_n(\lambda)$. However, an inspection of the proof of \cite[Theorem 2.1]{RS17} where this assumption is used, immediately reveals that in all occurrences $\theta'$ does belong to $\Theta_n(\lambda)$, hence we can restrict the supremum on $\Theta_n(\lambda)$ as in \eqref{eqA2.1} above.

Having established that the MMLE belongs to $\Lambda_0$ with  $P^n_{\theta_0}$-probability tending to 1, the following additional assumptions are employed in \cite[Corollary 2.1]{RS17} to establish that the MMLE empirical Bayes posterior contracts at rate $M_n\gep_{n,0}$. This is done by controlling $\Pi(\norm{\theta-\theta_0}_2\leq M_n\gep_{n,0}|X^{(n)},\lambda)$ uniformly over $\lambda\in\Lambda_0$.
\begin{itemize}
\item (C1) For every $c_2>0$ there exists an $N>0$ such that for all $\lambda \in \Lambda_0$ and $n \geq N$, there exists $\Theta_n(\lambda)$ satisfying
\begin{equation}\label{eqC1}
\sup_{\lambda \in \Lambda_0} \int_{\Theta_{n}(\lambda)^{c}} Q_{\lambda, n}^{\theta}(\mX) d\Pi(\theta \mid \lambda) \leq e^{-c_2 n \gep_{n,0}^2}.
\end{equation}

\item (C2) There exist $\zeta <1$ and $c_1 >0$ such that for all $\lambda \in \Lambda_0$ and all $\theta \in \Theta_{n}(\lambda)$, there exist tests $\varphi_{n}(\theta)$ satisfying \eqref{eqA2.1} and \eqref{eqA2.2}, where \eqref{eqA2.2} is supposed to hold for all $u \geq M M_n \gep_{n,0}$ for some $M>0$.

\item (C3) There exists $C_0 > 0$ such that for all $\lambda \in \Lambda_0$ and for all $\theta \in \{ \norm{\theta-\theta_0}_2 \leq M_n \gep_{n,0}\} \cap \Theta_{n}(\lambda)$,
\begin{equation}\label{eqC3}
\sup_{\rho(\lambda, \lambda') \leq u_n} \norm{\theta-\psi_{\lambda, \lambda'}(\theta)}_2 \leq C_0 M_n \gep_{n,0}.
\end{equation}
\end{itemize}

The last set of additional assumptions is used in \cite[Theorem 2.3]{RS17} to show that the hierarchical Bayes approach attains the same rates, $M_n\gep_{n,0}$, as the empirical Bayes procedure. The idea here, is that the hyper-prior on $\lambda$ must charge the set $\Lambda_0$ of 'good values' of $\lambda$ with sufficient mass.
For a sequence $\tilde{w}_n$ satisfying $\tilde{w}_n=o(M_n\wedge w_n)$ where $w_n$ is as in (A1), denote by $\Lambda_0(\tilde{w}_n)$ the set from \eqref{eq:L0Mnint}.

\begin{itemize}
\item(H1) We assume that there exist $\tilde{\Lambda}_0 \subseteq \Lambda_0(\tilde{w}_n)$, such that for some sufficiently small $\bar{c}_0>0$, there exists $N>0$ such that
\begin{equation}\label{eq:H1a}
\int _{\tilde{\Lambda}_0} \tilde{\pi}(\lambda) d\lambda \gtrsim e^{-n \tilde{w}_n^2 \gep_{n,0}^2}
\end{equation}
and
\begin{equation}\label{eq:H1b}\int _{\Lambda_n^c} \tilde{\pi}(\lambda) d\lambda \leq e^{- \bar{c}_0 n \gep_{n,0}^2},\end{equation}
for all $n \geq N$.
\item(H2) Uniformly over $\lambda\in\tilde{\Lambda}_0$ (where $\tilde{\Lambda}_0$ is as in (H1) above) and $B_n:=\{\theta:\norm{\theta-\theta_0}_2\leq K\gep_n(\lambda)\}$, there exists $c_3>0$ such that 
\[P^n_{\theta_0}\Big\{\inf_{\rho(\lambda,\lambda')\leq u_n}\ell_n\big(\psi_{\lambda,\lambda'}(\theta)\big)-\ell_n\big(\theta_0\big)\leq -c_3 n\eps_n(\lambda)^2\Big\}= O\big(e^{-n\eps_{n,0}^2}\big).\]
\end{itemize}
{For sieve sets $\Theta_n(\lambda)$ for which there exists a constant $R>1$ such that
\begin{equation}\label{eq:aux}\Pi(B_n|\lambda)\le R\Pi\big(B_n\cap\Theta_n(\lambda)|\lambda)\big), \quad \forall \lambda\in\Lambda_n,\end{equation} condition (H2) with $\theta$ restricted on $B_n\cap \Theta_n(\lambda)$ is sufficient.} Indeed, (H2) is used in the proof of \cite[Theorem 2.3]{RS17}  to upper bound the posterior probability of $B_n^c$ uniformly for $\lambda\in\Lambda_0(M_n)$ {(the denominator in the posterior is lower bounded by an integral over $B_n$)}, by referring to the proof of \cite[Theorem 1]{DRRS18} where the latter theorem's assumptions are satisfied with $\tilde{B}_n=B_n$. It is straightforward to check that under \eqref{eq:aux}, the latter assumptions are also satisfied for $\tilde{B}_n=B_n\cap\Theta_n(\lambda)$, hence the proof of \cite[Theorem 2.3]{RS17} is still valid when restricting (H2) on $B_n\cap\Theta_n(\lambda)$.
Finally, note that \eqref{eq:H1a} is slightly weaker than the first part of (H1) of \cite{RS17}, which has $\exp(-n\gep_{n,0}^2)$ on the right hand side. An inspection of the proof of \cite[Theorem 2.3]{RS17}, in particular of the last two displays where this assumption is used in order to show that the hyper-posterior $\tilde{\pi}(\lambda|X^{(n)})$ has little mass outside $\Lambda_0(M_n)$, immediately reveals that \eqref{eq:H1a} is sufficient.

We thus summarize the general results of \cite{RS17}, in the white noise model setting:

\begin{cor}[Theorem 2.1, Corollary 2.1 and Theorem 2.3 of \cite{RS17} in the white noise model]\label{corRS17}\quad\\
Let $M_n$ tend to infinity arbitrarily slowly.
\begin{itemize}
\item[i)] Assume that there exists $K>0$ such that conditions (A1) and (A2) hold with $w_n=o(M_n)$. Then if $\log N_n(\Lambda_n\setminus \Lambda_0)=o(nw_n^2\gep^2_{n,0})$, 
\[\lim_{n\to\infty}P^n_{\theta_0}\big(\hat{\lambda}_n\in\Lambda_0(M_n)\big)=1.\]
\item[ii)] Assume that $\hat{\lambda}_n\in\Lambda_0(M_n)$ with probability going to 1 under $P^n_{\theta_0}$ and that conditions (C1)-(C3) hold. Then if $\log N_n(\Lambda_0)\leq O(n\gep_{n,0}^2)$, there exists $M>0$ such that the empirical Bayes posterior satisfies
\[E^n_{\theta_0}\Pi\Big(\theta\,:\,\norm{\theta-\theta_0}_2 \ge M M_n\gep_{n,0}|X^{(n)}, \hat{\lambda}_n\Big)=o(1).\]
\item[iii)] Assume that the conditions in the previous two items hold and consider a hyper-prior on $\lambda$ which satisfies the conditions (H1) and (H2) with $\tilde{w}_n=o(M_n\wedge w_n)$. Then there exists $M>0$ such that the hierarchical posterior satisfies
\[E^n_{\theta_0}\Pi\Big(\theta\,:\,\norm{\theta-\theta_0}_2 \ge MM_n\gep_{n,0} | X^{(n)}\Big)=o(1).\]
\end{itemize}
\end{cor}

\subsection{Bounds on $\gep_n(\lambda)$ for $p$-exponential priors}

In Lemma \ref{lem:existenceenl} in the Supplement below, it is shown that for $\alpha$-regular $\tau$-scaled $p$-exponential priors, there exists a unique solution $\gep_n=\gep_n(\alpha,\tau; K,\theta_0)$ to equation \eqref{eq:en}, which is strictly decreasing to zero. [In fact the same proof works for any 'continuous' prior supported in $\Theta$ {(more precisely, any prior which puts positive mass on all balls in $\Theta$ but no mass on any sphere in $\Theta$)}.]  Hereafter, we will suppress the dependence of $\gep_n$ on $K$ and $\theta_0$ in the notation. In this subsection we derive bounds for $\gep_n$ based on the concentration function defined in \eqref{defn:concfcn}.

The first basic result, associates the concentration function to $\gep_n$ and establishes a (naive) lower bound which holds for all $\theta_0\in\Theta$.

\begin{lem}\label{lem:2.2b}
Fix $K>0$, let $\at>0$ and $n \geq 1$, and consider $\gep_{n} = \gep_{n}(\at)$ as defined in \eqref{eq:en}. 
Then $\gep_{n}$ satisfies 
\begin{equation} \label{lem2.2_2}
\varphi_{0}(K\gep_{n}) \leq n\gep_{n}^{2} \leq \varphi_{\theta_0}\Big(\frac{K\gep_n}{2}\Big),
\end{equation}
where $\varphi_{\theta}(\cdot)$ is the concentration function of $\Pi(\cdot \mid \at)$ at $\theta \in \Theta$, as defined in \eqref{defn:concfcn}.

Furthermore, for any $\tilde{C}>0$, there exists $\tilde{c}_1=\tilde{c}_1(\alpha, \tilde{C})$ such that for all $\gep, \tau>0$ satisfying $\gep\leq \tilde{C}\tau$, it holds
\begin{equation} \label{lem2.2_3}
\tilde{c}_1^{-1} \Big(\frac{K\gep}{\tau}\Big)^{-1/\alpha} \leq \varphi_{0}(K\gep) \leq \tilde{c}_1 \Big(\frac{K\gep}{\tau}\Big)^{-1/\alpha}.
\end{equation}
The constant $\tilde{c}_1$  can be chosen uniformly for $\alpha>0$ in any bounded interval $[\ubar{\alpha}, \bar{\alpha}]$ with $\bar{\alpha}>\ubar{\alpha}>0$.

In particular, the above bounds imply that 
\begin{equation} \label{lem2.2_4}
n \gep_{n}^{2} \geq \tilde{c}_1^{-1} \Big(\frac{K}{\tau}\Big)^{-1/\alpha} \gep_{n}^{-1/\alpha},
\end{equation}
and
\begin{equation} \label{lem2.2_5}
\gep_{n} \geq (\tilde{c}_1^{-1} K^{-1/\alpha})^{\frac{\alpha}{1+2\alpha}} \tau^{\frac{1}{1+2\alpha}} n^{-\frac{\alpha}{1+2\alpha}},
\end{equation}
provided $\gep_n\leq \tilde{C}{\tau}$.
\end{lem}

\begin{proof}
By \cite[Theorem 2.13]{ADH21} we have that
$$\Pi (\theta: \norm{\theta-\theta_0}_2 \leq K\gep_{n} \mid \at) \geq e^{-\varphi_{\theta_0}(\frac{K\gep_n}{2})},$$
and by Anderson's inequality, see \cite[Proposition 2.4]{ADH21}, and the definition of the concentration function, we have that
$$\Pi (\theta: \norm{\theta-\theta_0}_2 \leq K\gep_{n} \mid \at) \leq \Pi (\theta: \norm{\theta}_2 \leq K\gep_n \mid \at) = e^{-\varphi_{0}(K\gep_n)}.$$
Combining these bounds with \eqref{eq:en} we get \eqref{lem2.2_2}. Using Lemma \ref{lem:centeredconstants} in the Supplement below (which follows from \cite[Theorem 4.2]{FA07} and other related work of the same author) we also get \eqref{lem2.2_3}.

The expressions in \eqref{lem2.2_4} and \eqref{lem2.2_5} follow successively using \eqref{lem2.2_2}, the lower bound in \eqref{lem2.2_3} and simple computations.
\end{proof}

We next employ Lemma \ref{lem:2.2b} to derive explicit upper bounds on $\gep_n(\at)$, depending on the regularity of the truth $\theta_0$. 

\begin{lem}\label{lem:Lemma3.2} Assume that $\Pi(\cdot \mid \at)$, $\alpha>0, \tau>0$, is an $\alpha$-regular $\tau$-scaled $p$-exponential prior in $\Theta$. Then, we have the following upper bounds for $\gep_n(\at)$ defined in \eqref{eq:en}, depending on the regularity of $\theta_0$:
\begin{enumerate}
\item[i)] For $\theta_0 \in H^{\beta}$, $\beta>0$, we have
\begin{equation}\label{gepn_t} \gep_{n}(\at)  \lesssim \tilde{\gep}_n(\at) :=
\begin{cases}    
   n^{-\frac{\alpha}{1+2\alpha}} \tau^{\frac{1}{1+2\alpha}} + (n\tau^{p})^{\frac{\beta}{\beta(p-2)-\alpha p-1}}, & \text {$\beta < \alpha + 1/p$}\\
    n^{-\frac{\alpha}{1+2\alpha}} \tau^{\frac{1}{1+2\alpha}} + \frac{1}{\sqrt{n\tau^{p}}}, & \text{$\beta > \alpha + 1/p$}\\
    n^{-\frac{\alpha}{1+2\alpha}} \tau^{\frac{1}{1+2\alpha}} + \frac{1}{\sqrt{n\tau^{p}}} \Big[\log(\sqrt{n\tau^{p}})\Big]^{\frac{1}{2}-\frac{p}{4}}, & \text {$\beta = \alpha + 1/p$}.
 \end{cases}\end{equation}
 \item[ii)]For $\theta_0\in B^\beta_{qq}$, where $\beta\ge 1/p$, $1\leq q<2$ and $p\leq q$, we have
 \begin{equation}\label{gepn_tBesov} \gep_{n}(\at)  \lesssim \tilde{\gep}_n(\at) :=
\begin{cases}    
   n^{-\frac{\alpha}{1+2\alpha}} \tau^{\frac{1}{1+2\alpha}} + (n\tau^{p})^{-\frac{2\beta q+q-2}{4\beta q+4q-4-2\beta pq+2\alpha pq}}, & \text {$\beta < \alpha + 1/p$}\\
    n^{-\frac{\alpha}{1+2\alpha}} \tau^{\frac{1}{1+2\alpha}} + \frac{1}{\sqrt{n\tau^{p}}}, & \text{$\beta > \alpha + 1/p$}\\
    n^{-\frac{\alpha}{1+2\alpha}} \tau^{\frac{1}{1+2\alpha}} + \frac{1}{\sqrt{n\tau^{p}}} \Big[\log(\sqrt{n\tau^{p}})\Big]^{\frac{q-p}{2q}}, & \text {$\beta = \alpha + 1/p$}.
 \end{cases}\end{equation}
 \end{enumerate}
 The constants in the above bounds depend on $\theta_0$ only through the norms $\norm{\theta_0}_{H^\beta}$ and $\norm{\theta_0}_{B^\beta_{qq}}$, respectively. Furthermore, the constants can be chosen uniformly over $\tau$ such that $\gep_n/\tau$ is bounded.
\end{lem}

\begin{proof}
Both cases can be treated simultaneously. Using the definition of the concentration function of $\Pi(\cdot \mid \at)$, cf. \eqref{defn:concfcn}, together with \eqref{lem2.2_3} and Lemma 5.13 of \cite{ADH21} (applied with $q\leq 2$ and $p\leq q$), we obtain that
\begin{equation}\label{concfcn_t} 
\varphi_{\theta_{0}}(\gep_{n}) \lesssim 
\begin{cases}    
    \tau^{-p} \gep_{n}^{2q\frac{\beta p - \alpha p -1}{2\beta q+q-2}} + \gep_{n}^{-\frac{1}{\alpha}} \tau^{\frac{1}{\alpha}}, & \text {$\beta < \alpha + 1/p$}\\
    \tau^{-p} + \gep_{n}^{-\frac{1}{\alpha}} \tau^{\frac{1}{\alpha}}, & \text {$\beta > \alpha +1/p$}\\
        \tau^{-p} \Big[\log \big( \frac{1}{\gep_{n}} \big)\Big]^{\frac{q-p}{q}} + \gep_{n}^{-\frac{1}{\alpha}} \tau^{\frac{1}{\alpha}}, & \text{$\beta = \alpha + 1/p$},
 \end{cases}\end{equation}
where the constant in the above upper bound depends on $\theta_0$ only through the norm  $\norm{\theta_0}_{B^\beta_{qq}}$, and can be chosen uniformly over $\tau$ such that $\gep_n/\tau$ is bounded.
The claimed upper bounds in \eqref{gepn_t} can thus be obtained using (the second inequality in) \eqref{lem2.2_2}. The calculations for $\beta\neq\alpha+1/p$ are straightforward, while in the case $\beta=\alpha+1/p$ one can solve $n\gep_{n}^{2}\leq\tau^{-p} \big[\log \big(1/\gep_{n}\big)\big]^{(q-p)/q}$ using \cite[Lemma G.2]{ADH21}.
\end{proof}


We next optimize the upper bounds $\tilde{\gep}_n(\at)$ of $\gep_n(\at)$ from the first part of the last lemma, over $\tau$ and $\alpha$, separately. We thus derive upper bounds for $\gep_{n,0}$ defined in \eqref{defn:en0}, for a Sobolev truth.

\begin{lem}\label{lem:Lemma3.3} Assume that $\Pi(\cdot \mid \at), \alpha>0, \tau>0$ is an $\alpha$-regular $\tau$-scaled $p$-exponential prior in $\Theta$ and that $\theta_0 \in H^{\beta}$, $\beta>0$. The upper bounds $\tilde{\gep}_n(\alpha,\tau)$ from Lemma \ref{lem:Lemma3.2} are optimized over the choices of $\tau$ and $\alpha$ as follows:
\begin{enumerate}
\item[i)] Given fixed $\alpha>0$, it holds \[\tilde{\gep}_n(\at)\gtrsim \bar{\gep}_n:=\tilde{\gep}_n(\alpha,\tau_0), \;\;\text{for all}\;\;{\tau>0},\]
where
\begin{equation}\label{gepn0_t} 
\tilde{\gep}_n(\alpha,\tau_0) \asymp
\begin{cases}    
    \mathrlap{n^{-\frac{\beta}{1+2\beta}}}\hphantom{n^{-\frac{\beta}{1+2\beta}} \big(\log{n}\big)^{\frac{2-p}{2p(1+2\beta)}}}\quad \text{for} \quad \tau_0 = n^{\frac{\alpha-\beta}{1+2\beta}}, &\text{$\beta < \alpha + 1/p$}\\
      \mathrlap{n^{-\frac{1+\alpha p}{2+p(1+2\alpha)}}}\hphantom{n^{-\frac{\beta}{1+2\beta}} \big(\log{n}\big)^{\frac{2-p}{2p(1+2\beta)}}}\quad \text{for} \quad \tau_0 = n^{-\frac{1}{2+p(1+2\alpha)}},  &\text{$\beta > \alpha + 1/p$}\\
    n^{-\frac{\beta}{1+2\beta}} \big(\log{n}\big)^{\frac{2-p}{2p(1+2\beta)}} \quad\text{for} \quad \tau_0 =s_n^{\frac{1}{\beta p}} \Big[\log \Big( \frac{1}{s_{n}} \Big)\Big]^{\frac{(2-p) (\beta p-1)}{2 \beta p^{2}}},  &\text {$\beta = \alpha + 1/p$},
 \end{cases}\end{equation}
and where $s_n:=\tilde{\gep}_n(\beta-1/p,\tau_0)\asymp n^{-\beta/(1+2\beta)} \big(\log{n}\big)^{(2-p)/ [2p(1+2\beta)]}$.
\item[ii)] For fixed $\tau=1$, it holds
\begin{equation}\label{gepn0_a} 
\tilde{\gep}_n(\alpha,1)\gtrsim \bar{\gep}_n:=\tilde{\gep}_n(\beta,1)
, \;\;\text{for all}\;\;{\alpha>0},
\end{equation}
where \[\tilde{\gep}_n(\beta,1)\asymp \minimax,\] for $\minimax$ the (global) minimax rate defined in \eqref{eq:minimax}.
\end{enumerate}
The constants in the above bounds depend on $\theta_0$ only through the norm $\norm{\theta_0}_{H^\beta}$.
\end{lem}

\begin{proof}
Fix $\alpha>0$ and notice that the upper bounds in Lemma \ref{lem:Lemma3.2} (in all three cases) consist of two terms, one  increasing and one decreasing in $\tau$. We can thus optimize the choice of $\tau$, by balancing the respective terms. For $\beta\neq\alpha+1/p$, a straightforward calculation shows that the two terms are balanced for the stated choices of $\tau$, $\tau_0$, resulting in the claimed bounds. The same holds for $\beta=\alpha+1/p$ and $p=2$, in which case the upper bound in \eqref{gepn_t} is the same as for $\beta>\alpha+1/p$. 
The case $\beta = \alpha + 1/p$ with $1\leq p<2$, is more complicated and we slightly alter the order of the steps of the proof: we first determine $\tau_0$ as the $\tau$ balancing the two terms in the upper bound, subsequently denoted by $u_{\theta_0}(\gep_n,\tau)$, on the concentration function from \eqref{concfcn_t}. We then solve $u_{\theta_{0}}(\gep_{n}, \tau_0)=n\gep_n^2$ in order to find the corresponding optimized bound $\tilde{\gep}_n(\alpha,\tau_0)$; notice that this change in order does not affect the bounds on $\gep_n(\alpha,\tau)$. Balancing the two terms on the right hand side of \eqref{concfcn_t} gives \[\tau_0=\gep_n^{1/(\beta p)} \big[\log \big(1/ \gep_{n}\big)\big]^{(2-p) (\beta p-1)/(2 \beta p^{2})}.\] The equation $u_{\theta_{0}}(\gep_{n}, \tau_0)=n\gep_n^2$ can then be seen to be equivalent to \[\gep_n^{\frac{1+2\beta}{\beta}}[\log\big(1/\gep_n\big)]^{-\frac{2-p}{2\beta p}}=1/n,\] and using \cite[Lemma G.2]{ADH21} we obtain the bound claimed in \eqref{gepn0_t}. {The potential blow up of the constants in \eqref{gepn_t} as we vary $\tau$ is treated in Remark \ref{rem:rem:rem} below.} Finally, for fixed $\tau=1$, it is straightforward to verify the claimed bound \eqref{gepn0_a}. 
\end{proof}

\begin{rem}\label{rem:rem:rem}
A straightforward calculation shows that for all relationships between $\alpha, \beta$ and $p$, the optimized bounds in the case of fixed $\alpha$ and varying $\tau$ satisfy $\tilde{\gep}_n(\alpha,\tau_0)\lesssim \tau_0$. Therefore,  for $\tau=\tau_0$ the constant in \eqref{gepn_t} does not blow up  as $n\to\infty$, and $\tilde{\gep}_n(\alpha,\tau_0)$ are indeed upper bounds for $\gep_n(\alpha, \tau_0)$  hence also for $\gep_{n,0}$.
\end{rem}

Optimized upper bounds for $\gep_n(\alpha,\tau)$ when the truth is in a $B^\beta_{qq}$ Besov space with $p\leq q<2$, were studied in \cite[Proposition 5.8]{ADH21}. For the reader's convenience, we recall these bounds.

\begin{lem}\label{lem:BBesov} Assume that $\Pi(\cdot \mid \at), \alpha>0, \tau>0$ is an $\alpha$-regular $\tau$-scaled $p$-exponential prior in $\Theta$ and that $\theta_0 \in B^{\beta}_{qq}$, where $\beta\ge1/p$, $p\leq q$ and $1\leq q<2$. The upper bounds $\tilde{\gep}_n(\alpha,\tau)$ from part (ii) of Lemma \ref{lem:Lemma3.2} are optimized over the choices of $\tau$ and $\alpha$ as follows:
 \[\tilde{\gep}_n(\at)\gtrsim \bar{\gep}_n:=\tilde{\gep}_n\big(\beta-1/p,\tau_0\big), \;\;\text{for all}\;\;{\alpha, \tau>0},\]
where
\[\tilde{\gep}_n\big(\beta-1/p,\tau_0\big)\asymp \minimax(\log{n})^\frac{q-p}{pq(1+2\beta)},\]
and \[\tau_0=n^{-\frac1{p(1+2\beta)}}(\log{n})^\omega,  \;\omega=\Big(p-\frac{1}{1+2\beta}\Big) \frac{q-p}{p^2q} \ge0.\]
The constants in the above bounds depend on $\theta_0$ only through the norm $\norm{\theta_0}_{B^\beta_{qq}}$.
\end{lem}

\begin{rem}\label{rem:rem:rem:Besov}
Similarly to Remark \ref{rem:rem:rem}, it is straightforward to check that the optimized bounds in the last result satisfy $\tilde{\gep}_n\big(\beta-1/p,\tau_0\big)\lesssim \tau_0$. Therefore,  for $\tau=\tau_0$ the constant in \eqref{gepn_tBesov} does not blow up  as $n\to\infty$, and $\tilde{\gep}_n\big(\beta-1/p,\tau_0\big)$ are indeed upper bounds for $\gep_n\big(\beta-1/p, \tau_0\big)$  hence also for $\gep_{n,0}$.
\end{rem}

\subsection{Proofs of main theorems}
In order to prove our three main results, we verify the conditions of Corollary \ref{corRS17} in each of the corresponding settings, in a similar way to the proof of \cite[Proposition 3.2]{RS17}. 
The verification of these conditions is based on a number of lemmas contained in Section  \ref{sec:lemmas} of the Supplement.
\begin{proof}[Proof of Theorem \ref{thm:Sobolev-tau}]
Here we have $\lambda=\tau$, while $\alpha>0$ is fixed. Let $M_n, w_n$ tending to infinity, $w_n=o(M_n)$. Verifying the conditions of Corollary \ref{corRS17} and combining with the bounds on $\gep_n(\tau)$ and $\gep_{n,0}$ in the first part of Lemma \ref{lem:Lemma3.3}, establishes the result. 

We start the proof with some observations. First, note that for all possible relationships between $\alpha, \beta$ and $p$, the optimal choices of $\tau$, denoted by $\tau_0$ in Lemma \ref{lem:Lemma3.3}, are contained in the set of candidate hyper-parameters $\Lambda_n=[n^{-1/(2+p+2\alpha p)}, n^{\alpha}]$. Indeed, for $\beta>\alpha+1/p$ we have $\tau_0=n^{-1/(2+p+2\alpha p)}$, for $\beta<\alpha+1/p$ we have $n^\alpha>\tau_0=n^{(\alpha-\beta)/(1+2\beta)}>n^{-1/(2+p+2\alpha p)}$, while for $\beta=\alpha+1/p$ we have $n^\alpha>\tau_0\ge n^{-1/(2+p+2\alpha p)}(\log n)^\varsigma$ for some $\varsigma\ge0$. 

Moreover, note that there exists $\tilde{C}>0$ such that $\gep_n (\tau) \le \tilde{C} \tau$, for any $\tau\in\Lambda_n$ and in all considered relationships between $\alpha, \beta, p$. Indeed, we use \eqref{gepn_t}. First, notice that $n^{-\alpha/(1+2\alpha)}\tau^{1/(1+2\alpha)}\leq \tau$ is equivalent to $\tau\ge n^{-1/2}$, which indeed holds for any $\tau\in\Lambda_n$. For $\beta<\alpha+1/p$, we check that $(n\tau^p)^{\beta/[\beta(p-2)-\alpha p -1]} \le \tau$ or equivalently $\tau \ge n^{-\beta/(1+\alpha p+2\beta)}$, which holds for all $\tau\ge n^{-1/(2+p+2\alpha p)}$ by the condition $\beta\geq (1+\alpha p)/(p+2\alpha p)$. For $\beta>\alpha+1/p$, we check that $n^{-1/2}\tau^{-p/2}\leq \tau$ or equivalently $\tau\ge n^{-1/(2+p)}$, which holds trivially for all $\tau\geq n^{-1/(2+p+2\alpha p)}$. For $\beta=\alpha+1/p$, we check that {$n^{-1/2}\tau^{-p/2}\big[\log(n\tau^p)\big]^{1/2-p/4}\leq \tau$}, which holds similarly to the previous case [using that for any $\delta>0$ (arbitrarily small), there exists $C_\delta>0$ such that $\log (n\tau^p)\leq C_\delta(n\tau^p)^\delta,$ for all $n\ge1$ and for all $\tau\in\Lambda_n$]. 

Finally, note that the considerations in the last paragraph, enable us to employ Lemma \ref{lem:2.2b} to show that for all $\tau\in\Lambda_n$ we have $\gep_n(\tau)\gtrsim n^{-t}$ where $t<1/2$. By the definition of $\gep_{n,0}$, \eqref{defn:en0}, we thus get that $\gep_n(\tau)\gtrsim \gep_{n,0}$ for all $\tau\in\Lambda_n$.


We take $u_n\asymp n^{-(3+2\alpha)}$. First, we observe that $N_n(\Lambda_n)\leq n^H$ for some $H>0$. Indeed, by the definition of $\rho$, starting from the lower limit of $\Lambda_n$, $n^{-1/(2+p+2\alpha p)}$, the number $N$ of intervals of $\rho$-length $u_n$ required to reach the upper limit $n^\alpha$, is given as $N\asymp \log(n)/u_n\leq n^H$ for a sufficiently large $H$. Since by definition \eqref{defn:en0} it holds that $\gep_{n,0}^2\geq m_n(\log n)/n$ for some $m_n$ tending to infinity arbitrarily slowly, we obtain $\log N_n(\Lambda_n)=o(n\gep_{n,0}^2)$, thus $\log N_n(\Lambda_n\setminus \Lambda_0)=o(nw_n^2\gep^2_{n,0})$ and $\log N_n(\Lambda_0)=O(n\gep_{n,0}^2)$ as required.

Conditions (A1) and (A2) follow from Lemmas \ref{lem:C1}, \ref{lem:E1}, \ref{lem:E2} and \ref{lem:E3}  (applied with $\t_n=n^\alpha$, see also Remark \ref{rem:taugep}), with $c_1=1/32, \zeta=1/4$ and  $K^2\geq 2M'\eta^p\tilde{c}_1^2/c_1$ [for \eqref{eqA1.2}, note that these conditions concern $\tau\in \Lambda_n\setminus \Lambda_0$ so that $nw_n^2\gep_{n,0}^2=o(n\gep_n^2)$]. Conditions (C1) and (C2) also follow from the same lemmas with $c_2=\eta/4, M=K$. For (C1), which is implied by display  \eqref{eq:lemE.2} of the Supplement, note that as explained above it holds $\gep_n(\tau)\geq \gep_{n,0}$ for all $\tau\in\Lambda_0$, and that we can replace the constant $C$ in  \eqref{eq:lemE.2} by 1, by replacing $\eta/2$ by $\eta/4$ in the exponent. For (C2), notice that for $\lambda\in \Lambda_0$ we have $\gep_n(\lambda)\leq M_n\gep_{n,0}$, so that if \eqref{eqA2.2} holds for all $u\geq K\gep_n(\lambda)$, it also holds for  all $u\geq KM_n\gep_{n,0}$. Moreover, notice that condition (C3) is implied by display \eqref{eq:normdifbound} in the Supplement below combined with the definition of $\pth(\at)$ in Lemma \ref{lem:C1}, in particular by the fact that for $\theta\in\pth(\at)$ it holds $\norm{\theta}_{H^s}\lesssim \tau n\gep_n^2$. These considerations complete the proof of part (i).

To prove part (ii) we need to additionally verify conditions (H1) and (H2). Due to the lack of sharp lower bounds for $\gep_n(\lambda)$, we do so for $\tilde{w}_n=\tilde{M}_n \enab/\gep_{n,0},$ for $\enab$ the optimized upper bounds on $\gep_n(\tau)$ defined in Lemma \ref{lem:Lemma3.3} and where $\tilde{M}_n$ tends to infinity arbitrarily slowly. By part (iii) of Corollary \ref{corRS17}, this will secure that the hierarchical Bayes posterior satisfies
\[E^n_{\theta_0}\Pi\Big(\theta\,:\,\norm{\theta-\theta_0}_2 \ge MM_n\gep_{n,0} | X^{(n)}\Big)=o(1),\]
where $M_n\gep_{n,0}=M_n \tilde{M}_n \enab/\tilde{w}_n$ and where, since the only restriction on $\tilde{w}_n$ is that $\tilde{w}_n=o(M_n\wedge w_n)$, the sequences $M_n$ and $w_n$ can be chosen so that $M_n \tilde{M}_n/\tilde{w}_n$ (which in the statement we denote again by $M_n$) tends to infinity arbitrarily slowly.

In order to allow hyper-priors with polynomial tails for large values of $\tau$ (for example inverse Gamma distributions), we use a larger set $\Lambda_n$ compared to the empirical Bayes setting, in particular we set $\Lambda_n=[n^{-1/(2+p+2\alpha p)}, \,\exp({2c_{0} \bar{c}_{0} n \tilde{w}_{n}^{2} \gep_{n,0}^{2}}) ]$ (this change only concerns the proof, since the formulation of the hierarchical prior does not involve $\Lambda_n$). For this larger set $\Lambda_n$, there still exists a constant $\tilde{C}>0$ such that $\gep_n(\tau)\leq \tilde{C}\tau$ for all $\tau\in \Lambda_n$ (with an identical verification as earlier in this proof, since what matters is the lower boundary of $\Lambda_n$) and it still holds that $\gep_n(\tau)\gtrsim \gep_{n,0}$ for all $\tau\in\Lambda_n$. 

An inspection of the proof of \cite[Theorem 2.3]{RS17} reveals that we can use two different choices of $u_n$, one for verifying the assumptions that concern $\lambda\in\Lambda_n\setminus \Lambda_0$ and one for those that concern $\lambda\in\Lambda_0$ (see also \cite[Remark 3.3]{RS17}). On $\Lambda_n\setminus \Lambda_0$ we use $u_n\asymp \exp({-4c_0\bar{c}_0n\tilde{w}_n^2\gep_{n,0}^2})$, and conditions (A1), (A2) can again be verified using Lemmas \ref{lem:C1}, \ref{lem:E1}, \ref{lem:E2} and \ref{lem:E3},  applied with $\bar{\tau}_n=\exp({2c_0\bar{c}_0n\tilde{w}_n^2\gep_{n,0}^2})$. In $\Lambda_0$ we use $u_n\asymp n^{-(3+2\alpha)}$ and for this choice the conditions (C1)-(C3) have been verified in the proof of part (i). 
The corresponding conditions on the (hyper-)entropy still hold too. Indeed, starting from $n^{-1/(2+p+2\alpha p)}$, the number $N$ of intervals of $\rho$-length  $u_n\asymp \exp({-4c_0\bar{c}_0n\tilde{w}_n^2\gep_{n,0}^2})$ required to reach $\exp({2c_{0} \bar{c}_{0} n \tilde{w}_{n}^{2} \gep_{n,0}^{2}})$ satisfies $N\lesssim  n\tilde{w}_n^2\gep_{n,0}^2/u_n$, so that $\log N=O(\tilde{w}_n^2n\gep_{n,0}^2)=o(w_n^2n\gep_{n,0}^2)$ since $\tilde{w}_n=o(w_n)$. This implies that $\log N_n(\Lambda_n\setminus \Lambda_0)\leq \log N_n(\Lambda_n)=o(w_n^2n\gep_{n,0}^2)$ as required. On the other hand, noticing that \eqref{lem2.2_5} implies that $\Lambda_0\subset [n^{-1/(2+p+2\alpha p)}, n^\alpha]$, the second hyper-entropy condition $\log N_n(\Lambda_0)=O(n\gep_{n,0}^2)$ follows for $u_n\asymp n^{-(3+2\alpha)}$ as in the proof of part (i).

Turning to assumption (H1), it is implied by Lemma \ref{lem:3.5} of the Supplement. 
Finally, notice that the sieve sets $\Theta_n(\at)$ as defined in Lemma  \ref{lem:C1}, satisfy \eqref{eq:aux}, since by \eqref{eq:en} and display \eqref{eq:pthbnd} in the Supplement below
\[\Pi(B_n\cap \pth|\lambda)\ge \Pi(B_n|\lambda)-\Pi(\pth^c|\lambda)\geq e^{-n\gep_n^2(\lambda)}-2s_1e^{-\eta n\gep_n^2(\lambda)},\]
for constants $s_1>0$ and $\eta>1$.  Since there exist $C>0$ and $t<1/2$ such that for all $\tau\in\Lambda_n$ it holds $n\gep_n^2(\tau)\ge C n^{1-2t}$, for sufficiently large $n$ the right hand side in the last display can be lower bounded by $\exp({-n\gep_n^2(\lambda)})/2 = \Pi(B_n|\lambda)/2$ for all $\tau\in\Lambda_n$, thus \eqref{eq:aux} holds. As discussed in Subsection \ref{sec:conditions}, under \eqref{eq:aux} it suffices to verify condition (H2) for $\theta\in B_n\cap \pth(\at)$. Indeed, this follows from Lemma \ref{lem:E4} of the Supplement below 
with $c_3 \geq 2 + 2K^2$, recalling that for $\gep_n(\tau)\gtrsim \gep_{n,0}$ for all $\tau\in\Lambda_n$.
\end{proof}

\begin{proof}[Proof of Theorem \ref{thm:Sobolev-alpha}]
Here we have $\lambda=\alpha$, while $\tau=1$. We  verify the conditions of Corollary \ref{corRS17} and combine with the bounds on $\gep_n(\alpha)$ and $\gep_{n,0}$ in the second part of Lemma \ref{lem:Lemma3.3} to establish the result. By \eqref{lem2.2_5}, for all $\alpha>0$ it holds $\gep_n(\alpha)\gtrsim n^{-t}$ with $t<1/2$, hence by definition \eqref{defn:en0} it holds $\gep_n(\alpha)\gtrsim \gep_{n,0}$ for all $\alpha>0$. 

We take $u_n\asymp n^{-3}$. The (hyper-)entropy conditions again follow from the fact that $N_n(\Lambda_n)\leq n^3$ and the fact that $\gep_{n,0}\geq m_n(\log n)/n$ for some $m_n$ tending to infinity arbitrarily slowly.
The verification of conditions (A1), (A2), (C1)-(C3) is almost identical to the $\lambda=\tau$ case in the proof of Theorem \ref{thm:Sobolev-tau}, based on Lemmas  \ref{lem:C1}, \ref{lem:E1}, \ref{lem:E2} and \ref{lem:E3} as well as  Remark \ref{rem:eta} of the Supplement. {Note the absence of the constant $M>0$ appearing in Corollary \ref{corRS17}(ii) from assertion \eqref{eq:thm2328}. An inspection of the proof of \cite[Corollary 2.1]{RS17} shows that the origin of $M>0$ is assumption (C2), which here we verify with $M=K$, where $K$ is defined in \eqref{eq:en} and, in particular, is independent of the sequence $M_n$. As a result, it is straightforward to check that $M$ can be absorbed in the diverging sequence $M_n$.} This proves part~(i).

The proof of part (ii) is similar to part (ii) of Theorem \ref{thm:Sobolev-tau}, with the difference that the verification of condition (H1) is now based on Lemma \ref{lem:3.6} of the Supplement instead of Lemma \ref{lem:3.5} (for $\lambda=\alpha$ we use the same $\Lambda_n$ in both the empirical and hierarchical settings). {Similarly to the proof of part (i), it can be verified that the constant $M$ appearing in Corollary \ref{corRS17}(iii) can be absorbed in $M_n$.}
\end{proof}

\begin{proof}[Proof of Theorem \ref{thm:Besov}]
Here we have $\lambda=(\at)$. We  verify the conditions of Corollary \ref{corRS17} and combine with the bounds on $\gep_{n,0}$ in Lemma \ref{lem:BBesov} to obtain the result. 

Similarly to the proof of Theorem \ref{thm:Sobolev-tau}, using the bounds in \eqref{gepn_tBesov}, it is straightforward to verify that there exists $\tilde{C}>0$ such that for all considered $\beta, p, q$, it holds $\gep_n(\at)\leq \tilde{C}\tau$ for all $(\alpha,\tau)\in\Lambda_n$. It is also straightforward to check that the optimized choices of the hyper-parameter $\lambda=(\at)$, $\alpha=\beta-1/p$ and $\tau=\tau_0$, where $\tau_0$  is defined in Lemma \ref{lem:BBesov}, are contained in $\Lambda_n$.
The fact that $\gep_n(\at)\gtrsim\gep_{n,0}$ for all $(\at)\in\Lambda_n$ and the (hyper)-entropy conditions, follow from a combination of the considerations in the proofs of Theorems \ref{thm:Sobolev-tau} and \ref{thm:Sobolev-alpha}. 

The proof of part (i) is then almost identical to the proofs of Theorems \ref{thm:Sobolev-tau} and \ref{thm:Sobolev-alpha}, where we take $u_n\asymp n^{-(3+2\bar{\alpha})}$ (Lemmas \ref{lem:E1}, \ref{lem:E2} and \ref{lem:E3} are applied with $\bar{\tau}_n=n^{\bar{\alpha}}$).

The proof of part (ii) is also almost identical to the proofs of Theorems \ref{thm:Sobolev-tau} and \ref{thm:Sobolev-alpha}, where we use
$$\Lambda_{n} := \Big\{(\at): \alpha\in[\ubar{\alpha}, \bar{\alpha} ], \tau\in\big[n^{-\frac1{2+p+2\alpha p}}, e^{2{\conr_0} \bar{c}_0 n \tilde{M}_{n}^{2} {(\minimax)}^2 (\log n)^{2\omega'}} \big]\Big\}.$$
On $\Lambda_n\setminus \Lambda_0$ we choose $u_n\asymp \exp({-4{\conr_0} \bar{c}_0 n \tilde{M}_{n}^{2} {(\minimax)}^2 (\log n)^{2\omega'}})$, in which case Lemmas \ref{lem:E1}, \ref{lem:E2} and \ref{lem:E3} are applied with $\bar{\tau}_n=\exp({2{\conr_0} \bar{c}_0 n \tilde{M}_{n}^{2} {(\minimax)}^2 (\log n)^{2\omega'}})$, while on $\Lambda_0$ we choose $u_n \asymp n^{-(3+2\bar{\alpha})}$. Condition (H1) is now based on Lemma \ref{lem:3.7} of the Supplement below.  
\end{proof}

\section{Outlook and computation}\label{sec:conc}
Combined, the three main theorems of this article indeed show that $\alpha$-regular $\tau$-scaled Laplace priors lead to adaptivity over all Besov spaces, spatially homogeneous \emph{and} spatially inhomogeneous. This is in contrast to Gaussian priors which, due to \cite[Theorem 4.1]{AW21}, can only adapt over spatially homogeneous Besov spaces. 

For this theoretical benefit of Laplace priors to be realized in practice, the efficient numerical implementation of the procedures studied in this work is crucial.
Naturally, this implementation is expected to pose some challenges and will be the topic of future work by the authors. The standard way of probing the joint posterior on $\theta,\lambda|X$ in the studied hierarchical setting, uses a (Metropolis within) Gibbs sampler which alternates between updating $\theta|X,\lambda$ and $\lambda|X,\theta$. In high dimensions, the parameters $\theta$ and $\lambda$ are strongly dependent under the assumed $\lambda$-dependent priors on $\theta$ and, as a result, the $\lambda$-chain in such Gibbs samplers mixes poorly. This has been studied for Gaussian priors with hyper-priors on scaling or regularity hyper-parameters in \cite{ABPS14} based on \cite{PRS07}; the intuition carries over to Laplace priors as well. To resolve this issue, \cite{ABPS14, PRS07} propose to re-parametrize the prior by writing $\theta=T(v,\lambda)$ for some appropriate transformation $T$, where $v$ and $\lambda$ are a priori independent with priors chosen so that $T(v,\lambda)$ has the desired prior distribution for $\theta$. For example, for an $\alpha$-regular $\tau$-scaled $p$-exponential prior with $\lambda=\tau$, we can write $\theta=\tau v$, where $v$ is $\alpha$-regular $p$-exponential with $\tau=1$ and $\tau$ has the originally postulated hyper-prior distribution. The non-centered parametrization of the initial Gibbs sampler, alternates between updating $v|X,\lambda$ and $\lambda|X,v$; samples for $\theta|X,\lambda$ can be obtained using $T(v,\lambda)$. While this method leads to robust algorithms with respect to dimension, for small noise the $\lambda$-chain again mixes poorly because $v$ and $\lambda$ are strongly a posteriori dependent (since they are constrained to satisfy $T(v,\lambda)\approx \theta_0$). As a result, when the discretization level is high \emph{and} the noise is small, both Gibbs sampler parametrizations have poor performance. A candidate alternative approach is the (pseudo) marginal algorithm which alternates between updating the  marginal $\lambda|X$ and $\theta|X, \lambda$, \cite{AR09}; see also \cite[Section 2.3]{ABPS14}.  

In addition, $p$-exponential priors with $p\neq2$ are not conjugate to the likelihood in the studied white noise model (conditionally on $\lambda$), hence one needs to employ Markov chain Monte Carlo (MCMC) methods for sampling the conditional distributions $\theta|X,\lambda$ or $v|X,\lambda$. Such methods will need to be used even with Gaussian priors, when implementing models for which Gaussian priors are not conjugate, such as density estimation \cite{MG22} and non-linear inverse problems \cite{GN20}. For Gaussian priors, MCMC algorithms which are well defined in function spaces and thus remain robust for high discretization levels, have been introduced in the last decade, see for example \cite{CRSW13, CLM16, BGLFS17}. The contribution \cite{CDPS18} discusses how to modify these algorithms to achieve dimension-robust  posterior sampling for non-Gaussian priors including $p$-exponential. The idea is to write $\theta=T(\xi,\lambda)$ where $\xi$ is a Gaussian white noise and $T$ is an appropriate transformation, which is such that $T(\xi,\lambda)$ has the desired prior for $\theta|\lambda$, and to use one of the known dimension-robust algorithms for Gaussian priors to sample the posterior $\xi|X, \lambda$. Samples for $\theta|X,\lambda$ can be obtained using $T(\xi,\lambda)$. In the same contribution, it is discussed how to combine this approach with a non-centered Gibbs sampler, to achieve dimension-robust sampling of the joint distribution of $\theta,\lambda|X$, when the prior on $\theta$ is non-Gaussian.

{For the implementation of the MMLE empirical Bayes approach, one can maximize the marginal likelihood over a fine grid of candidate values $\lambda\in \Lambda_n$. For the computation of the marginal likelihood $\bar{m}(X^{(n)}|\lambda)$, which for $p<2$ is not available explicitly, one can exploit the product structure in the assumed setting to express $\bar{m}$ as a product of univariate integrals, which can be approximated, for example, using quadrature. Note that, even for $p=2$, it is possible that $\bar{m}(X^{(n)}|\cdot)$ is multimodal with respect to $\lambda$, see \cite[Section 4.6]{SVZ13}.}

As a preliminary illustration, we provide two simulation experiments relating to our hierarchical Bayes results under Sobolev regularity (part (ii) of Theorems \ref{thm:Sobolev-tau} and \ref{thm:Sobolev-alpha}). In the first, we consider the setting studied in \cite[Section 3]{SVZ13}. In particular, we consider the white noise model \eqref{eq:wnm} for $n=200$ and with underlying truth $\theta_0\in L_2[0,1]$ which has coefficients $\theta_{0,\ell}=\ell^{-2.25}\sin(10\ell)$ with respect to the orthonormal basis $e_\ell(t)=\sqrt{2}\sin(\pi \ell t), \,t\in[0,1],\, \ell\in\N$. Such a function corresponds to regularity $\beta=1.75$ (more precisely it has Sobolev regularity $\beta$ for any $\beta<1.75$). We consider $\alpha$-regular and $\tau$-scaled Laplace priors defined via the basis $(e_\ell)$, for regularities $\alpha=\beta-1, \beta-1/2, \beta, \beta+1/2, \beta+1$ and $\lambda=\tau$ with hyper-prior an Inv-Gamma$(1,1)$ distribution, left-truncated at $n^{-1/(3+2\alpha)}$ (recall Assumption \eqref{thm1_1}). We truncate the relevant series expansions up to $L=200$, which, since for all choices of $\alpha$ the function $\theta$ is assumed to have Sobolev regularity larger than $0.5$ (under the prior, hence, by absolute continuity under the posterior as well), ensures that the approximation error induced by truncation is of lower order compared to the minimax estimation error over $H^\beta$. In Figure \ref{fig:tau} we present the resulting posterior means and 95\% credible regions (the latter computed by taking the 95\% out of the final 20000 iterations of the Gibbs sampler which are closest to the mean in $L^2$-sense). Evidently, the performance of the posterior means is comparable to the results in Figure 6 of \cite{SVZ13}, obtained using the corresponding Gaussian priors. The large width of the credible regions is attributed to the relatively large size of the noise (running the simulation with Gaussian priors shows that they give rise to credible regions of similar width).

\begin{figure}
    \centering
    \includegraphics[width=0.3\textwidth]{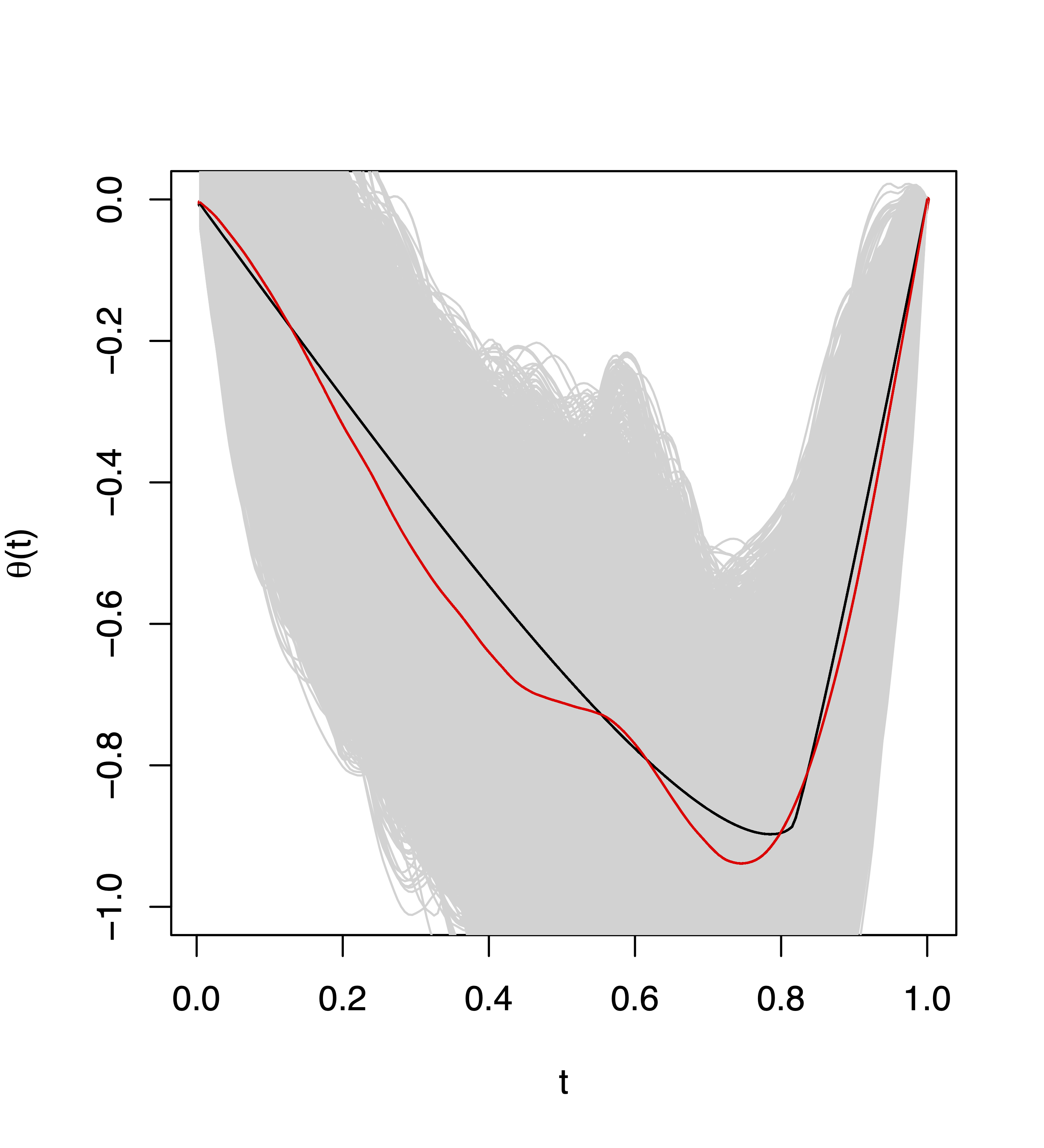}
     \includegraphics[width=0.3\textwidth]{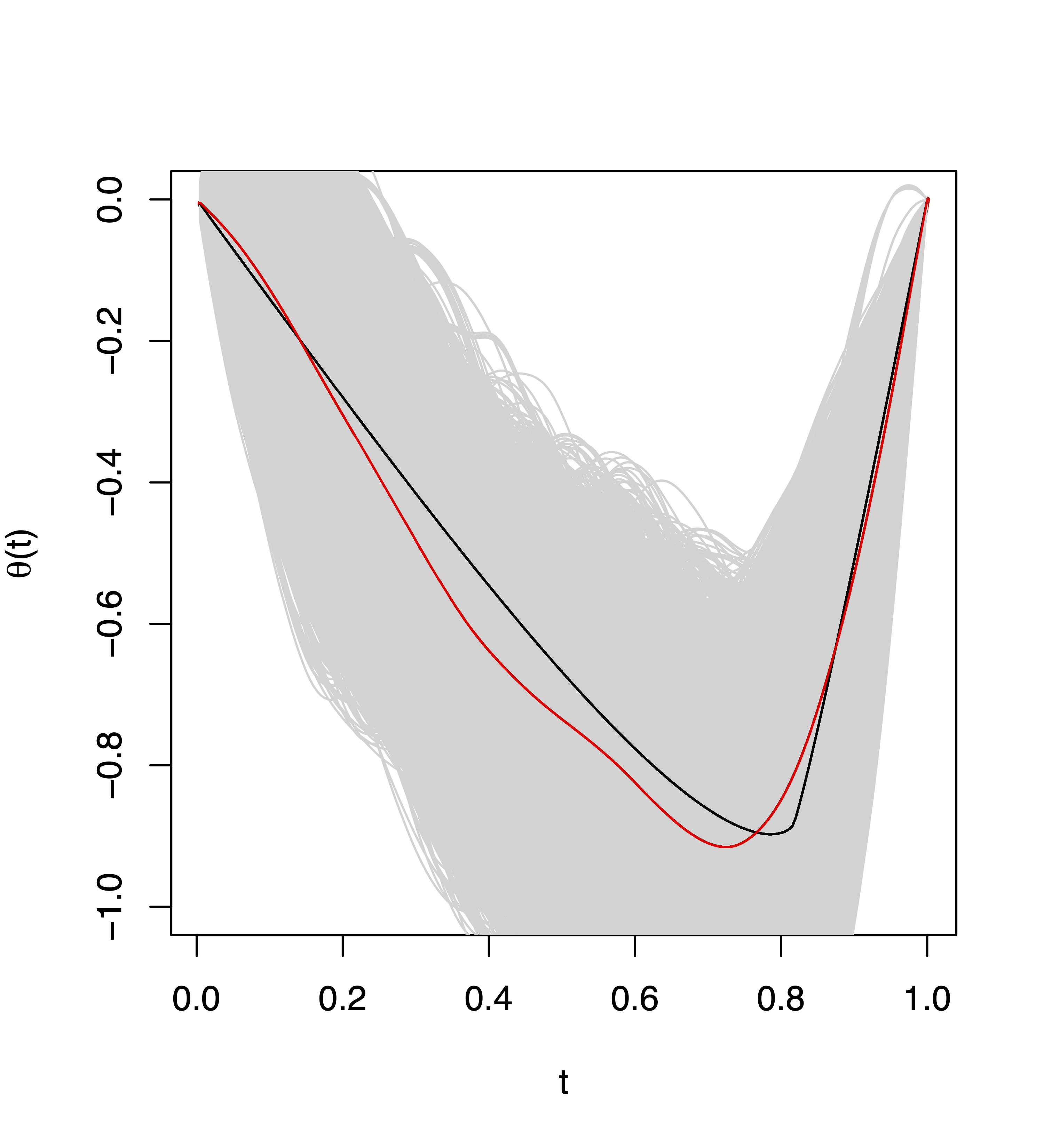}
    \includegraphics[width=0.3\textwidth]{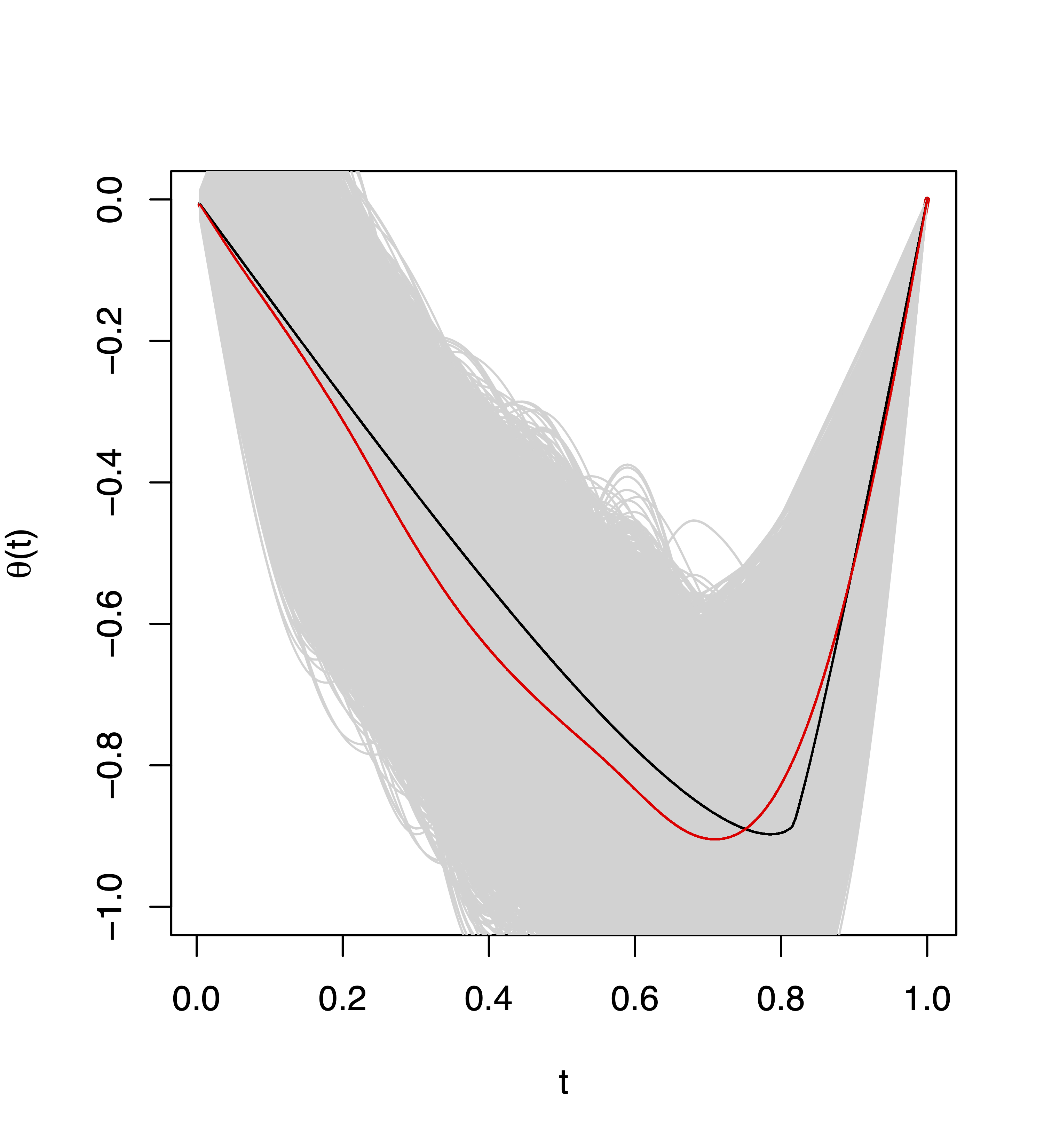}
    \includegraphics[width=0.3\textwidth]{./plots/tau-175.jpg}
    \includegraphics[width=0.3\textwidth]{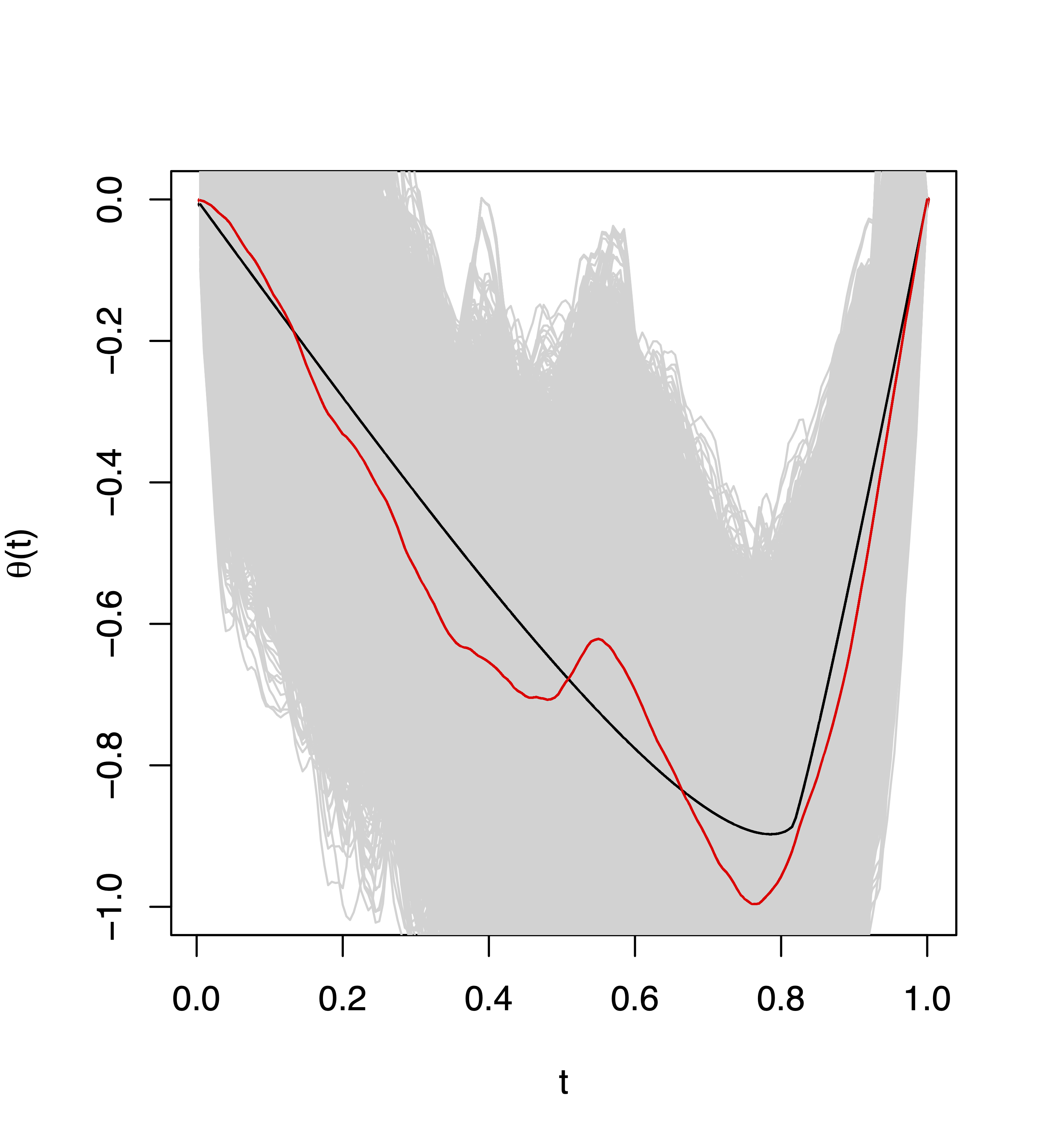}
    \includegraphics[width=0.3\textwidth]{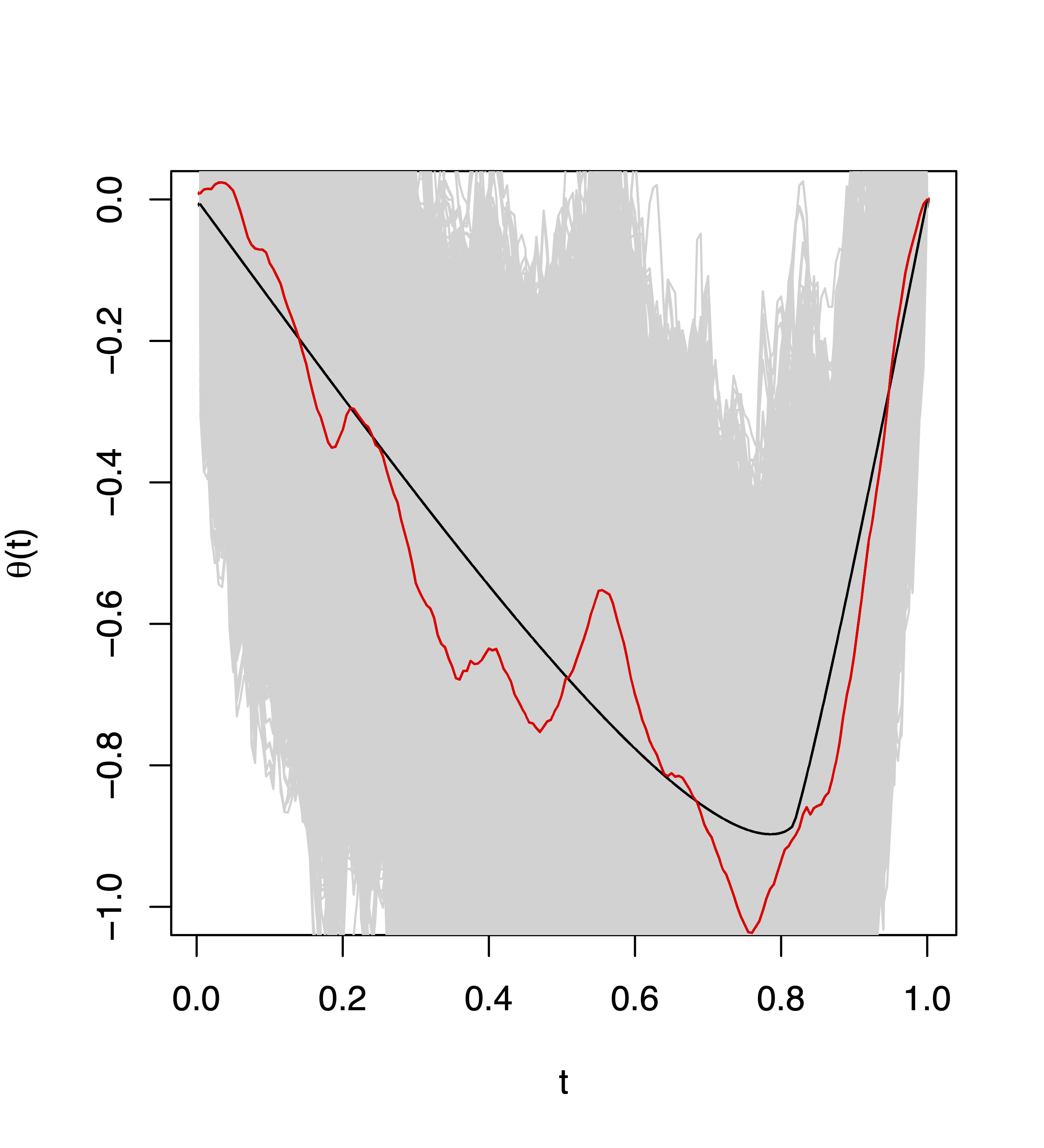}
  \caption{Top row left to right: truth (black), hierarchical Bayes posterior means (red) and 95\% credible regions (grey) for Laplace priors with regularities $\alpha=\beta, \beta+1/2, \beta+1$ and $\lambda=\tau$. Bottom row left to right: truth, posterior means and credible regions for Laplace priors with regularities $\alpha=\beta, \beta-1/2, \beta-1$ and $\lambda=\tau$.}

     \label{fig:tau}
\end{figure}

In the second experiment, we consider a similar setting to the one studied in \cite[Section 3]{KSVZ16} (but without inversion). In particular, as in the previous experiment we consider the white noise model, this time for $n=10^3$ and $n=10^5$, and with $\theta_{0,\ell}=\ell^{-3/2}\sin(\ell)$ with respect to the orthonormal basis $e_\ell(t)=\sqrt{2}\cos\big(\pi(\ell-1/2)t\big), \,t\in[0,1], \,\ell\in \N$, which corresponds to Sobolev regularity (almost) $\beta=1$. We use $\alpha$-regular and $\tau$-scaled Gaussian and Laplace priors defined on the basis $(e_\ell)$, with $\tau=1$ fixed, and $\lambda=\alpha$ with Exponential$(1)$ hyper-prior distribution, truncated to be supported in $[\ubar{\alpha}, \bar{\alpha}]=[0.5,100]$ (recall Assumption \eqref{thm1_2}). We truncate the infinite expansions up to $L\asymp n^{1/1.5}$ which again for $\theta$ with regularity at least $0.5$ ensures that the truncation error is of lower order compared to the minimax estimation rate. Figure \ref{fig:alpha} contains the resulting posterior means and 95\% credible regions; once more, the performance of the Laplace and Gaussian hierarchical priors is similar, with Laplace giving rise to slightly narrower credible regions.

\begin{figure}
    \centering
    \includegraphics[width=0.38\textwidth]{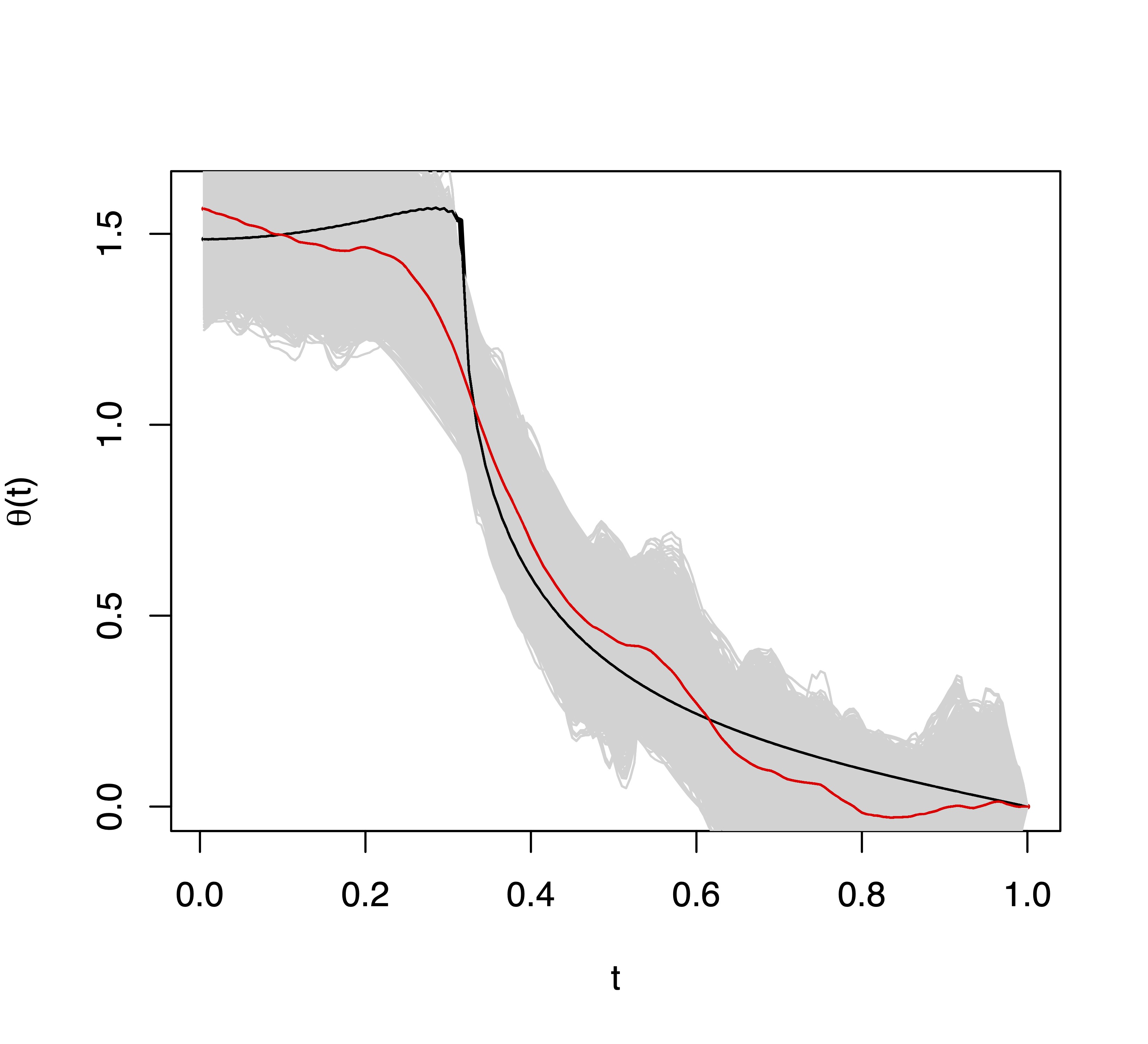}
     \includegraphics[width=0.38\textwidth]{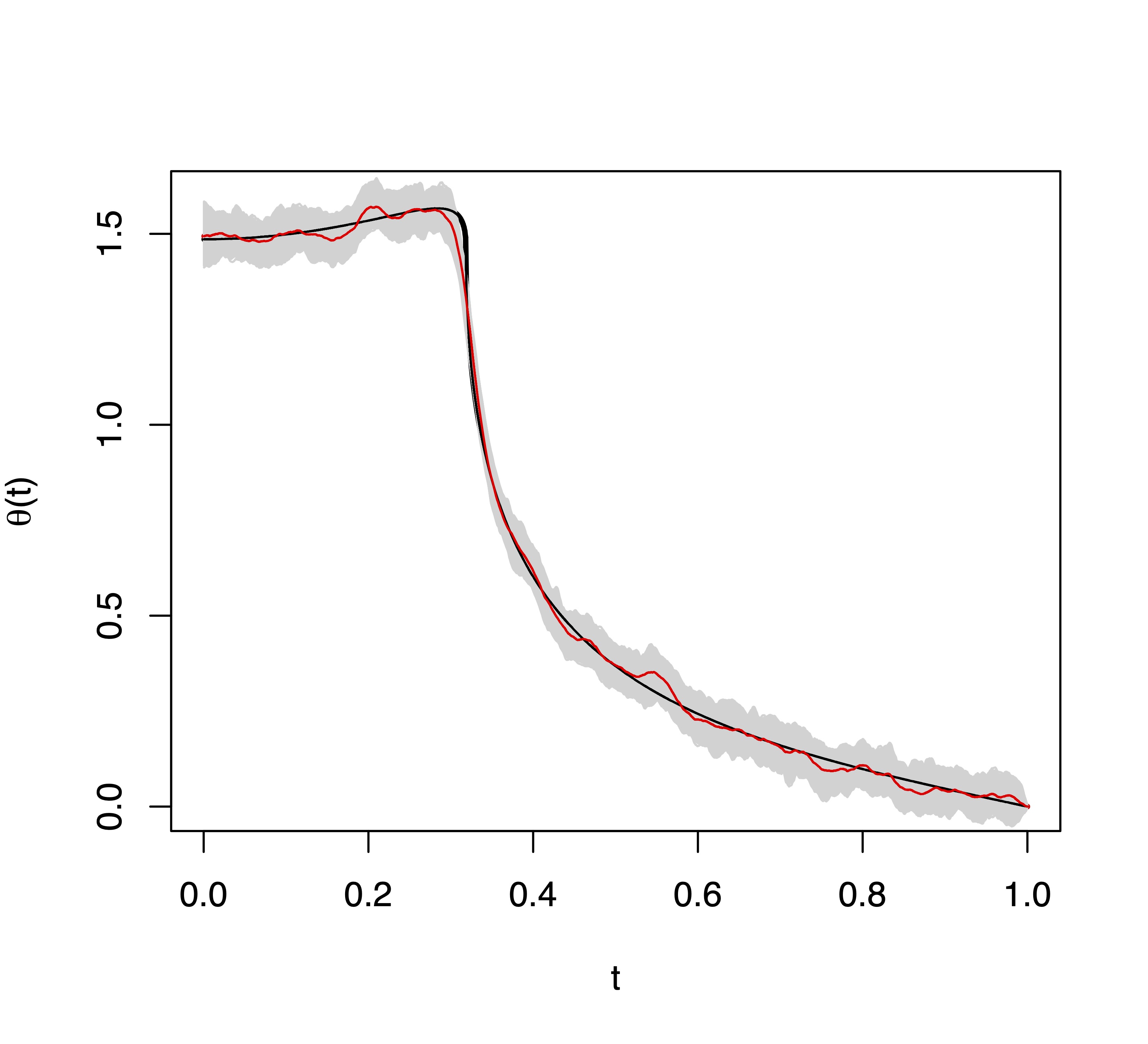}
    \includegraphics[width=0.38\textwidth]{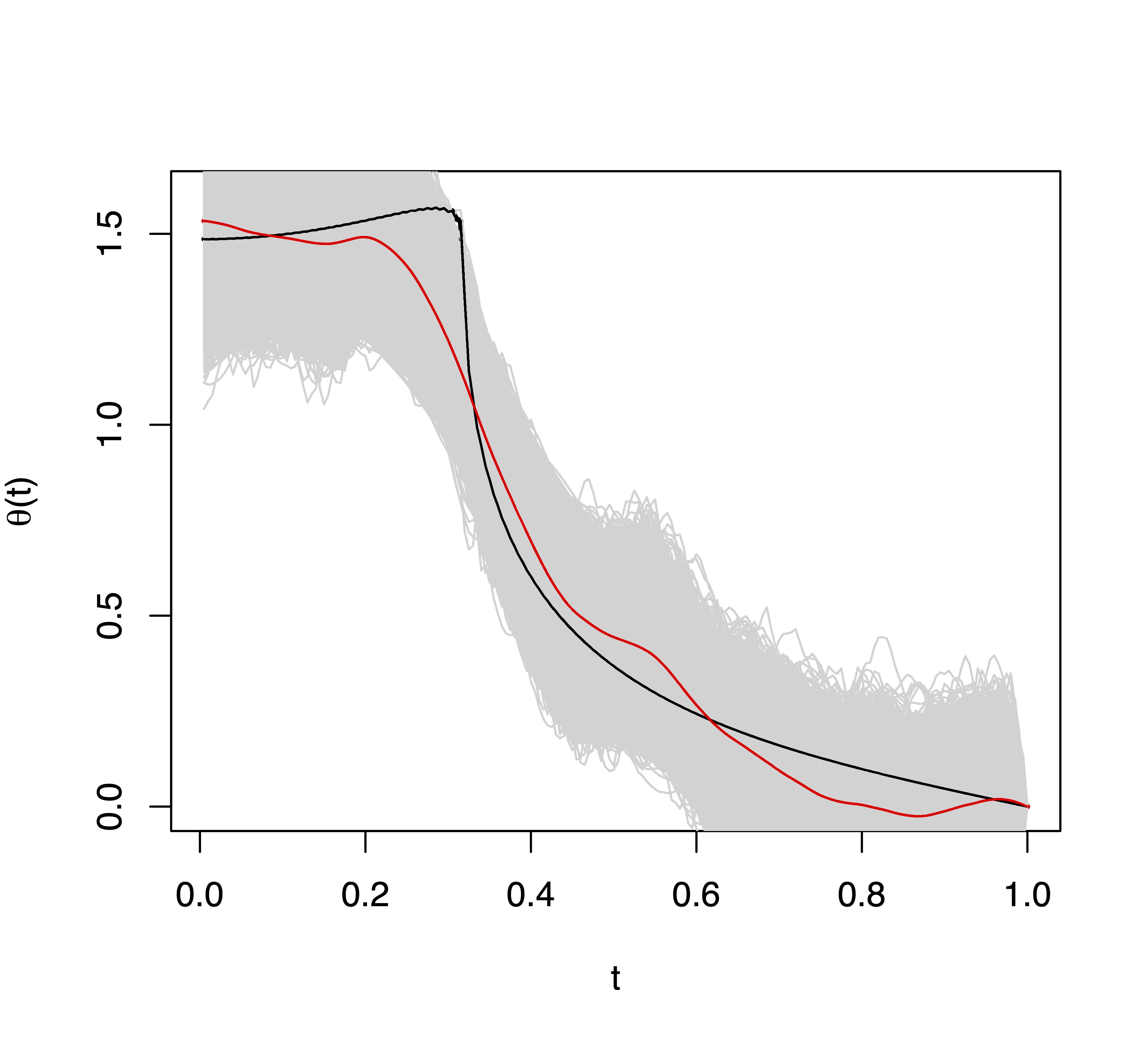}
    \includegraphics[width=0.38\textwidth]{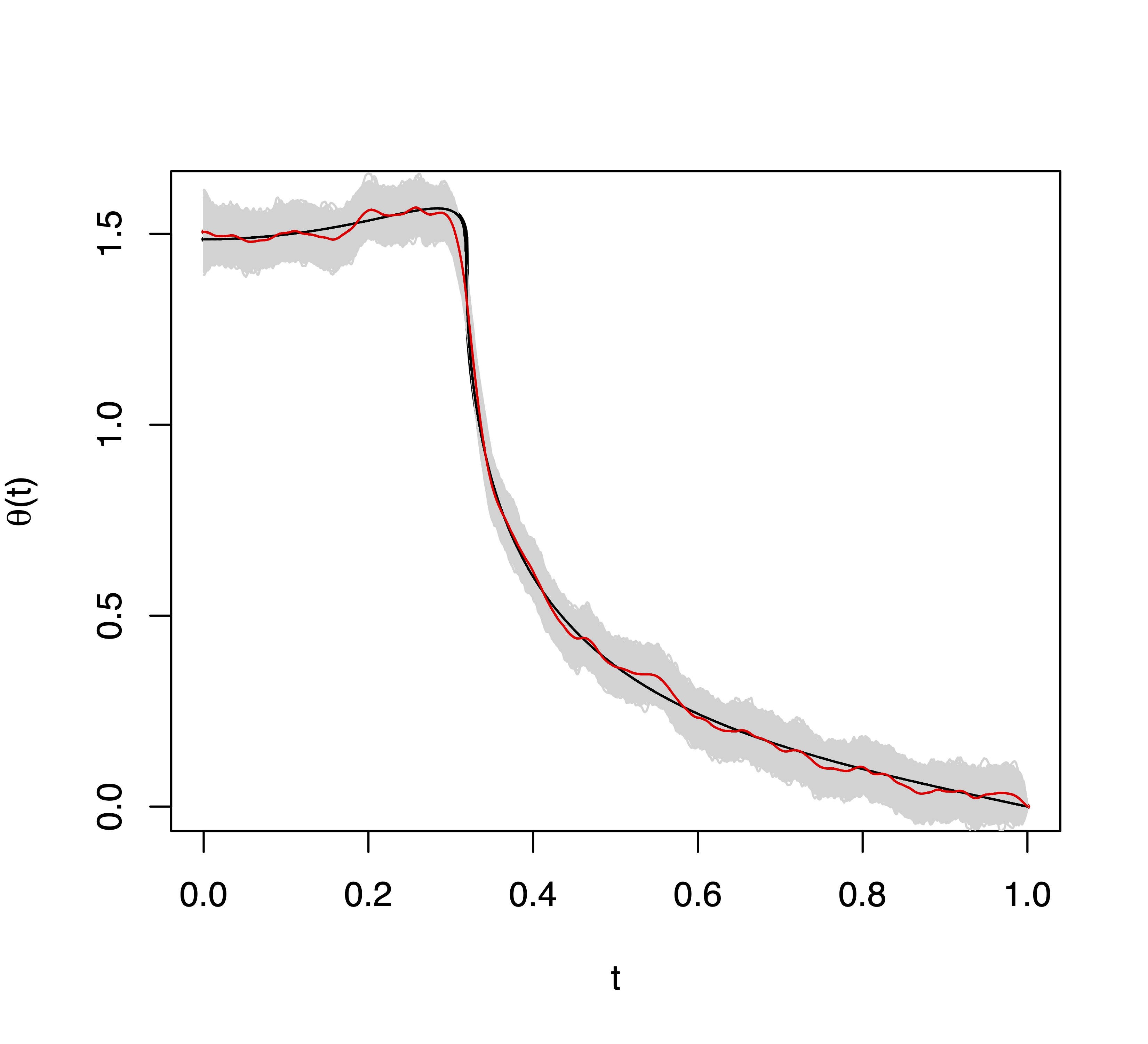}

   \caption{Top row: truth (black), hierarchical Bayes posterior means (red) and 95\% credible regions (grey) for Laplace priors with $\tau=1$ and $\lambda=\alpha$ for $n=10^3$ (left) and $n=10^5$ (right). Bottom row: truth, posterior means  and credible regions for Gaussian priors with regularities $\tau=1$ and $\lambda=\alpha$ for the corresponding noise levels.}
    \label{fig:alpha}
\end{figure}

The above experiments were implemented using the non-centered (whitened) preconditioned Crank-Nicolson within Gibbs algorithm \cite[Algorithm 4]{CDPS18}. All relevant codes are available at \url{https://bit.ly/3ns9kBA}.

%
%
%

\begin{acks}[Acknowledgments]
 The authors are deeply grateful to Botond Szabo for numerous explanations and many useful discussions regarding the general theory of \cite{RS17}. The authors are also grateful to Omiros Papaspiliopoulos for useful guidance regarding the implementation of the studied procedures. Finally, the authors thank two anonymous referees, the AE and the editor for many helpful comments.
 \end{acks}
%

\begin{supplement}
\stitle{}\vspace{-.4cm}
\sdescription{The supplement contains a number of technical results required for our proofs.}
\end{supplement}


\bibliographystyle{imsart-number} 
\bibliography{references}       

\bigskip\bigskip
\center{\large{\bf SUPPLEMENTARY MATERIAL}}\\\smallskip

We provide a number of technical results required for establishing the theory presented in the main article.

\renewcommand\thesection{\Alph{section}}
\setcounter{section}{0}

\section{The Fernique theorem for $\alpha$-regular $p$-exponential priors}\label{appn}
We formulate a result based on Christer Borell's generalization of Fernique's theorem for Gaussian measures, to the class of log-concave measures, see Theorem 3.1. in \cite{CB74}. Here, we consider the case of $\alpha$-regular $p$-exponential measures, which are known to be log-concave \cite[Proposition 2.3]{ADH21}, and track the dependence of the constants appearing in the cited theorem on the regularity parameter $\alpha$.

\begin{lem}\label{lem:fernique}
Let $\Pi_\alpha\coloneqq\Pi(\cdot\mid\alpha, \tau=1)$ be an $\alpha$-regular  $p$-exponential measure. Fix $\ubar{\alpha}>0$ and assume that $\mu$ is any weighted-$\ell_r$ norm, $r\ge 1$, which is finite almost surely with respect to $\Pi_\alpha$, $\forall \alpha \geq \ubar{\alpha}$. Then, there exist $s_1, s_2 > 0$ (depending on $\ubar{\alpha}$ and $\mu$), such that 
\begin{equation} \label{Fernique_1}
E_{\Pi_\alpha} \big[ e^{s_2 \mu(\theta)} \big] \leq s_1, \quad \forall \alpha \geq \ubar{\alpha}.
\end{equation}
\end{lem} 

\begin{proof}
The proof follows the proof of \cite[Theorem 3.1]{CB74}, with the difference that the cited result considers a single logarithmically-concave measure, while we consider the family of $\alpha$-regular $p$-exponential priors with $\alpha\ge\ubar{\alpha}$ simultaneously. 

Since $\mu$ is finite almost surely with respect to $\Pi_{\ubar{\alpha}}$ and since $\Pi_{\ubar{\alpha}}$ is non-degenerate, there exists a $\lambda>0$ such that 
$$\Pi_{\ubar{\alpha}}\big( \{\theta: \mu(\theta)<\lambda\} \big)\in(1/2,1).$$ 
Set $$\gamma:=\Pi_{\ubar{\alpha}}\big( \{\theta: \mu(\theta)<\lambda\} \big).$$
By the definition of $\Pi_\alpha$ and since $\mu$ is a weighted-$\ell_r$ norm with $r\ge1$, it is straightforward to see that for any $s>0$ it holds
 $$\Pi_{{\alpha}}\big( \{\theta: \mu(\theta)<s\} \big)\ge \Pi_{\ubar{\alpha}}\big( \{\theta: \mu(\theta)<s\} \big), \quad \forall \alpha\ge\ubar{\alpha}.$$ 
Let $A\coloneqq \{\theta:\mu(\theta)< \lambda \}$. Combining with \cite[Lemma 3.1]{CB74}, we thus get that  for all $t\ge 1$
\begin{equation} \label{Fernique_2} 
\Pi_{\alpha} \big( (tA)^c \big) = \Pi_{\alpha} \big( \{\theta:\mu(\theta) \geq \lambda t \} \big)\leq \Pi_{\ubar{\alpha}} \big( \{\theta:\mu(\theta) \geq \lambda t \} \big) \leq \gamma \Big( \frac{1-\gamma}{\gamma} \Big)^{\frac{1+t}{2}},
\end{equation}
where $0<(1-\gamma)/\gamma <1$. 

For any $s_2>0$, using \cite[Exercise 2.2.7.]{Durrett} and a change of variables, we get 
\begin{equation} \begin{split}
E_{\Pi_{\alpha}} &\big[ e^{s_2 \mu(\theta)} \big] =  \int_{0}^{\infty}  s_2 e^{s_2 \lambda \psi } \Pi_{\alpha} \big( \{\theta:\mu(\theta) \geq \lambda \psi  \} \big) \lambda d\psi \\
&=s_2\lambda\Big( \int_{0}^{1}   e^{s_2 \lambda \psi } \Pi_{\alpha} \big( \{\theta: \mu(\theta) \geq \lambda \psi  \} \big)  d\psi + \int_{1}^{\infty}  e^{s_2 \lambda \psi } \Pi_{\alpha} \big( \{\theta:\mu(\theta) \geq \lambda \psi  \} \big) d\psi\Big).
\end{split} \end{equation}
Let us study each of the two integrals on the above right hand side separately. 

For the first integral, we have that
$$\int_{0}^{1}  e^{s_2 \lambda\psi} \Pi_{\alpha} \big( \{\theta:\mu(\theta) \geq \lambda\psi\} \big)d\psi \leq \int_{0}^{1}  e^{s_2  \lambda\psi}  d\psi  =\frac{e^{s_2 \lambda}-1}{s_2\lambda}.$$

For the second integral, using \eqref{Fernique_2} we have
\begin{equation*} \begin{split}
\int_{1}^{\infty}  e^{s_2 \lambda \psi } \Pi_{\alpha} \big( \{\theta:\mu(\theta) \geq \lambda \psi  \} \big) d\psi  & \leq \int_{1}^{\infty}   e^{s_2 \lambda \psi } \gamma \Big( \frac{1-\gamma}{\gamma} \Big)^{\frac{1+\psi}{2}}  d\psi \\
& = \gamma  \int_{1}^{\infty} e^{s_2 \lambda \psi  + \frac{1+\psi}{2} \log \big(\frac{1-\gamma}{\gamma}\big)} d\psi \\
& =  \gamma  e^{-\frac{a}{2}}  \int_{1}^{\infty} e^{\psi \big(s_2 \lambda - \frac{a}{2}\big)} d\psi,
\end{split} \end{equation*}
where $a = \log \big(\gamma / (1-\gamma) \big)>0$. Choosing $s_2 < a / 2 \lambda,$ we have
\begin{equation*} \begin{split}
\int_{1}^{\infty}  e^{s_2 \lambda \psi } \Pi_{\alpha} \big( \{\theta:\mu(\theta) \geq \lambda \psi \} \big)  d\psi  \leq \frac{\gamma   e^{-\frac{a}{2}}}{\frac{a}{2}-s_2 \lambda}.
\end{split} \end{equation*}

Combining we have that for $s_2 < a / 2 \lambda,$ it holds
\[E_{\Pi_{\alpha}} \big[ e^{s_2 \mu(\theta)} \big]\leq\frac{s_2\lambda \gamma e^{-\frac{a}2}}{\frac{a}2-s_2\lambda}+e^{s_2\lambda}-1,\quad \forall \alpha\ge\ubar{\alpha},\]
which completes the proof.
\end{proof}

\section{Existence of $\gep_n(\lambda)$ solving \eqref{eq:en}}

\begin{lem}\label{lem:existenceenl}
Let $\Pi(\cdot \mid \at)$ be an $\alpha$-regular $\tau$-scaled $p$-exponential prior in $\Theta$, for some $\at>0$. Let $\theta_{0} \in \Theta$ and fix $K>0$. Then, $\forall n$ there exists a unique $\gep_n=\gep_{n}(\at;K,\theta_0)$, such that
\begin{equation}
\Pi \big( \theta : \norm{\theta - \theta_{0}}_2 \leq K\gep_{n} \mid \at \big) = e^{-n \gep_{n}^{2}}.
\end{equation}
Furthermore, given $\at>0$, the sequence $(\gep_{n})$ is strictly decreasing to zero.
\end{lem}

\begin{proof}
We first verify that $r(\gep) \coloneqq\Pi(\theta: \norm{\theta-\theta_{0}}_2 \leq K\gep \mid \at)$ is a continuous function of $\gep>0$. Indeed this can be checked by verifying left and right continuity using the continuity property of probability measures, where for left continuity one needs to use that  $\Pi(\theta: \norm{\theta-\theta_{0}}_2 = \gep \mid \at) = 0, \forall \gep>0$, implied by  \cite[Lemma 6.1]{ABDH18}. 

The function $r(\cdot)$ is obviously non-decreasing with $\gep$, with $\lim_{\gep \searrow 0} r(\gep)=0$ and $\lim_{\gep \nearrow \infty} r(\gep)=1$. 
In fact, $r(\cdot)$ is strictly increasing. Indeed, if not, then there would exist $\gep_1 > \gep_2$ such that $r(\gep_1) = r(\gep_2)$. In particular, there would exist a $\theta_1\in \Theta$ and a small enough radius $\delta >0$, such that $\Pi (\theta: \norm{\theta-\theta_1}_2 < \delta \mid \at)=0$. This would be a contradiction to the fact that $\text{supp}\big(\Pi(\cdot \mid \at)\big)=\Theta$, as implied by \cite[Proposition 2.10]{ADH21}.

The above considerations imply the claims in the statement, since $E_n(\gep)\coloneqq e^{-n\gep^2}$ is decreasing in both $\gep$ and $n$ with $\lim_{\gep \searrow 0} E_{n}(\gep)=1$ and $\lim_{\gep \nearrow \infty} E_{n}(\gep)=0.$
\end{proof}


\section{Technical lemmas for verifying the conditions of Corollary \ref{corRS17}}\label{sec:lemmas}
This section contains a number of technical lemmas used for verifying the assumptions of Corollary \ref{corRS17}, in order to prove our main theorems. 

The next four lemmas combined, allow the verification of conditions (A1), (A2), (B1) and (C1)-(C3). 
\begin{rem}\label{rem:taugep}In all of these lemmas, it is silently assumed that for some fixed $\tilde{C}>0$, $\gep_n=\gep_n(\alpha,\tau)$ from \eqref{eq:en} satisfies $\gep_n\leq \tilde{C}\tau$; this restriction stems from \eqref{lem2.2_3}. In particular in the proofs of our main results, we will employ these lemmas for $\tau_n\in \Lambda_n$, where $\Lambda_n:=[\ubar{\tau}_n, \t_n]$ and $\t_n>\ubar{\tau}_n>0$, hence we will need to ensure that $\gep_n\lesssim \tau_n$. We do not spell this out explicitly in the lemmas, in order to avoid an overload of technical statements.
\end{rem}


\begin{lem}\label{lem:C1}
Fix $\ubar{\alpha}$ and $s$ such that $\ubar{\alpha}>s>0$. For any $\alpha\geq\ubar{\alpha}, \tau> 0$ and $\zeta \in (0,1)$, take $\eta \geq \tilde{c}_1^{2} \big(4\zeta^{-1}\big)^{1/\alpha},$ where $\tilde{c}_1=\tilde{c}_1(\alpha)$ is as in \eqref{lem2.2_3} and define the sets 
$$\tth(\alpha, \tau)\coloneqq \frac{\zeta K}{4} \gep_{n} B_{\ell_2} + R_{n}^{\frac{p}{2}}B_{\sht} + R_{n}B_{\omat},$$
where $\gep_{n}=\gep_{n}(\alpha, \tau)$ is as in \eqref{eq:en} and 
$$R_{n} = R_n(\at)\coloneqq \big( 2\eta\tilde{K} n\gep_{n}^{2} \big)^{\frac{1}{p}}, \quad p \in [1,2].$$
Here, $\tilde{K}>0$ is a fixed constant depending only on $p$, defined in the concentration inequality \cite[Proposition 2.15]{ADH21}.

Let  $s_1=s_1(\ubar{\alpha};s)>0, s_2=s_2(\ubar{\alpha};s)>0$ be  as in Lemma \ref{lem:fernique}, for $\mu(\cdot)=\norm{\cdot}_{{H^s}}$. For $Y_n=\eta \tau n \gep_n^2/s_2$, define the sets 
\begin{equation}\label{eq:sieve}\pth(\at)\coloneqq \tth(\at)\cap Y_nB_{{H^s}},\quad\alpha\geq \ubar{\alpha}, \tau>0.\end{equation}
Then, 
\begin{equation}\label{eq:pthbnd}\Pi \big( \pth^{c}(\alpha, \tau) \mid \alpha,\tau \big) \leq 2s_1e^{-\eta n \gep_{n}^{2}}\end{equation}
and there exists a constant $M'>0$ (independent of $\alpha$, $\tau$ and $K$) such that
$$\log N \big(K\zeta \gep_{n}, \pth(\alpha, \tau), \norm{\cdot}_{2} \big) \leq M'\eta^p\tilde{c}_1^2 n \gep_{n}^2.$$
\end{lem}

\begin{proof}
We use ideas from the proof of \cite[Theorem 3.1]{ADH21}. 

For the first claim, by \cite[Proposition 2.15]{ADH21} (here we denote by $\tilde{K}$ the constant $K$ in the notation of that source) we have
\begin{align*}
\Pi \big(\tth^{c}(\alpha, \tau) \mid \alpha, \tau \big) &\leq \frac{1}{\Pi \big( \frac{\zeta K}{4} \gep_{n} B_{\ell_2} \mid \at\big)} e^{-R_{n}^{p}/\tilde{K}}
= e^{\varphi_{0}(\frac{\zeta K}{4} \gep_{n})} e^{-2\eta n\gep_{n}^{2}}.
\end{align*}
Since $\eta\geq  \tilde{c}_1^{2} (4\zeta^{-1})^{1/\alpha},$ and by combining \eqref{lem2.2_3} and \eqref{lem2.2_4} (see also Remark \ref{rem:taugep}), it holds
\begin{align*}
\varphi_{0} \Big(\frac{\zeta K}{4} \gep_{n} \Big)-\eta n\gep_n^2\leq \varphi_{0} \Big(\frac{\zeta K}{4} \gep_{n} \Big)- \tilde{c}_1^{2} \big(4\zeta^{-1}\big)^{1/\alpha} n\gep_{n}^{2} 
\leq0,
\end{align*}
so that 
\[
\Pi \big(\tth^{c}(\alpha, \tau) \mid \alpha, \tau \big) \leq e^{-\eta n\gep_n^2}.
\]
By the definition of $\pth(\at)$, it thus remains to show that 
\[\Pi \big( (Y_nB_{{H^s}})^c|\at \big) \leq s_1e^{-\eta n\gep_n^2}.\] 
Indeed, using the (exponential) Markov inequality, for any $c, r>0$ we have that 
\[
\Pi(\norm{\theta}_{{H^s}}\geq r\mid\at)=\Pi(\norm{\theta}_{{H^s}}\geq r/{\tau}\mid \alpha,\tau=1)\leq \frac{E_{\alpha,1}[e^{c\norm{\theta}_{{H^s}}}]}{e^{c\frac{r}\tau}},\]
where $E_{\alpha,1}$ denotes expectation with respect to $\Pi(\cdot\mid \alpha,\tau=1)$. Using Fernique's theorem, see Lemma \ref{lem:fernique}, and by \cite[Lemma 5.2]{ADH21} which implies that the ${H^s}$-norm is almost surely finite with respect to $\Pi(\cdot\mid\at)$ for $s<\alpha$, we get that there exist $s_1=s_1(\ubar{\alpha};s)>0, s_2=s_2(\ubar{\alpha};s)>0$ such that
\[
\Pi(\norm{\theta}_{{H^s}}\geq r\mid\at)\leq s_1e^{-s_2\frac{r}\tau} \quad\forall \alpha\geq\ubar{\alpha}, \quad \tau>0.\]
Choosing $r=Y_n$, completes the proof.

For the second claim, note that since $\pth(\at)\subset\tth(\at)$ it suffices to establish the same bound for covering $\tth(\at)$ instead of $\pth(\at)$. We first show that $\forall \gep, R>0$ 
\begin{equation}\label{eq:qzapprox}
R^{\frac{p}{2}} B_{\sht} \subseteq \gep B_{\ell_2} + a B_{\omat},
\end{equation}
 where \[a=a(\gep, R) = R^{\frac{p}{2}+\frac{2-p}{2(1+2\alpha)}} \tau^{\frac{2-p}{p(1+2\alpha)}} \gep^{\frac{p-2}{p(1+2\alpha)}}.\] Indeed, let $w \in R^{\frac{p}2}B_{\sht}$, that is 
$$\sum_{\ell=1}^{\infty} \ell^{1+2\alpha}  w_{\ell}^{2} \leq \tau^{2} R^{p}.$$
Consider $h_{L} = (w_{1}, \cdots, w_{L}, 0, \cdots, 0)$, where $L$ is sufficiently large so that $\norm{h_{L}-w}_2 \leq \gep$. It holds
\begin{equation} \label{lemC.1_1}
\norm{h_{L}-w}_2^{2} = \sum_{\ell=L+1}^{\infty} w_{\ell}^{2} = \sum_{\ell=L+1}^{\infty} \ell^{1+2\alpha} \ell^{-1-2\alpha} w_{\ell}^{2}  \leq L^{-1-2\alpha}\tau^{2}R^{p}, 
\end{equation}
so that it suffices to take $L=\Big(R^p \tau^2 / \gep^2\Big)^{\frac{1}{1+2\alpha}}.$
Using H\"older inequality with $\big(2/p, \,2/(2-p) \big)$, we then get that
\begin{align*}
\norm{h_L}^{p}_{\omat} &= \tau^{-p} \sum_{\ell=1}^{L} \ell^{\frac{p}{2} +\alpha p} |w_{\ell}|^{p} \leq \tau^{-p} \bigg( \sum_{\ell=1}^{L} \ell^{1 + 2\alpha} |w_{\ell}|^{2} \bigg)^{\frac p2}  L^{\frac{2-p}{2}}\\
& \leq \tau^{-p} \big(\tau^2 R^p \big)^{\frac p2} L^{\frac{2-p}{2}} = R^{\frac{p^2}{2}} \bigg( \frac{R^p \tau^2}{\gep^2} \bigg)^{\frac{2-p}{2(1+2\alpha)}} = R^{\frac{p^2}{2} + {\frac{p(2-p)}{2(1+2\alpha)}}} \tau^{\frac{2-p}{1+2\alpha}} \gep^{\frac{p-2}{1+2\alpha}}.
\end{align*}
The inclusion \eqref{eq:qzapprox} thus holds.

We thus obtain that 
$$\tth(\alpha, \tau) \subseteq \frac{\zeta K}{2} \gep_{n} B_{\ell_2} + \bigg( a\Big(\frac{\zeta K}{4}\gep_{n}, R_{n} \Big) + R_{n} \bigg) B_{\omat},$$
where using the definition of $R_n$ and \eqref{lem2.2_5} (which we solve for $\tau$), we can bound
\begin{align*}
a\Big(\frac{\zeta K}{4}\gep_{n}, R_{n} \Big) &= R_{n}^{\frac{p}{2} + {\frac{2-p}{2(1+2\alpha)}}} \tau^{\frac{2-p}{p(1+2\alpha)}} \gep_{n}^{\frac{p-2}{p(1+2\alpha)}} \Big( \frac{\zeta K}{4} \Big)^{\frac{p-2}{p(1+2\alpha)}} \\
&\leq \big( 2\eta \tilde{K}n\gep_n^2 \big)^{\frac12+\frac{2-p}{2p(1+2\alpha)}} \big(\gep_n n^{\frac{\alpha}{1+2\alpha}} \big)^{\frac{2-p}p} \big(\tilde{c}_1K^\frac1\alpha \big)^\frac{\alpha(2-p)}{p(1+2\alpha)}\gep_n^\frac{p-2}{p(1+2\alpha)}\Big( \frac{\zeta K}{4} \Big)^{\frac{p-2}{p(1+2\alpha)}}\\
&\leq \big(2\eta\tilde{K}\big)^{\frac12+\frac{2-p}{2p(1+2\alpha)}}\tilde{c}_1^{\frac{\alpha(2-p)}{p(1+2\alpha)}}(4/\zeta)^{\frac{2-p}{p(1+2\alpha)}} \big(n\gep_n^2 \big)^\frac1p.
\end{align*}
Noting that $\eta\ge1$ and taking into account that $p\in[1,2]$ and $\alpha>0$ (which implies that the exponents $\frac{\alpha(2-p)}{p(1+2\alpha)}, \frac{2-p}{p(1+2\alpha)}$ are bounded above by 1, uniformly in $\alpha$), we can further bound
\[a\Big(\frac{\zeta K}{4}\gep_{n}, R_{n} \Big)\leq \eta\tilde{c}_1 M_1 \big(n\gep_n^2 \big)^{\frac1p},\]
where the constant $M_1$ does not depend on $\alpha, \tau, K$.
We thus get that 
\begin{equation} \label{lemC.1_2}
\tth(\alpha, \tau) \subseteq \frac{\zeta K}{2} \gep_{n} B_{\ell_2} + \eta \tilde{c}_1M_2 \big(n\gep_n^2 \big)^{\frac1p} B_{\omat},
\end{equation}
for some constant $M_2$ again independent of  $\alpha,\tau, K$, where we used that $\tilde{c}_1\geq1$ and $\eta^{\frac1p}\leq \eta$ since $\eta\geq1$.

Now, let $h_{1}, \cdots, h_{N} \in \eta \tilde{c}_1M_2 \big(n\gep_n^2 \big)^{\frac1p} B_{\omat}$ be $\zeta K  \gep_{n} /2 $-apart in $\norm{\cdot}_2$. Then, the balls $h_j + \frac{\zeta K }{4} \gep_{n} B_{\ell_2}$ are disjoint, and hence by \cite[Proposition 2.11]{ADH21} we obtain
\begin{equation} \label{lemC.1_3}
1 \geq \sum_{j=1}^{N} \Pi \Big( h_j + \frac{\zeta K}{4} \gep_{n} B_{\ell_2} |\at\Big) \geq \sum_{j=1}^{N} e^{-\frac{\norm{h_j}_{\omat}^{p}}{p}} \Pi \Big( \frac{\zeta K}{4} \gep_{n} B_{\ell_2} |\at\Big) \geq N e^{-\frac{\eta^p\tilde{c}_1^p M_2^p n\gep_n^2}{p} -\varphi_{0} \big(\frac{\zeta K}{4} \gep_{n} \big)}.
\end{equation}
If the set of points $h_{1}, \cdots, h_{N}$ is maximal in $\eta \tilde{c}_1M_2 \big(n\gep_n^2 \big)^{\frac1p} B_{\omat}$, then the balls $h_j + \frac{\zeta K}{2} \gep_{n} B_{\ell_2}$ cover $\eta\tilde{c}_1M_2 \big(n\gep_n^2 \big)^{\frac1p} B_{\omat}$, and we get that
\begin{equation} \label{lemC.1_4}
\tth(\alpha, \tau) \subset \bigcup_{j=1}^{N} (h_j + \zeta K \gep_{n} B_{\ell_2}).
\end{equation}
Combining \eqref{lemC.1_3} and \eqref{lemC.1_4} we obtain 
\[
N \big(\zeta K\gep_{n}, \tth(\alpha, \tau), \norm{\cdot}_2 \big) \leq N \leq e^{\frac{\eta^p \tilde{c}_1^p M_2^p n\gep_n^2}{p} + \varphi_{0}(\frac{\zeta K}{4} \gep_{n})},\]
hence, using  \eqref{lem2.2_3} and \eqref{lem2.2_4} as for the first claim, we get \[\log N \big(\zeta K \gep_{n}, \tth(\alpha, \tau), \norm{\cdot}_2 \big) \leq \frac{\eta^p \tilde{c}_1^pM_2^p n\gep_n^2}{p} + \tilde{c}_1^2\big(\frac4{\zeta}\Big)^{\frac1\alpha}n \gep_{n}^{2}.\]
Taking into account that $\alpha\geq\ubar{\alpha}$, the second claim follows.
\end{proof}

\begin{rem}\label{rem:eta}
A careful examination of the sharp results of \cite{FA07}, shows that for all $p\in[1,2]$, $\tilde{c}_1$ from \eqref{lem2.2_3}  blows-up linearly as $\alpha\to\infty$. In particular, restricting $\alpha\in[\ubar{\alpha},\bar{\alpha}]$ for any $\bar{\alpha}>\ubar{\alpha}>0$, enables a uniform choice of $\tilde{c}_1$, hence also a uniform choice of $\eta$ in the last result. We study this in Lemma \ref{lem:centeredconstants} below.
\end{rem}

\begin{rem}\label{rem:intersect}
The set $\tth(\lambda)$ is a natural candidate for the sieve set due to Talagrand's inequality \eqref{eq:tal}. Here, we  intersect $\tth(\lambda)$ with a ball in a Sobolev space $H^s$ for $s>0$ which can be arbitrarily small, with a suitably growing radius as $n\to\infty$. The additional regularity infused on the sieve set $\pth$, will allow us to uniformly control small perturbations of its elements, resulting in a uniform control of the data-dependence of the hyper-parameter (see the proofs of Lemmas \ref{lem:E1} and \ref{lem:E3} below). This technique is inspired by the recent literature on rates of contraction for PDE inverse problems, see for example \cite{N20, GN20, AW21}, where the sieve set is intersected with a ball of sufficiently large (but fixed) radius in some suitably regular space (there the motivation is different: it allows the exploitation on the one hand of local-Lipschitz assumptions on the forward operator to get rates of contraction for the forward problem and on the other hand of stability estimates to pass from "forward" to "inverse" rates).
\end{rem}


\begin{lem}\label{lem:E1}
Let $\t_n\to\infty$, and $s,\ubar{\alpha}, \bar{\alpha}$ such that $\bar{\alpha}>\ubar{\alpha}>s>0$. Recall $\pth(\at)$ from Lemma \ref{lem:C1}. Then for $u_n\lesssim n^{-5/2}\t_n^{-1}$, we have 
\begin{equation}\label{eq:E1}
\sup_{\substack{\ubar{\alpha}\leq\alpha\leq\bar{\alpha} \\\ 0<\tau<\t_n}}\;\sup_{\theta\in \pth(\at)} Q^\theta_{\lambda,n}(\mX)=O(1).
\end{equation}
The bound holds for $\lambda=\tau$ or $\alpha$ or $(\at)$ and the loss $\rho$ in the definition of $Q^\theta_{\lambda,n}$ is as in \eqref{eq:rhodefn}.
\end{lem}

\begin{proof}
The proof follows the lines of the proof of \cite[Lemma E.1]{RS17} which considers $\theta\in \R^n$ in the non-parametric regression setting. 
Apart from considering the white noise model, the main difference here is that we need to additionally control the tail of the sequence $\theta\in\ell_2$, and to this end we restrict the inner supremum on the set $\pth$, which in particular implies that $\theta$ has a certain minimum regularity, $\theta\in Y_nB_{H^s}$. We crucially exploit this regularity. Furthermore, we need to study the case $\lambda=(\at)$. We in fact only study that case, while the cases $\lambda=\tau$ or $\alpha$ follow with obvious modifications.

Assume $\lambda=(\alpha,\tau)$, and recall the loss function $\rho(\lambda, \lambda') = |\log \tau - \log \tau'| + | \alpha - \alpha' |$ from \eqref{eq:rhodefn}. We introduce the notation
$$\underline{\psi}_{\lambda, j}^{\theta} \coloneqq \inf_{\rho(\lambda, \lambda') \leq u_n} \psi_{\lambda, \lambda'} (\theta_{j}),$$
$$\overline{\psi}_{\lambda, j}^{\theta} \coloneqq \sup_{\rho(\lambda, \lambda') \leq u_n} \psi_{\lambda, \lambda'} (\theta_{j}),$$
where the transformation $\psi$ introduced in  \eqref{eq:psi} is given as $\psi_{\lambda, \lambda'} (\theta_j) = \frac{\tau'}{\tau} j^{\alpha-\alpha'} \theta_{j}$. We then have that

\begin{align*}
	\underline{\psi}_{\lambda, j}^{\theta} 
	&=\inf_{|\log \tau - \log \tau'| + | \alpha - \alpha' | \leq u_n} \frac{\tau'}{\tau}  j^{\alpha-\alpha'} \theta_{j}\\
	&\geq \inf_{\substack{|\log \tau - \log \tau'| \leq u_n \\ | \alpha - \alpha' | \leq u_n}} \frac{\tau'}{\tau}  j^{\alpha-\alpha'} \theta_{j} =  		\inf_{|\log \tau - \log \tau'| \leq u_n} \Big( \frac{\tau'}{\tau} \inf_{|\alpha-\alpha'| \leq u_n}  j^{\alpha-\alpha'}  \theta_j \Big) \\ 	
	&\geq  \inf_{|\log \tau - \log \tau'| \leq u_n} \frac{\tau'}{\tau} j^{-\sign(\theta_j) u_n}  \theta_j=j^{-\sign(\theta_j) u_n}  e^{-			\sign(\theta_j) u_n} \theta_j.
\end{align*}
	Similarly,
\begin{align*}
	\overline{\psi}_{\lambda, j}^{\theta} \leq j^{\sign(\theta_j) u_n}  e^{\sign(\theta_j) u_n}\theta_j.
\end{align*}	 
By triangle inequality, taking into account that $u_n\to0$, we obtain that for sufficiently large $n$ it holds
\begin{align}\label{eq:E1a}
\Big|  \underline{\psi}_{\lambda, j}^{\theta} - \overline{\psi}_{\lambda, j}^{\theta} \Big| \nonumber
&\leq \Big| j^{\sign(\theta_j) u_n}  e^{\sign(\theta_j) u_n}\theta_j-j^{-\sign(\theta_j) u_n}  e^{-\sign(\theta_j) u_n}\theta_j \Big| \nonumber\\
&\leq \Big|e^{-\sign(\theta_j)u_n}\big(j^{-\sign(\theta_j) u_n} -j^{\sign(\theta_j) u_n} \big)\theta_j\Big|+\Big|j^{\sign(\theta_j) u_n}  \theta_j(e^{-\sign(\theta_j) u_n}-e^{\sign(\theta_j) u_n})\Big|\nonumber\\
&\leq \Big(2\big|j^{-u_n} -j^{ u_n} \big|+j^{u_n}\big|e^{-u_n}-e^{u_n}\big|\Big)|\theta_j|\nonumber\\
&\leq \Big(2j^{u_n}(\log j^{2u_n})+2j^{u_n}u_ne^{u_n}\Big)|\theta_j\nonumber|\\
&\leq 10j^{s'-1/2}u_n|\theta_j|,\end{align}
where for the second to last bound we used the inequality $1- 1/x \leq \log x, \forall x>0$ twice and in the last bound we fixed $s'$ such that $1/2<s'<s+1/2$.

By the definition of $Q^{\theta}_{\lambda, n}$ \big(see \eqref{eq:Q}, \eqref{eq:lr}\big), we have
\begin{align}\label{lemE.1_3l}
Q^{\theta}_{\lambda, n}(\mX) &= \int_{\mX} \sup_{\rho(\lambda, \lambda') \leq u_n} \exp\Big(n\ip{{X}}{\psi_{\lambda, \lambda'} (\theta)}-\frac{n}2\norm{\psi_{\lambda, \lambda'} (\theta)}_{2}^2\Big)dP^n_0({X})\nonumber\\
&\leq\prod_{j=1}^\infty\int_{\R}\sup_{\rho(\lambda, \lambda') \leq u_n} \sqrt{\frac{n}{2\pi}} \ e^{-\frac{n}2\big(x_j-\psi_{\lambda, \lambda'} (\theta_j)\big)^2}dx_j,
\end{align}
where we can bound the one-dimensional integrals as follows:
\begin{align*}
 &\int_{\R} \sup_{\rho(\lambda, \lambda') \leq u_n} \sqrt{\frac{n}{2\pi}} \ e^{-\frac{n}2\big(x_j-\psi_{\lambda, \lambda'} (\theta_j)\big)^2}dx_j \\
 &\leq \int_{x_j<\underline{\psi}_{\lambda, j}^{\theta} }\sqrt{\frac{n}{2\pi}} \ e^{-\frac{n}2\big(x_j-\underline{\psi}_{\lambda, j}^{\theta}\big)^2}dx_j+\int_{x_j>\overline{\psi}_{\lambda, j}^{\theta} }\sqrt{\frac{n}{2\pi}} \ e^{-\frac{n}2\big(x_j-\overline{\psi}_{\lambda, j}^{\theta}\big)^2}dx_j+\int_{\underline{\psi}_{\lambda, j}^{\theta}}^{\overline{\psi}_{\lambda, j}^{\theta}}\sqrt{\frac{n}{2\pi}}dx_j\\
 &\leq 1 + \sqrt{\frac{n}{2\pi}} \ \Big|  \underline{\psi}_{\lambda, j}^{\theta} - \overline{\psi}_{\lambda, j}^{\theta} \Big| \\
&\leq 1 + 4  u_n  \sqrt{n} \ j^{s'-1/2}|\theta_j|.
\end{align*}
Using the embedding of $H^s$ into $B^{s'}_{11}$ (implied by $s'<s+1/2$, see for example \cite[Theorem 3.3.1]{T83}), we have that there exists a positive constant $c>0$ such that
\begin{equation} \label{lemE.1_6l}
Q^{\theta}_{\lambda, n}(\mX)  \leq e^{4  u_n  \sqrt{n} \sum_{j=1}^\infty j^{s'-1/2}|\theta_j|}=e^{4u_n\sqrt{n}\norm{\theta}_{B^{s'}_{11}}}\leq e^{cu_n\sqrt{n}\norm{\theta}_{H^s}},
\end{equation}
for any $\alpha>s>0$ and any $\tau>0$.

For $\alpha\geq\ubar{\alpha}>s>0$ and $\theta\in\pth(\at)$, the last bound gives
\begin{equation} \label{lemE.1_6la}
Q^{\theta}_{\lambda, n}(\mX)   \leq e^{c u_n n^{1/2} Y_n}=e^{c\frac{\eta}{s_2} \tau u_n n^{3/2}\eps_n^2}.
\end{equation}
Taking into account Remark \ref{rem:eta} which says that $\eta=\eta(\alpha)$ remains bounded for $\alpha\in[\ubar{\alpha},\bar{\alpha}]$, as well as Lemma \ref{lem:fernique} which says that $s_2$ depends only on the fixed $\ubar{\alpha}$ and $s<\ubar{\alpha}$, we have that the exponent on the right hand side is bounded above by a constant multiple of $u_n\t_nn^{3/2}$, and thus remains bounded for the assumed choice of $u_n$.
\end{proof}


\begin{lem}\label{lem:E2}
Consider the white noise model. Let $\t_n\to\infty$ and $s,\ubar{\alpha}$ such that $\ubar{\alpha}>s>0$. Let $\gep_n=\gep_n(\at)$ and recall $\pth(\at)$ from Lemma \ref{lem:C1}. Consider any of the settings:

\begin{enumerate}
\item[i)] $\lambda=\tau$, where $0<\tau<\t_n$ and $\alpha>0$ is fixed, $u_n\lesssim \t_n^{-1}n^{-1}$;\smallskip
\item[ii)] $\lambda=\alpha$, where $\alpha\geq\ubar{\alpha}$ and $\tau>0$ is fixed, $u_n\lesssim n^{-1}$;\smallskip
\item[iii)] $\lambda=(\at)$, where $0<\tau<\t_n$ and $\alpha\geq\ubar{\alpha}$, $u_n\lesssim \t_n^{-1}n^{-1}$.
\end{enumerate}

Then there exists $C>0$ such that for $\eta$ as in Lemma \ref{lem:C1}, it holds
\begin{equation} \label{eq:lemE.2} \int_{\pth^c(\at)} Q^{\theta}_{\lambda, n}(\mX) \Pi(d\theta \mid \alpha,\tau) \leq Ce^{-\frac{\eta}{2}n\gep_{n}^{2}}.\end{equation}
The constant $C$ is independent of $\tau$, while in the settings (ii) and (iii), it can be chosen uniformly over $\alpha>\ubar{\alpha}$.
\end{lem}
\begin{proof}
This lemma corresponds to \cite[Lemma E.2]{RS17}, however for the proof we use a different technique relying on the Fernique theorem, see Lemma \ref{lem:fernique}. We prove here the case $\lambda=(\at)$ which is based on the estimate \eqref{lemE.1_6l}, while the other two cases follow easily with obvious modifications. 

Let $\lambda=(\at)$ where $\alpha\geq\ubar{\alpha}$ and $0<\tau<\t_n$. By \eqref{lemE.1_6l} (which holds for any $\theta$) and Cauchy-Schwarz inequality, we get that
\begin{align*}
\int_{\pth^c}Q^\theta_{\alpha, \tau, n}(\mX)\Pi(d\theta|\alpha,\tau) & \leq \int_{\pth^c} e^{c u_n n^{1/2} \norm{\theta}_{H^s}} \Pi(d\theta|\alpha, \tau)\\
&\leq\Pi(\pth^c|\alpha, \tau)^\frac12\bigg(\int_{\Theta} e^{2c u_n n^{1/2} \norm{\theta}_{H^s}} \Pi(d\theta|\alpha, \tau)\bigg)^\frac12\\
&=\Pi(\pth^c|\alpha, \tau)^\frac12\bigg(\int_{\Theta} e^{2c u_n n^{1/2} \tau \norm{\theta}_{H^s}} \Pi(d\theta|\alpha, 1)\bigg)^\frac12.
\end{align*}
Recall $\Pi(\pth^c | \at)^{\frac12} \leq 2s_1e^{-\frac{\eta}2n\eps_n^2}$ from Lemma \ref{lem:C1}, where $s_1>0$ is defined in Lemma \ref{lem:fernique} and depends only on $\ubar{\alpha}$ and $s$ (also recall that by \cite[Lemma 5.2]{ADH21} the ${H^s}$-norm is almost surely finite with respect to $\Pi(\cdot\mid\at)$ for $s<\alpha$). Therefore, to complete the proof, it suffices to show that the parenthesis term is bounded by a constant depending only on $\ubar{\alpha}$ and $s$. By Lemma \ref{lem:fernique} it suffices to have that the exponent in the integrand is less than $s_2\norm{\theta}_{H^s}$ for a certain constant $s_2=s_2(\ubar{\alpha};s)>0$. This is indeed the case for the assumed choice of $u_n$.
\end{proof}


\begin{lem}\label{lem:E3}
Consider the white noise model. Let $\t_n\to\infty$ and $s,\ubar{\alpha}, \bar{\alpha}$ such that $\bar{\alpha}>\ubar{\alpha}>s>0$. Recall   $\pth(\at)$ from Lemma \ref{lem:C1}. Consider any of the settings:

\begin{enumerate}
\item[i)] $\lambda=\tau$ where $\alpha>0$ is fixed, $\bar{\Lambda}_n\coloneqq (0,\t_n)$ and $u_n\lesssim \t_n^{-2}n^{-3}$;\smallskip
\item[ii)] $\lambda=\alpha$ where $\tau>0$ is fixed, $\bar{\Lambda}_n\coloneqq[\ubar{\alpha},\bar{\alpha}]$ and $u_n\lesssim n^{-3}$;\smallskip
\item[iii)] $\lambda=(\at)$, $\bar{\Lambda}_n\coloneqq[\ubar{\alpha},\bar{\alpha}]\times(0,\t_n)$, and $u_n\lesssim \t_n^{-2}n^{-3}$.
\end{enumerate}

Then for $\theta\in\pth(\at)$, there exist tests $\varphi_n(\theta)$ such that
\begin{equation} \label{eq:lemE3} 
\begin{split} & E^{n}_{\theta_{0}} \varphi_{n}(\theta)  \leq e^{-\frac1{32} n \norm{\theta-\theta_{0}}^{2}_{2}}, \\ 
& \sup_{\substack{\lambda\in\bar{\Lambda}_n}} \sup_{\substack{\theta' \in \pth(\at)  \\ \norm{\theta-\theta'}_{2}<\norm{\theta-\theta_{0}}_{2}/4}} \int_{\mX} \big(1-\varphi_{n}(\theta) \big) dQ^{\theta'}_{\lambda, n} (X) \leq e^{-\frac1{32} n \norm{\theta-\theta_{0}}^{2}_{2}}.
\end{split} \end{equation}
\end{lem}

\begin{proof}
The proof proceeds broadly along the lines of the proof of \cite[Lemma E.3]{RS17}, which considers $\theta, \theta'\in \R^n$ in the non-parametric regression setting. In order to handle the infinite dimensionality of both the data and the parameter space in the presently assumed white noise model setting, we substantially exploit the $H^s$-regularity of elements of the sieve set $\pth$. Furthermore, we need to study the case $\lambda=(\at)$. We in fact only study that case, while the other two cases follow with obvious modifications.

The likelihood ratio test
\begin{equation} \label{eq:E3_1} 
\varphi_{n}(\theta) ={\bf 1 }\Big[  \ip{{X}}{\theta-\theta_0}> \norm{\theta}^2_2/2-\norm{\theta_0}^2_2/2\Big],
\end{equation}
satisfies for any $\theta, \theta_0\in\ell_2$
\begin{equation}\label{eq:E3_2}\sup_{\substack{\theta' \in \ell_2: \norm{\theta-\theta'}_{2}<\norm{\theta-\theta_{0}}_{2}/4}} E^{n}_{\theta'} \big(1-\varphi_{n}(\theta)\big)\leq e^{-\frac1{32} n \norm{\theta-\theta_{0}}^{2}_{2}}\end{equation}
 and 
 $E^{n}_{\theta_{0}} \varphi_{n}(\theta)  \leq e^{-\frac1{32} n \norm{\theta-\theta_{0}}^{2}_{2}},$
see \cite[Lemma 5]{GV07} and \cite[Lemma K.6]{GV17}. In particular, the first inequality in \eqref{eq:lemE3} is verified and it remains to prove the second.

Using the definition of $Q^{\theta}_{\lambda, n}$ and letting $Z=Z(X)\coloneqq\sup_{\rho(\lambda,\lambda')\leq u_n} \left|\ip{X}{\psi_{\lambda, \lambda'} (\theta')-\theta'}\right|$, we have
\begin{align}\label{eq:E3_2b}
\int_{\mX} \big(1-\varphi_{n}(\theta) \big) dQ^{\theta'}_{\lambda, n} (X)&=\int_{\mX} \big(1-\varphi_n(\theta)\big)\sup_{\rho(\lambda, \lambda') \leq u_n} \Big(e^{n\ip{{X}}{\psi_{\lambda, \lambda'} (\theta')}-\frac{n}2\norm{\psi_{\lambda, \lambda'} (\theta')}_{2}^2}\Big)dP^n_0({X})\nonumber\\
\leq&\int_{\{Z < u_n\t_n^2 n^2\}} \big(1-\varphi_n(\theta)\big)e^{n\ip{{X}}{\theta'}-\frac{n}2\norm{\theta'}_{2}^2}\nonumber\\
&\times\sup_{\rho(\lambda, \lambda') \leq u_n} \Big(e^{n\ip{{X}}{\psi_{\lambda, \lambda'} (\theta')-\theta'}-\frac{n}2\norm{\psi_{\lambda, \lambda'} (\theta')}_{2}^2+\frac{n}2\norm{\theta'}_{2}^2}\Big)dP^n_0({X})\nonumber\\
&+\int_{\{Z > u_n\t_n^2 n^2\}} \;\sup_{\rho(\lambda, \lambda') \leq u_n} e^{n\left|\ip{{X}}{\psi_{\lambda, \lambda'} (\theta')}\right|}dP^n_0({X}).
\end{align}
We deal with the two integral-terms on the right hand side separately.

We examine the first integral, and we notice that by \eqref{eq:E3_2} it suffices to show that the supremum is bounded above by a constant. Indeed, using triangle inequalities (both sided), we obtain
\begin{align*}
&\sup_{\rho(\lambda,\lambda')\leq u_n}\Big|2n\ip{{X}}{\psi_{\lambda, \lambda'} (\theta')-\theta'}-n\norm{\psi_{\lambda, \lambda'} (\theta')}_{2}^2+n\norm{\theta'}_{2}^2\Big|\\
&\leq 2nZ+n\sup_{\rho(\lambda,\lambda')\leq u_n}\left\{\Big|\norm{\psi_{\lambda, \lambda'} (\theta')}_2-\norm{\theta'}_2\Big|\Big(\norm{\psi_{\lambda, \lambda'} (\theta')}_2+\norm{\theta'}_2\Big)\right\}\\
&\leq 2nZ+n\sup_{\rho(\lambda,\lambda')\leq u_n}\left\{\norm{\psi_{\lambda, \lambda'} (\theta')-\theta'}_{2}\Big(\norm{\psi_{\lambda, \lambda'} (\theta')}_2+\norm{\theta'}_2\Big)\right\}\\
&\leq 2nZ+ n\sqrt{\sum_{j=1}^\infty \big( \overline{\psi}_{\lambda, j}^{\theta'}-\underline{\psi}_{\lambda, j}^{\theta'} \big)^2}\Big(\sup_{\substack{\rho(\lambda, \lambda')}\leq u_{n}} \norm{\psi_{\lambda, \lambda'}(\theta')}_{2}+\norm{\theta'}_{H^{s}}\Big),
\end{align*}
where $\underline{\psi}_{\lambda, j}^{\theta'}, \overline{\psi}_{\lambda, j}^{\theta'}$ are defined in the proof of Lemma \ref{lem:E1}. Recall from \eqref{eq:psi} that, in the studied setting, the transformation $\psi$ is given as $\psi_{\lambda, \lambda'} (\theta_j) = \frac{\tau'}{\tau} j^{\alpha-\alpha'} \theta_{j}$.
We bound the right hand side above, for $\lambda=(\at)$ and for $\theta'\in\pth(\at)$. First note that since $u_n\to0$, for sufficiently large $n$ we have
\begin{align}\label{eq:E3_3}
\sup_{\substack{\rho(\lambda, \lambda')}\leq u_{n}} \norm{\psi_{\lambda, \lambda'}(\theta')}_{2}&\leq e^{u_n}\sup_{|\alpha-\alpha'|\leq u_n}\bigg(\sum_{j=1}^\infty j^{2(\alpha-\alpha')}\theta_j'^2\bigg)^{1/2}
\leq 2\bigg(\sum_{j=1}^\infty j^{2s}\theta_j'^2\bigg)^{1/2}= 2\norm{\theta'}_{H^{s}},
\end{align}
where we used that $s>0$. Similarly to the derivation of \eqref{eq:E1a}, using the inequality $1-1/x\leq \log(x)$ for $x>0$, we have

\begin{align*}
\Big(  \underline{\psi}_{\lambda, j}^{\theta'} - \overline{\psi}_{\lambda, j}^{\theta'} \Big)^2 \nonumber
&\leq \Big( j^{\sign(\theta_j') u_n}  e^{\sign(\theta_j') u_n}\theta_j'-j^{-\sign(\theta_j') u_n}  e^{-\sign(\theta_j') u_n}\theta_j' \Big)^2 \nonumber\\
&\leq 2e^{-2\sign(\theta_j')u_n}\big(j^{-u_n} -j^{u_n} \big)^2{\theta_j'}^2+2j^{2\sign(\theta_j') u_n}  {\theta_j'}^2(e^{-u_n}-e^{ u_n})^2\nonumber\\
&\lesssim  \big(j^{u_n} -j^{-u_n} \big)^2{\theta_j'}^2+j^{2u_n}  {\theta_j'}^2(e^{u_n}-e^{ -u_n})^2\nonumber\\
&\lesssim u_n^2j^{2u_n}(\log j)^2{\theta_j'}^2+u_n^2e^{2u_n}j^{2u_n}{\theta_j'}^2\nonumber\\
&\lesssim u_n^2j^{2s}{\theta_j'}^2 ,
\end{align*}
hence
\begin{equation}\label{eq:normdifbound}
\sqrt{\sum_{j=1}^\infty \big(\overline{\psi}_{\lambda, j}^{\theta'}-\underline{\psi}_{\lambda, j}^{\theta'}\big)^2}\lesssim u_n\sqrt{\sum_{j=1}^\infty j^{2s}\theta'^2_j}=u_n\norm{\theta'}_{H^s}.
\end{equation}
Combining 
with the definition of $\pth(\at)$ and Remark \ref{rem:eta}, we obtain
\begin{equation}\label{eq:E3_4}\Big|2n\ip{{X}}{\psi_{\lambda, \lambda'} (\theta')-\theta'}-n\norm{\psi_{\lambda, \lambda'} (\theta')}_{2}^2+n\norm{\theta'}_{2}^2\Big|\lesssim u_n\t_n^2  n^3=O(1),\end{equation}
for the given choice of $u_n$.

Returning to the second integral in \eqref{eq:E3_2b}, we first notice that for $X\sim P^n_0$, $Z$ is equal in distribution to $\frac{|\xi|}{\sqrt{n}}\sup_{\rho(\lambda,\lambda')\leq u_n}\norm{\psi_{\lambda, \lambda'} (\theta')-\theta'}_2,$ where $\xi\sim N(0,1)$ and similarly \[\sup_{\rho(\lambda,\lambda')\leq u_n}\ip{X}{\psi_{\lambda,\lambda'}(\theta')}\stackrel{D}{=}\frac{|\xi|}{\sqrt{n}}\sup_{\rho(\lambda,\lambda')\leq u_n}\norm{\psi_{\lambda,\lambda'}(\theta')}_2.\] The integral of interest can thus be written as
\[\frac{1}{\sqrt{2\pi}}\int_{|x|>r} \exp\Big(\sqrt{n}|x|\sup_{\rho(\lambda,\lambda')\leq u_n}\norm{\psi_{\lambda,\lambda'}(\theta')}_2-\frac{x^2}2\Big)dx,\]
where $r=\frac{u_n\t_n^2n^{5/2}}{\sup_{\rho(\lambda,\lambda')\leq u_n} \norm{\psi_{\lambda,\lambda'}(\theta')-\theta'}_2}$. By \eqref{eq:normdifbound}, the definition of $\pth(\at)$ and Remark \ref{rem:eta}, we have that (for sufficiently large $n$) the denominator in $r$ is bounded above by $u_n\t_n n$, hence the last displayed integral can be bounded above by
\[\frac{1}{\sqrt{2\pi}}\int_{|x|>\t_n n^{3/2}} \exp\Big(\sqrt{n}|x|\sup_{\rho(\lambda,\lambda')\leq u_n}\norm{\psi_{\lambda,\lambda'}(\theta')}_2-\frac{x^2}2\Big)dx.\]
Using \eqref{eq:E3_3}, we note that for $|x|>\t_n n^{3/2}$ and for $\theta'\in\pth(\at)$, it holds 
\[\sqrt{n}|x|\sup_{\rho(\lambda,\lambda')\leq u_n}\norm{\psi_{\lambda,\lambda'}(\theta')}_2\leq2\sqrt{n}|x|\norm{\theta'}_{H^s}=o(1)x^2.\]
Combined, the above considerations yield that the second integral in \eqref{eq:E3_2b} is bounded by
\[\frac{1}{\sqrt{2\pi}}\int_{|x|>\t_n n^{3/2}} \exp\Big(-\frac{x^2}4\Big)dx\leq ce^{-\frac{\t_n^2n^3}4},\]
where the latter bound is implied by  \cite[Lemma K.6]{GV17} for a sufficiently large constant $c>0$. Observing that for $\theta\in\pth(\at)$ it holds
\[\norm{\theta-\theta_0}_2\leq\norm{\theta_0}_2+\norm{\theta}_2\lesssim 1+\norm{\theta}_{H^s}\lesssim 1+\t_n n\eps_n^2(\lambda)=o(\t_nn),\]
we have that the obtained bound is of smaller order than $e^{-\frac1{32}n\norm{\theta-\theta_0}^2_2}$ which completes the proof.
\end{proof}


The next lemma concerns condition (H2).

\begin{lem}\label{lem:E4}
Consider the white noise model. Let $\t_n\to\infty$ and $s,\ubar{\alpha},\bar{\alpha}$ such that $\bar{\alpha}>\ubar{\alpha}>s>0$. Recall  $\pth(\at)$ from Lemma \ref{lem:C1} and consider any of the three settings studied in Lemma \ref{lem:E3}. Fix $K>0$. Then for all $\lambda\in\bar{\Lambda}_n$, for any $c_3\geq2+2K^2$ it holds
\[\sup_{\substack{\theta\in\pth(\at)\\\theta:\norm{\theta-\theta_0}_2\leq K\eps_n(\lambda)}} P^n_{\theta_0}\Big\{\inf_{\rho(\lambda,\lambda')\leq u_n}\ell_n\big(\psi_{\lambda,\lambda'}(\theta)\big)-\ell_n\big(\theta_0\big)\leq -c_3 n\eps_n^2(\lambda)\Big\}\leq 2 e^{-n\eps_n^2(\lambda)}.\]
\end{lem}

\begin{proof}
The proof follows the reasoning of the proof of \cite[Lemma E.4]{RS17}, which refers to nonparametric regression. As in Lemma \ref{lem:E3} above, we exploit the $H^s$-regularity induced by $\pth$, in order to handle the infinite dimensionality of both the data and the parameter space in the white noise model. As a result, we have the additional restriction $\theta\in\pth$ in the supremum, compared to \cite[Lemma E.4]{RS17}. Once more, we prove the case $\lambda=(\at)$ while the other two cases follow with obvious modifications.

We first note that it suffices to show that 
\begin{equation}\label{eq:E4_1}\sup_{\substack{\theta\in\pth(\at)\\\theta:\norm{\theta-\theta_0}_2\leq K\eps_n(\lambda)}} P^n_{\theta_0}\Big\{\sup_{\rho(\lambda,\lambda')\leq u_n}|\ell_n\big(\psi_{\lambda,\lambda'}(\theta)\big)-\ell_n\big(\theta\big)|\geq \frac{c_3}2 n\eps_n^2(\lambda)\Big\}\leq e^{-n\eps_n^2(\lambda)}\end{equation}
and 
\begin{equation}\label{eq:E4_2}\sup_{\substack{\theta\in\pth(\at)\\\theta:\norm{\theta-\theta_0}_2\leq K\eps_n(\lambda)}} P^n_{\theta_0}\Big\{\ell_n(\theta)-\ell_n(\theta_0)\leq -\frac{c_3}2 n\eps_n^2(\lambda)\Big\}\leq e^{-n\eps_n^2(\lambda)}.\end{equation}

As in the proof of Lemma \ref{lem:E3}, define $Z=Z(X)\coloneqq\sup_{\rho(\lambda,\lambda')\leq u_n} \left|\ip{X}{\psi_{\lambda, \lambda'} (\theta)-\theta}\right|$ and assume that $Z<u_n\t_n^2n^2$. Then by \eqref{eq:E3_4}, for $u_n$ as chosen in Lemma \ref{lem:E3}, it holds
\[\sup_{\rho(\lambda,\lambda')\leq u_n}|\ell_n\big(\psi_{\lambda,\lambda'}(\theta)\big)-\ell_n\big(\theta\big)|=O(1).\]
On the other hand, again similarly to the proof of Lemma \ref{lem:E3}, for sufficiently large $n$, it holds
\begin{align*}P^n_{\theta_0}\big(Z(X)>u_n\t_n^2n^2\big)&=2\bigg(1-\Phi\Big(\frac{u_n\t_n^2 n^{5/2}}{\sup_{\rho(\lambda,\lambda')\leq u_n} \norm{\psi_{\lambda, \lambda'} (\theta)-\theta}_2}\Big)\bigg)\\ &\leq 2\big(1-\Phi(\t_n n^{3/2})\big)\leq e^{-\frac{n^3\t_n^2}2}=o(e^{-n\eps^2_n(\lambda)}),\end{align*}
where $\Phi$ denotes the cumulative distribution function of the standard normal distribution, and we have used \cite[Lemma K.6]{GV17} to bound $1-\Phi$. Combining we get that \eqref{eq:E4_1} holds for any $c_3>0$.

For \eqref{eq:E4_2}, let $K(\theta_0, \theta)=\frac12n\norm{\theta-\theta_0}_2^2$ and note that this quantity coincides with the Kullback-Leibler divergence ${\rm KL}(P^n_{\theta_0};P^n_{\theta})=E^{n}_{\theta_0}[\ell_n(\theta_0)-\ell_n(\theta)]$, see for example \cite[Lemma 8.30]{GV17}. Using the (exponential) Markov inequality, we then have
\begin{align*}
&\sup_{\substack{\theta\in\pth(\at)\\\theta:\norm{\theta-\theta_0}_2\leq K\eps_n(\lambda)}} P^n_{\theta_0}\Big\{\ell_n(\theta)-\ell_n(\theta_0)\leq -\frac{c_3}2 n\eps_n^2(\lambda)\Big\}\\
&=\sup_{\substack{\theta\in\pth(\at)\\\theta:\norm{\theta-\theta_0}_2\leq K\eps_n(\lambda)}} P^n_{\theta_0}\Big\{\ell_n(\theta)-\ell_n(\theta_0)+K(\theta_0,\theta)\leq K(\theta_0,\theta)-\frac{c_3}2 n\eps_n^2(\lambda)\Big\}\\
&\leq\sup_{\substack{\theta\in\pth(\at)\\\theta:\norm{\theta-\theta_0}_2\leq K\eps_n(\lambda)}} P^n_{\theta_0}\Big\{\ell_n(\theta)-\ell_n(\theta_0)-E^n_{\theta_0}[\ell_n(\theta)-\ell_n(\theta_0)]\leq \frac{K^2-c_3}2 n\eps_n^2(\lambda)\Big\}\\
&\leq \sup_{\substack{\theta\in\pth(\at)\\\theta:\norm{\theta-\theta_0}_2\leq K\eps_n(\lambda)}} P^n_{\theta_0}\Bigg\{\exp\Big(\ell_n(\theta_0)-\ell_n(\theta)-E^n_{\theta_0}[\ell_n(\theta_0)-\ell_n(\theta)]\Big)\geq e^{\frac{c_3-K^2}2 n\eps_n^2(\lambda)}\Bigg\}\\
&\leq \sup_{\substack{\theta\in\pth(\at)\\\theta:\norm{\theta-\theta_0}_2\leq K\eps_n(\lambda)}} E^{n}_{\theta_0}\Bigg\{\exp\Big(\ell_n(\theta_0)-\ell_n(\theta)-E^n_{\theta_0}[\ell_n(\theta_0)-\ell_n(\theta)]\Big)\Bigg\}e^{\frac{K^2-c_3}2 n\eps_n^2(\lambda)}.
\end{align*}
Since for $W\sim P^{(1)}_0$, it holds that $\ip{W}{\theta_0-\theta}$ is a univariate normal random variable with mean zero and variance $\norm{\theta_0-\theta}_2^2$, we obtain that the last expectation is equal to 
\[e^{-\frac{n}2\norm{\theta_0}_2^2+\frac{n}2\norm{\theta}_2^2-\frac{n}2\norm{\theta-\theta_0}_2^2}E^n_{\theta_0}\Big[e^{n\ip{{\mathbf x}}{\theta_0-\theta}}\Big]=e^{\frac{n}2\norm{\theta-\theta_0}_2^2}.\] Choosing $c_3\geq 2+2K^2$, we thus verify \eqref{eq:E4_2} and the proof is complete. 
\end{proof}

The next three lemmas concern condition (H1) for $\lambda=\tau$, $\lambda=\alpha$ and $\lambda=(\at)$, respectively.

\begin{lem}\label{lem:3.5}
Consider an $\alpha$-regular $\tau$-scaled $p$-exponential prior, for fixed $\alpha > 0$ and $\lambda=\tau$ being the hyper-parameter. Let $\theta_0\in H^{\beta}$, $\beta\geq(1+\alpha p)/(p+2\alpha p)$. Consider a hyper-prior $\tilde{\pi}$ on $\tau$, satisfying Assumption \ref{ass:hyper_prior}(i) for some ${\conr_0}>0$. 
For $\enab=\enab(\alpha,\beta,p)$ as defined in Lemma \ref{lem:Lemma3.3}, any $\bar{c}_0>0$ and $\tilde{M}_n$ tending to infinity arbitrarily slowly, let
$$\Lambda_{n}:=\big[n^{-\frac{1}{2+p+2\alpha p}}, e^{2\conr_{0} \bar{c}_{0} n \tilde{M}_{n}^{2} \enab^{2}} \big].$$ 
Then Assumption (H1) from Subsection \ref{sec:conditions} is verified for these $\Lambda_n, \bar{c}_0$ and for $\tilde{w}_n=\tilde{M}_n \enab/\gep_{n,0}$.
\end{lem}

\begin{proof}
The proof proceeds similarly to the proof of \cite[Lemma 3.5]{RS17}, with the difference that, due to the lack of good lower bounds for $\gep_n(\lambda)$, here we work with the upper bounds obtained in Lemma \ref{lem:Lemma3.3} rather than the oracle rate $\gep_{n,0}$ defined in \eqref{defn:en0}. 

We first explicitly identify $\tilde{\Lambda}_0 \subseteq \Lambda_0(\tilde{w}_n)=\{\tau\in\Lambda_n:\gep_n(\alpha,\tau)\leq \tilde{M}_n\enab\}$, which has enough mass under the hyper-prior $\tilde{\pi}(\tau)$, so that \eqref{eq:H1a} is satisfied. To this end, we consider the value of $\tau$, denoted by $\tau_0$, which balances the two terms in the upper bounds on $\gep_n(\at)$ found in part (i) of Lemma \ref{lem:Lemma3.2} and results in the optimized bounds derived in Lemma \ref{lem:Lemma3.3}. Clearly, for sufficiently large $n$ it holds $\tau_0 \in \Lambda_{0} (\tilde{w}_{n})$. We will show that $\tilde{\Lambda}_0$ can be chosen as the interval $\big[\tau_0, 2\tau_0 \big]$.

First, notice that $[\tau_0, 2\tau_0 ]\subseteq \Lambda_{0}(\tilde{w}_{n})$ for sufficiently large $n$. Indeed, considering the upper bounds in Lemma \ref{lem:Lemma3.2} and setting $\tau=c\tau_0$ for $c \in [1,2]$, it is straightforward to check that for any relationship between $\alpha, \beta$ and $p$, there exists $M>0$, such that $\gep_{n}(\at) \leq M\enab$ for all $c\in[1,2]$. [For the case $\beta=\alpha+1/p$, notice that since $\beta>1/p$ and since $\enab \gtrsim n^{-\beta/(1+2\beta)}$, we have $\log \big(n \tau_0^p \big) \to \infty$ which streamlines the calculation.] Thus, there exists $N>0$ sufficiently large, such that for all $n\geq N$ and any $\tau\in [{\tau_0}, 2\tau_0 ]$, it holds $\gep_{n}(\at) \leq \tilde{M}_n\enab$. Moreover, it is straightforward to check that for all relationships between $\alpha, \beta, p$, we have $[{\tau_0}, 2\tau_0 ]\subseteq \Lambda_n$. Combining, we get that $[{\tau_0}, 2\tau_0 ]\subseteq \Lambda_{0}(\tilde{w}_{n})$.

Next, the assumed lower bounds on $\tilde{\pi}(\tau)$ in \eqref{thm1_1} (taking into account that $p\in[1,2]$), imply that there exist $\conr'_1, \conr_3'>0$, such that
$$\int _{\tau_0}^{2\tau_0} \tilde{\pi}(\tau) d\tau \gtrsim e^{-{\conr_1}' \tau_0^{2/(1+2\alpha)}} \wedge e^{-{\conr_3}' \tau_0^{-p}}.$$ 
In addition, using the definitions of $\enab$ and $\tau_0$ in Lemma \ref{lem:Lemma3.3},  it is straightforward to check that $n \enab^{2} \gtrsim \tau_0^{2/(1+2\alpha)} \vee \tau_0^{-p}$. Combining, we obtain that for $\tilde{\Lambda}_0:= [{\tau_0},2\tau_0 ]$ there exists a constant $c'>0$ such that
\[\int_{\tilde{\Lambda}_0}\tilde{\pi}(\tau)d\tau\gtrsim e^{-c'n\enab^2},\] which, by the definition of $\tilde{w}_n$ and since $\tilde{M}_n\to\infty$, verifies \eqref{eq:H1a}.

Turning to condition \eqref{eq:H1b}, by the definition of $\Lambda_n$, the fact that $\tilde{\pi}(\tau)$ is supported on $\big[n^{-\frac{1}{2+p+2\alpha p}},\infty\big)$ and the upper bound on $\tilde{\pi}(\tau)$ in \eqref{thm1_1}, we can bound
\begin{align*}
\int _{\Lambda_n^c} \tilde{\pi}(\tau) d\tau 
 &\leq  \int _{e^{2{\conr_0}\bar{c}_0 n \tilde{M}_{n}^{2} \enab^2}}^{\infty} \tau^{-{\conr_2}} d\tau  \leq {\conr_0} e^{- 2\bar{c}_0 n \tilde{M}_{n}^{2} \enab^2}\leq e^{-\bar{c}_0n\tilde{M}_n^2\enab^2}.
\end{align*}
 By the definition of $\tilde{w}_n$, \eqref{eq:H1b} is verified and the proof is complete.
\end{proof}


\begin{lem}\label{lem:3.6}
Consider an $\alpha$-regular $p$-exponential prior (with $\tau=1$), and  let $\lambda=\alpha$ be the hyper-parameter. Consider a hyper-prior $\tilde{\pi}$ on $\alpha$, satisfying Assumption \ref{ass:hyper_prior}(ii) for some $\bar{\alpha}>\ubar{\alpha}>0$. Let $\theta_0 \in H^{\beta}$, $\beta\in[\ubar{\alpha},\bar{\alpha}]$. Define
$\Lambda_{n}$ to be the support of $\tilde{\pi}$, $\Lambda_n:=[\ubar{\alpha}, \bar{\alpha}]$ (constant for all $n\in\N$).
Recall $\minimax$ from \eqref{eq:minimax} and consider any $\tilde{M}_n \to \infty$ such that $\tilde{M}_n=o(1/\minimax)$. 
Then, there exists $\bar{c}_0>0$ such that Assumption (H1) from Subsection \ref{sec:conditions} holds for these $\Lambda_n$, $\bar{c}_0$ and for $\tilde{w}_n=\tilde{M}_n \minimax / \gep_{n,0}$.
\end{lem}

\begin{proof}
We fix any sequence $\tilde{M}_n \to \infty$ such that $\tilde{M}_n=o(1/\minimax)$ and explicitly identify $\tilde{\Lambda}_0 \subseteq \Lambda_0(\tilde{w}_n)$, which has enough mass under the prior $\tilde{\pi}(\alpha)$, so that \eqref{eq:H1a} is satisfied. By Lemma \ref{lem:Lemma3.3}, for sufficiently large $n$ it holds $\beta \in  \Lambda_0(\tilde{w}_n)=\{\alpha\in\Lambda_n:\gep_n(\alpha,1)\leq \tilde{M}_n\minimax\}$. We will show that $\tilde{\Lambda}_0$ can be chosen as the interval $[\beta-r_n, \beta+r_n]$ for appropriately decaying $r_n \to 0$. 

To this end, we first use the upper bounds in part (i) of Lemma \ref{lem:Lemma3.2} (with $\tau=1$) in order to identify how fast $r_n$ needs to decay, so that $[\beta-r_n, \beta+r_n]\subseteq\Lambda_0(\tilde{w}_n)$. 
It suffices to study the case $\beta < \alpha + 1/p$ and the related upper bound
\[\gep_n(\alpha,1)\lesssim\tilde{\gep}_n(\alpha,1)= n^{-\frac{\alpha}{1+2\alpha}}+n^{\frac{\beta}{\beta(p-2)-\alpha p-1}}.\] 
The first term dominates in the bound for $\alpha<\beta$ (decreasing with $\alpha$), the second term dominates for $\alpha>\beta$ (increasing with $\alpha$), while the two terms are balanced for $\alpha=\beta$. Replacing the extreme values $\alpha=\beta\pm r_n$ in the corresponding dominating term, it is straightforward to check that for sufficiently large $N>0$ and for all $n \geq N$, it holds $\gep_{n}(\alpha,1) \leq \tilde{M}_n \minimax$ for all $\alpha\in[\beta-r_n,\beta+r_n]$ with $r_n=1/\log n$. Since $[\beta-1/\log{n}, \beta+1/\log{n}] \subseteq \Lambda_{n}$ for sufficiently large $n$, we indeed get that $[\beta-r_n, \beta+r_n] \subseteq \Lambda_{0} (\tilde{w}_{n})$ (also for large $n$). 

Next, using the assumed  lower bound of $\tilde{\pi}(\alpha)$ in \eqref{thm1_2}, we obtain that
$$\int_{\beta-r_n}^{\beta+r_n}\tilde{\pi}(\alpha) d\alpha  \gtrsim r_n  \gtrsim \frac1{\log n} \gtrsim e^{-n \tilde{M}_n^{2} {(\minimax)}^2},$$ where for the last bound, we have used that $n({\minimax})^2\to\infty$ polynomially fast. By the definition of $\tilde{w}_n$, we have therefore verified \eqref{eq:H1a} for $\tilde{\Lambda}_0=[\beta-r_n, \beta+r_n],\; r_n=1/\log n$.

Due to the fact that $\tilde{\pi}$ is supported on $\Lambda_n$, the upper bound \eqref{eq:H1b} is trivially satisfied, thus the proof is complete.
\end{proof}


\begin{lem}\label{lem:3.7}
Consider an $\alpha$-regular $\tau$-scaled $p$-exponential prior and let $\lambda=(\alpha, \tau)$ be the hyper-parameter. Consider a hyper-prior on $\lambda$,  with density $\tilde{\pi}(\at)$, which satisfies Assumption \ref{ass:hyper_prior}(iii) for some $\bar{\alpha}>\ubar{\alpha}>0$ and some ${\conr_0}>0$. 

Let $\theta_0 \in B^{\beta}_{qq}$, for $1\leq p\leq q<2$, $\beta\in(\ubar{\alpha}+1/p, \bar{\alpha}+1/p)$. Recall $\minimax$ from \eqref{eq:minimax} and consider $\tilde{M}_n \to \infty$ such that $\tilde{M}_n=o(1/\minimax)$. For any $\bar{c}_0>0$,  define
$$\Lambda_{n} := \Big\{(\at): \alpha\in[\ubar{\alpha}, \bar{\alpha} ], \tau\in\big[n^{-\frac1{2+p+2\alpha p}}, e^{2{\conr_0} \bar{c}_0 n \tilde{M}_{n}^{2} {(\minimax)}^2 (\log n)^{2\omega'}} \big]\Big\}, $$
where $\omega'=(q-p)/[pq(1+2\beta)]$.

Then Assumption (H1) in Subsection \ref{sec:conditions} holds for these $\Lambda_n, \bar{c}_0, \tilde{M}_n$ with $\tilde{w}_n=\tilde{M}_n \minimax (\log n)^{\omega'}/\gep_{n,0}$. 
\end{lem}

\begin{proof}
We fix any sequence $\tilde{M}_n \to \infty$ such that $\tilde{M}_n=o(1/\minimax)$ and explicitly identify $\tilde{\Lambda}_0  \subseteq \Lambda_0(\tilde{w}_n)$, which has enough mass under the prior $\tilde{\pi}(\at)$, so that \eqref{eq:H1a} is satisfied. We consider $\alpha_0=\beta-1/p$ and $\tau_0=n^{-\frac{1}{p(1+2\beta)}}(\log n)^\omega$, with $\omega=\Big( p- \frac{2}{1+2\beta}\Big) \frac{q-p}{p^2 q}$, where $\omega\ge0$, since $\beta>1/p$ and $p\le q$.
By Lemma \ref{lem:BBesov}, these choices of $\alpha$ and $\tau$ lead to  optimized bounds on $\gep_n(\alpha,\tau)$ for $\theta_0\in B^{\beta}_{qq}$, while it is easy to check that $(\alpha_0,\tau_0)\in\Lambda_n$ using the assumption on $\beta$, thus for sufficiently large $n$ it holds that $(\alpha_0, \tau_0) \in  \Lambda_0(\tilde{w}_n)$. We will show that $\tilde{\Lambda}_0$ can be chosen as the rectangle $[\alpha_0 - r_n, \alpha_0] \times [{\tau_0}, 2\tau_0 ]$ for appropriately decaying $r_n \to 0$. 

We first identify how fast $r_n$ needs to decay, so that $[\alpha_0 - r_n, \alpha_0] \times[ {\tau_0}, 2\tau_0] \subseteq \Lambda_0(\tilde{w}_n)$. 
By part (ii) of Lemma \ref{lem:Lemma3.2}, for $p\leq q$ and $\alpha<\beta-1/p$ we have the bound
\begin{equation} \label{lem3.7_2}
\gep_{n}(\lambda) \lesssim\tilde{\gep}_n(\at) = \tau^{\frac{1}{1+2\alpha}} n^{-\frac{\alpha}{1+2\alpha}} + \tau^{-\frac{p}{2}} n^{-\frac12}.
\end{equation} 
Notice that since $\alpha<\beta-1/p$, for all $\tau\in[{\tau_0},2\tau_0]$ it holds that  $\tilde{\gep}_n(\at)\lesssim \tau_0^{\frac{1}{1+2\alpha}}n^{-\frac{\alpha}{1+2\alpha}}$, thus $\tilde{\gep}_n(\at)\lesssim n^{-\frac{1+\alpha p (1+2\beta)}{p(1+2\alpha)(1+2\beta)}} (\log n)^{\frac{\omega}{1+2\alpha}}$. The last upper bound, can be seen to be  decreasing in $\alpha$, therefore, in order to determine $r_n$ as above, it suffices to study $\tilde{\gep}_n$ for $\alpha=\alpha_0-r_n$. Since $\beta>1/p$ and $r_n\to0$, it is straightforward to check that for all sufficiently large $n$, it holds $\tilde{\gep}_{n}(\alpha_0-r_n,\tau_0) \leq \tilde{M}_n \minimax (\log n)^{\omega'}$ (where the right hand side is equal to $w_n\gep_{n,0}$ by the definition of $w_n$), provided $r_n \leq\frac{\log \tilde{M}_n- s \log{\log{n}}}{\log n}$, where $s=\frac{\omega}{1+2\beta -\frac2p-2r_n} - \omega'$. In particular, it can be seen that for the choice $r_n=1/\log{n}$, it holds that $s\log{\log{n}}=o(1)$ and hence this choice of $r_n$ satisfies the required upper bound for sufficiently large $n$.
In conclusion, and noticing that, by the assumption on $\beta$, it holds that $[\alpha_0 - r_n, \alpha_0] \times [ {\tau_0}, 2\tau_0 ] \subseteq \Lambda_n$ for sufficiently large $n$, we have that $\tilde{\Lambda}_\alpha\times \tilde{\Lambda}_\tau :=[\alpha_0 - 1/\log{n}, \alpha_0] \times [ {\tau_0}, 2\tau_0 ] \subseteq \Lambda_0(\tilde{w}_n)$ (also for sufficiently large $n$). 

We next verify \eqref{eq:H1a} for $\tilde{\Lambda}_\alpha\times \tilde{\Lambda}_\tau$, using the assumed lower bounds on the hyper-priors in \eqref{thm1_1} and  \eqref{thm1_2}. Similarly to Lemmas \ref{lem:3.6} and \ref{lem:3.5}, noting that $\tau_0\to0$ hence only the second bound in \eqref{thm1_1} is relevant, there exists ${\conr_3}'>0$ such that
\begin{align*}
\int_{\beta-\frac1p-\frac1{\log{n}}}^{\beta-\frac1p}&\Bigg( \int_{{\tau_0}}^{2\tau_0} \tilde{\pi}_{\tau}(\tau|\alpha) d\tau \Bigg) \tilde{\pi}_{\alpha}(\alpha) d\alpha \gtrsim \int_{\beta-\frac1p-\frac1{\log{n}}}^{\beta-\frac1p}\Bigg( \int_{{\tau_0}}^{2\tau_0} e^{-\conr_3 \tau^{-p}}d\tau \Bigg) \tilde{\pi}_{\alpha}(\alpha) d\alpha\\
&\gtrsim e^{-\conr_3' \tau_0^{-p}} \int_{\beta-\frac1p-\frac1{\log{n}}}^{\beta-\frac1p} \tilde{\pi}_{\alpha}(\alpha)d\alpha \gtrsim  \frac{1}{\log n}  e^{-{\conr_3}' \tau_0^{-p}} \gtrsim e^{-n \tilde{M}_n^{2} {(\minimax)}^2(\log n)^{2\omega'}},
\end{align*} 
where for the last bound we have used that $n{(\minimax)}^2\to\infty$ polynomially fast and $\tau_0^{-p} \lesssim n {(\minimax)}^2$. By the definition of $\tilde{w}_n$, we have therefore verified \eqref{eq:H1a}.

Finally, we verify \eqref{eq:H1b} using that $\tilde{\pi}_{\alpha}$ is supported on $[\ubar{\alpha},\bar{\alpha}]$,  that $\tilde{\pi}_\tau(\cdot|\alpha)$ is supported on $\big[n^{-\frac1{2+p+2\alpha p}},\infty\big)$ and the upper bound on $\tilde{\pi}_{\tau}(\cdot|\alpha)$ from \eqref{thm1_1}. Indeed, using similar calculations as in Lemma \ref{lem:3.5},   we can bound 
\begin{align*}
\int_{\Lambda^c_n} \tilde{\pi}(\alpha, \tau) d\alpha d\tau &= 
\int_{\ubar{\alpha}}^{\bar{\alpha}}\Bigg(\int_{e^{2{\conr_0} \bar{c}_0 n \tilde{M}_n^2 {(\minimax)}^2 (\log n)^{2\omega'}}}^\infty \tilde{\pi}_\tau(\tau|\alpha)d\tau\Bigg)\tilde{\pi}_\alpha(\alpha)d\alpha\\
& \lesssim {\conr_0} e^{-2\bar{c}_0 n \tilde{M}_n^{2} {(\minimax)}^2 (\log n)^{2\omega'}} \\
& \lesssim e^{-\bar{c}_0n \tilde{M}_n^{2} {(\minimax)}^2 (\log n)^{2\omega'}}.
\end{align*}
By the definition of $\tilde{w}_n$, the last bound verifies \eqref{eq:H1b}  and the proof is complete.
\end{proof}

\begin{rem}
Notice that the assumed lower bound for $\tau\geq1$ in \eqref{thm1_1}, $\tilde{\pi}_\tau(\tau)\gtrsim e^{-{\conr_1}\tau^\frac{2}{1+2\alpha}}$, is not used in the last proof. 
\end{rem}


\section{The constant in the centered small ball probabilities}
\begin{lem}\label{lem:centeredconstants}
Fix $\ubar{\alpha}>0$ and consider all $\alpha$-regular $p$-exponential measures $\Pi_\alpha\coloneqq\Pi(\cdot\mid\alpha, \tau=1)$, with $\alpha\ge\ubar{\alpha}$. Recall the concentration function $\varphi_\theta(\cdot)$ from \eqref{defn:concfcn} and consider
\[\varphi_0(\gep;\alpha)\coloneqq -\log \Pi_\alpha(\gep B_{\ell_2}).\]
There exists a constant $\tilde{c}>1$, such that 
\begin{equation}\label{1stclaim}\tilde{c}^{-1}\alpha\leq\lim_{\gep\to0}\gep^{1/\alpha}\varphi_0(\gep;\alpha)\leq \tilde{c}\alpha,\quad \forall \alpha\geq\ubar{\alpha}.\end{equation}
In particular, for any fixed $\bar{\alpha}>\ubar{\alpha}$ and any $\gep_0>0$, there exist constants $\tilde{c}_1>0$, such that 
\begin{equation}\label{2ndclaim}\tilde{c}_1^{-1}\gep^{-1/\alpha}\leq \varphi_0(\gep;\alpha)\leq \tilde{c}_1\gep^{-1/\alpha},\quad\forall \gep\in(0,\gep_0], \quad \forall \alpha\in[\ubar{\alpha},\bar{\alpha}].\end{equation}
\end{lem}

\begin{proof}
By \cite[Theorem 4.2]{FA07} (where the small ball probabilities in $\ell_2$ for $\alpha$-regular $p$-exponential priors correspond to  powers $p=2$ and $\mu=1/2+\alpha$ in that source's notation), we have that 
\[\lim_{\gep\to0}\gep^{1/\alpha}\varphi_0(\gep;\alpha)=C(\alpha),\]
where 
\begin{equation}\label{eq:ca}C(\alpha)\coloneq\alpha\bigg(\frac{K(\alpha)^{\alpha+\frac12}}{\sqrt{2}(\alpha+\frac12)^{\alpha+\frac12}}\bigg)^{1/\alpha},\end{equation}
and \[K(\alpha)\coloneqq -\int_0^\infty \log\E \big[e^{-z^{-1-2\alpha}\xi^2}\big]dz.\]
Here $\xi$ is a real valued $p$-exponentially distributed random variable, with probability density function given by $f_p(x)=\frac1{c_p}\exp(-{|x|^p}/p),$ where $c_p=2\Gamma(1/p)p^{1/p-1}$. 

Using the bound 
\begin{equation}\label{eq:ccc}\exp(-|x|^p/p)\ge c\exp(-x^2/p), \quad \forall x\in\R, \end{equation} where $c$ is a sufficiently small positive constant, we can show that there exists another constant $c>0$, such that 
\begin{equation}\label{lem:lb}\E\big[e^{-z^{-1-2\alpha}\xi^2}\big]=\frac{1}{c_p}\int_{-\infty}^\infty e^{-z^{-1-2\alpha}x^2-\frac{|x|^p}p}dx\ge c \Big(\frac1p+z^{-1-2\alpha}\Big)^{-1/2}.\end{equation} 
On the other hand, we have
\begin{equation}\label{lem:ub}\E\big[e^{-z^{-1-2\alpha}\xi^2}\big] \leq \min\bigg\{\int_{-\infty}^\infty \frac{e^{-z^{-1-2\alpha}x^2}}{c_p}dx, \int_{-\infty}^\infty \frac{e^{-\frac{|x|^p}p}}{c_p}dx\bigg\} = \min\Big\{\frac{\sqrt{\pi}}{c_p}z^{\frac12+\alpha}, 1\Big\} \leq \min \big\{z^{\frac12+\alpha}, 1\big\}.\end{equation}
Note that 
\[-\int_1^\infty \log \E \big[e^{-z^{-1-2\alpha}\xi^2}\big]dz\leq K(\ubar{\alpha})<\infty, \quad \forall \alpha\ge \ubar{\alpha}.\]

Combined with \eqref{lem:lb}, this gives
\[K(\alpha)\leq -\int_0^1\log \E\big[e^{-z^{-1-2\alpha}\xi^2}\big]dz+K(\ubar{\alpha})\leq -\int_0^1\log\bigg(c\Big(\frac1p+z^{-1-2\alpha}\Big)^{-1/2}\bigg)dz+K(\ubar{\alpha}),\] and it is straightforward to check that there exists another constant $c>1/2$ depending on $p$ and $\ubar{\alpha}$ but not on $\alpha$, such that 
\begin{equation}\label{eq:Kbound}K(\alpha)\leq c+\alpha.\end{equation} 
Here, it is possible to take $c>1/2$, due to the fact that in \eqref{eq:ccc} the constant $c$ can be taken as small as we wish.
Moreover, using the upper bound \eqref{lem:ub}, it is straightforward to show that 
\[K(\alpha)\geq-\int_0^\infty\log\Big(\min\big\{z^{1/2+\alpha},1\big\}\Big)dz\ge -\int_0^1\log\big(z^{1/2+\alpha}\big)dz=1/2+\alpha.\]

By the above bounds on $K(\alpha)$, we thus have 
\[2^{-1/2\alpha} \leq \frac{C(\alpha)}{\alpha} \leq \bigg(\frac{(\alpha+c)^{\alpha+\frac12}}{\sqrt{2}(\alpha+\frac12)^{\alpha+\frac12}}\bigg)^{1/\alpha},\]
and the proof of \eqref{1stclaim} is complete since both the upper and lower bounds of $C(\alpha)/\alpha$ are continuous over the range of $\alpha$ and tend to 1 as $\alpha\to\infty$.

In order to prove \eqref{2ndclaim}, we delve into the theory developed in \cite{FA07} (see also the PhD thesis \cite{FA06thesis}), which lead to \cite[Theorem 4.2]{FA07} which we quoted in the proof of the first claim, and we extract the desired nonasymptotic bound with constants which are uniform for $\alpha\in[\ubar{\alpha},\bar{\alpha}]$. 

Let $\xi, \xi_1, \xi_2, \dots$ be independent and identically distributed real valued $p$-exponential random variables, and denote $\mu=1/2+\alpha$. Recall that we are interested in studying the probability 
\[\Pi_\alpha(\gep B_{\ell_2}) \equiv \rp\Big(\sum_{\ell=1}^\infty \ell^{-1-2\alpha} \xi_\ell^2\leq \gep^2\Big) \equiv \rp \Big(\sum_{\ell=1}^\infty \ell^{-2\mu} \xi_\ell^2\leq \gep^2 \Big).\] 
By \cite[Theorem 2.2]{FA07}, applied with $S(x)=x^{-\mu}$ so that $S^{-1}(x)=x^{-1/\mu}$, we have for all $\lambda>0$
\begin{equation}\label{eq:thm22a}
\log \E e^{-\lambda\xi^2}+J_{\mu}\leq \log \E e^{-\lambda\sum_{\ell=1}^\infty \ell^{-2\mu}\xi_\ell^2} \leq  J_{\mu},
\end{equation}
where \[J_\mu\coloneqq -\int_{0}^{\sqrt{\lambda}} \log \E e^{-y^2\xi^2} \frac{d}{dy}\big[S^{-1}\big(y\lambda^{-\frac12}\big)\big]dy=\lambda^{\frac1{2\mu}}\int_0^{\sqrt{\lambda}}\log \E e^{-y^2\xi^2}\frac1\mu y^{-\frac1\mu-1}dy.\]
A change of variables then shows that $J_{\mu}=\lambda^{\frac1{2\mu}}I_{\mu},$
where 
\[I_\mu=I_\mu(\lambda)\coloneqq \int_{\lambda^{-\frac1{2\mu}}}^\infty \log \E e^{-z^{-2\mu}\xi^2}dz.\]
We bound $I_\mu$ from above and below, uniformly for all $\mu\in[\ubar{\mu},\bar{\mu}]$, where $\ubar{\mu}\coloneqq 1/2+\ubar{\alpha}, \bar{\mu}\coloneqq1/2+\bar{\alpha}.$
Noting that the integrand is negative, by \eqref{eq:Kbound} we have that  
\[I_\mu\ge -K(\alpha)\ge -c-\alpha, \quad\forall \lambda>0,\] where $c$ depends on $\ubar{\alpha}$ and $p$, but not on $\alpha$. As a result there is a finite constant $R_1=R_1(p,\ubar{\alpha},\bar{\alpha})>0$ independent of $\alpha$, such that \[I_\mu\ge  -R_1, \quad\forall \lambda>0,\quad\forall \alpha\in[\ubar{\alpha},\bar{\alpha}].\]
Furthermore, noting that for any $\lambda\ge1$
\[I_\mu\leq \int_{1}^\infty \log \E e^{-z^{-2\mu}\xi^2}dz,\] and  bounding $z^{-2\mu}$ from below by $z^{-2\bar{\mu}}$, it is straightforward to show that there exists a finite constant $R_2=R_2(p,\ubar{\alpha},\bar{\alpha})>0$ independent of $\alpha$, such that 
\[I_\mu\le -R_2, \quad\forall \lambda\ge 1, \quad\forall \alpha\in[\ubar{\alpha},\bar{\alpha}].\]

We next turn our attention to the term  $\log \E e^{-\lambda\xi^2}$ appearing in \eqref{eq:thm22a}. Observe that for all $p\in[1,2]$, $\alpha$-regular $p$-exponential priors satisfy \cite[Condition (O)]{FA07} for any $r>0$, that is there exists $C_1=C_1(r)>0$ such that \[\rp(|\xi|\leq r)\ge e^{-C_1t^{-1/r}}, \quad \forall 0<t\leq 1.\] Indeed, by \cite[Equation (2.21)]{FA06thesis}, it suffices to verify the existence of $\delta, C>0$ such that 
\begin{equation}\label{eq:221}\rp(|\xi|\leq t)\geq Ct^\delta, \quad\forall 0<t\leq 1.\end{equation}
Since for any $t\in(0,1]$, for all $p\in[1,2]$ we have
\[\rp(|\xi|\leq t)=\frac2{c_p}\int_0^te^{-\frac{x^p}p}dx\ge \frac2{c_p}\int_0^t e^{-x}dx=\frac2{c_p}(1-e^{-t}),\]
it is straightforward to check the validity of \eqref{eq:221} (for example using L'H\^opital's rule). By \cite[Equation (4.5)]{FA07} and Condition (O) with a fixed $r>\bar{\mu}$, we have
\[\log \E e^{-\lambda\xi^2}\geq -1+\log\rp(|\xi|\leq \lambda^{-1/2})\geq -1-C_1(r)\lambda^\frac{1}{2r}=\lambda^{\frac1{2\mu}} \big(-\lambda^{-\frac1{2\mu}}-C_1(r)\lambda^{\frac1{2r}-\frac1{2\mu}}\big), \quad \forall \lambda>0.\] 

Combining the above considerations, we get that there exist finite constants, also denoted by $R_1, R_2$ and which are  independent of $\alpha$, such that for $\mu=1/2+\alpha$
\begin{equation}\label{eq:Laplacebd} -\lambda^{\frac1{2\mu}} R_1\leq \log \E e^{-\lambda \sum_{\ell=1}^\infty \ell^{-2\mu}\xi_\ell^2}\leq -\lambda^{\frac1{2\mu}} R_2, \quad\forall \lambda\ge 1, \quad \forall \alpha\in[\ubar{\alpha},\bar{\alpha}].\end{equation}

To complete the proof, we exploit the relationship between the Laplace transform $ \log \E e^{-\lambda V}$ and the small ball probabilities $\rp(V\leq \gep)$, as presented in the proof of \cite[Lemma 1.5]{FA06thesis}. In particular, by \cite[Equations (1.3) and (1.4)]{FA06thesis}, for all $\gep, \lambda>0$ we have
\begin{equation}\label{eq:13thesis}
\log\rp(V\le \gep)\leq \lambda\gep+\log \E e^{-\lambda V}
\end{equation}
and 
\begin{equation}\label{eq:14thesis}
\E e^{-\lambda V}\leq e^{-\lambda\gep}+(1-e^{-\lambda\gep})\rp(V\leq \gep).
\end{equation}
We will  use these bounds with
$V\coloneqq\sum_{\ell=1}^\infty \ell^{-1-2\alpha}\xi_\ell^2$.  

Let $\gamma\coloneqq 1 / 2\mu=1 / (1+2\alpha)$. On the one hand, by \eqref{eq:Laplacebd}, we have $\log \E e^{-\lambda V}\geq -R_1\lambda^\gamma$ for all $\lambda\ge1$, hence using \eqref{eq:14thesis} with $\lambda=K_1\gep^{-\frac1{1-\gamma}}$ for $K_1\ge1$ to be chosen below, we get
\[\rp(V\le \gep)\ge \frac{e^{-R_1K_1^\gamma \gep^{-\frac{\gamma}{1-\gamma}}}-e^{-K_1\gep^{-\frac{\gamma}{1-\gamma}}}}{1-e^{-K_1\gep^{-\frac\gamma{1-\gamma}}}}= \frac{e^{-R_1K_1^\gamma \gep^{-\frac{\gamma}{1-\gamma}}}\Big(1- e^{(R_1K_1^\gamma-K_1)\gep^{-\frac\gamma{1-\gamma}}}\Big)}{1-e^{-K_1\gep^{-\frac\gamma{1-\gamma}}}}.\]
Choosing $K_1>1\vee R_1^\frac{1+2\ubar{\alpha}}{2\ubar{\alpha}}$, we have that $R_1K_1^\gamma-K_1<0$  for all $\alpha\in[\ubar{\alpha},\bar{\alpha}]$, so that  for $\gep\in(0,1]$ the parenthesis term in the right hand side above is lower bounded by a constant $0<C''<1$. We hence obtain 
 \[\rp(V\le \gep)\geq C''e^{-K_1\gep^{-\frac\gamma{1-\gamma}}}=C''e^{-K_1\gep^{-\frac1{2\alpha}}}\geq e^{-K_1'\gep^{-\frac1{2\alpha}}}, \quad \forall \gep\in(0,1], \quad \forall\alpha\in[\ubar{\alpha},\bar{\alpha}],
\] where $K_1'>0$ is independent of $\alpha$
and replacing $\gep$ by $\gep^2$, we have 
\begin{equation}\label{eq:llbd}\log\rp\big(V\le \gep^2\big)\geq -K_1'\gep^{-\frac1\alpha}, \quad \forall\gep\in(0,1],\quad \forall \alpha\in[\ubar{\alpha},\bar{\alpha}].\end{equation}

On the other hand, by \eqref{eq:Laplacebd}, we have $\log \E e^{-\lambda V}\leq -R_2 \lambda^\gamma$ for all $\lambda\ge 1$, thus using \eqref{eq:13thesis} with $\lambda=K_2 \gep^{-\frac1{1-\gamma}}$ for $0<K_2\leq1$ to be chosen below and for $0<\gep\leq K_2^{\frac{2\bar{\alpha}}{1+2\bar{\alpha}}}$ (so that $\lambda\ge1$), we get 
\[\log\rp(V\leq \gep)\leq -\gep^{-\frac{\gamma}{1-\gamma}}(R_2K_2^\gamma- K_2).\]
Choosing $K_2\leq 1\wedge R_2^{\frac{1+2\ubar{\alpha}}{2\ubar{\alpha}}}$, we have that $R_2K_2^\gamma-K_2=C'>0$ for all $\alpha\in[\ubar{\alpha},\bar{\alpha}]$ and by replacing $\gep$ by $\gep^2$, we therefore obtain 
\begin{equation}\label{eq:uubd}\log \rp\big(V\leq \gep^2\big)\leq -C'\gep^{-\frac{2\gamma}{1-\gamma}}=-C'\gep^{-\frac1{\alpha}}, \quad \forall 0<\gep\leq K_2^{\frac{\bar{\alpha}}{1+2\bar{\alpha}}},\quad \forall\alpha\in[\ubar{\alpha},\bar{\alpha}],\end{equation}
where $C'>0$ is independent of $\alpha$.

Combined, \eqref{eq:uubd} and \eqref{eq:llbd} verify \eqref{2ndclaim} for $\gep\in \big(0,K_2^{\frac{\bar{\alpha}}{1+2\bar{\alpha}}}\big].$  Let $\gep_0>0$. Due to the positivity and the joint continuity of $\gep^{1/\alpha}\varphi_0(\gep;\alpha)$ for $(\gep, \alpha)\in \big[K_2^{\frac{\bar{\alpha}}{1+2\bar{\alpha}}}, \gep_0 \big]\times[\ubar{\alpha},\bar{\alpha}]$ (which follows from the continuity of $\varphi_0$ separately with respect to $\gep$ and $\alpha$ and its monotonicity properties, see \cite{KD69}), there exist constants $B,D>0$ such that \[B\leq-\gep^{1/\alpha}\log\rp\big(V\leq \gep^2\big)\leq D, \quad\forall \gep\in \big[K_2^{\frac{\bar{\alpha}}{1+2\bar{\alpha}}}, \gep_0\big], \quad\forall[\ubar{\alpha},\bar{\alpha}].\] The validity of \eqref{2ndclaim} is thus verified and the proof is complete.
\end{proof}
\end{document}